\documentclass[a4paper, 10pt, twoside, notitlepage.draft]{amsart}

\usepackage[utf8]{inputenc}
\usepackage{color}
\usepackage{amsmath} 
\usepackage{amssymb} 
\usepackage{amsthm}
\usepackage{geometry}
\usepackage{graphicx}
\graphicspath{{figures/}}
\usepackage{esint}
\usepackage[colorlinks=true,linkcolor=blue]{hyperref}

\theoremstyle{plain}
\newtheorem{thm}{Theorem}
\newtheorem{prop}{Proposition}[section]
\newtheorem{lem}[prop]{Lemma}
\newtheorem{cor}[prop]{Corollary}

\newtheorem{defi}[prop]{Definition}
\newtheorem{rmk}[prop]{Remark}

\newcommand {\R} {\mathbb{R}} \newcommand {\Z} {\mathbb{Z}}
\newcommand {\T} {\mathbb{T}} \newcommand {\N} {\mathbb{N}}
\newcommand {\C} {\mathbb{C}} 
\newcommand {\p} {\partial}

\newcommand {\supp} {\text{supp}}
\newcommand {\diam} {\text{diam}}
\newcommand {\rank} {\text{rank}}

\DeclareMathOperator {\dist} {dist}
\DeclareMathOperator {\diag} {diag}

\DeclareMathOperator{\inte} {int}

\DeclareMathOperator{\Per} {Per}
\DeclareMathOperator{\F} {\mathcal{F}}

\title[Energy Scaling for Higher Order Laminates]{On the Energy Scaling Behaviour of Singular Perturbation Models with Prescribed Dirichlet Data Involving Higher Order Laminates}

\author{Angkana R\"uland}
\address{Institut f\"ur Angewandte Mathematik, Ruprecht-Karls-Universit\"at Heidelberg, Im Neuenheimer Feld 205, 69120 Heidelberg, Germany}
\email{Angkana.Rueland@uni-heidelberg.de}

\author{Antonio Tribuzio}
\address{Institut f\"ur Angewandte Mathematik, Ruprecht-Karls-Universit\"at Heidelberg, Im Neuenheimer Feld 205, 69120 Heidelberg, Germany}
\email{Antonio.Tribuzio@uni-heidelberg.de}

\date{\today}

\begin{document}

\begin{abstract}
Motivated by complex microstructures in the modelling of shape-memory alloys and by rigidity and flexibility considerations for the associated differential inclusions, in this article we study the energy scaling behaviour of a simplified $m$-well problem without gauge invariances. Considering wells for which the lamination convex hull consists of one-dimensional line segments of increasing order of lamination, we prove that for prescribed Dirichlet data the energy scaling is determined by the \emph{order of lamination of the Dirichlet data}. This follows by deducing (essentially) matching upper and lower scaling bounds. For the \emph{upper} bound we argue by providing iterated branching constructions, and complement this with ansatz-free \emph{lower} bounds. These are deduced by a careful analysis of the Fourier multipliers of the associated energies and iterated ``bootstrap arguments'' based on the ideas from \cite{RT21}. Relying on these observations, we study models involving laminates of arbitrary order.
\end{abstract}

\maketitle
\tableofcontents
\addtocontents{toc}{\setcounter{tocdepth}{1}}

\section{Introduction}

Motivated by the study of multi-well energies in the vector-valued calculus of variations \cite{D,C95,C99,CM99,D04,P} and, in particular, the modelling of shape-memory alloys within the phenomenological theory of martensite \cite{BJ92,B,M1}, in this article we study the scaling behaviour of a simplified $m$-well problem without gauge invariance (i.e.~without frame-indifference and hence without $SO(n)$ or $\text{Skew}(n)$ symmetries for prescribed displacement boundary conditions. This leads to highly non-convex singularly perturbed energies consisting of an ``elastic'' and a ``surface energy'' contribution: For $\epsilon>0$ and $p\in [1,\infty)$ we consider energies of the following type
\begin{align}
\label{eq:elast_gen}
E_\epsilon^{(p)}(u):= \int\limits_{\Omega} \dist^p(\nabla u, K)dx + \epsilon E_{surf}(u).
\end{align}
For the corresponding models from elasticity, $\Omega\subset \R^{n}$ denotes the reference configuration, the set $K \subset \R^{n\times n}$ consists of the possible energy wells and $E_{surf}(u)$ models the surface energy which penalizes high oscillations between the wells. In what follows, we will restrict our attention to (two types of) sharp interface models, e.g. $E_{surf}(u):= \|D^2 u\|_{TV}$, where $\| \cdot \|_{TV}$ denotes the total variation norm, see \cite{CC15,CO,CO1,KK,KKO13,Rue16b} and \cite[Chapter 12]{B} for similar models. The contributions $\int\limits_{\Omega} \dist^p(\nabla u, K)dx $ models the elastic energy and the exponent $p$ is often taken to be equal to two.

Motivated by the hierarchical structures predicted by the phenomenological theory of martensite \cite{BJ92,B04,CS13,CS15,Bhat2,Bhat}, the theoretical analysis of convex hulls and (higher order) laminates \cite{KMS03} and the relevance of scaling in possibly selecting particular classes of (wild or regular) microstructure \cite{RTZ19, RZZ18}, we study the \emph{influence of the order of lamination} of the Dirichlet boundary conditions on the energy scaling behaviour (see Definition \ref{defi:laminates} in Section \ref{sec:not_pre} for the definition of laminates of finite order). In order to derive quantitative upper and (essentially) matching lower bounds, we consider settings which, on the one hand, are flexible enough to allow for higher order laminates, but which, on the other hand, are still rather rigid in the sense that we restrict our attention to energy wells in which only a single rank-one connection is present and where the lamination convex hull is a one-dimensional object. This leads to a special class of finite order laminates which we will here refer to as ``staircase type'' (see \cite{CFM05} for related but not equal staircase laminates).

As a main objective of this article, we study how the hierarchical structure of Dirichlet boundary data in terms of its order of lamination is reflected in energy scaling results. In this context, we
\begin{itemize}
\item revisit the two-well problem. Here we study the detailed scaling of the \emph{two-well problem} in dependence of the choice $p\in [1,\infty)$. This is motivated by the observation that for $p=1$ twinning structures provide the same scaling as branching structures (see \cite{CM99} and also \cite{Lorent06}). In accordance with the results from \cite{KM1,KM2}, for $p>1$ we show that branching structures provide a better scaling behaviour than simple laminates. Our proof follows the strategy from \cite{CC15}. 
We highlight that such behaviour was hinted at in \cite{CM99} and upper bounds for such type of constructions with an arbitrarily large number of wells consisting of vectors were studied in a finite element setting in a non-published manuscript by Chipot and Müller \cite{CM97}. 
We use the upper bound constructions for $p=2$ as building blocks in the higher order branching constructions in the later sections.
\item consider a particular \emph{three-well model} which had been introduced in \cite{Lorent06} for which double laminates may be enforced by corresponding boundary conditions. While \cite{Lorent06} investigated this problem for $p=1$ in which case both twins and branching twins are expected to provide the same scaling behaviour, we explore the setting for $p=2$ which is expected to enforce (doubly) branched structures. In order to deduce the (essentially) optimal lower bounds, we rely on a combination of the ideas from \cite{RT21} and \cite{KW16} (which in turn build on the earlier works \cite{CO,CO1,KKO13}). We prove the typical $L^2$-based branched second-order lamination scaling of size $\epsilon^{1/2}$. Both the upper and lower bounds show a hierarchical structure which is not present in the two-well setting and is a consequence of the full matrix-valued setting.
\item explore the scaling behaviour for a \emph{three-dimensional four-well} model which allows for laminates of order three. Using similar ideas as for the three-well setting, we prove that the order of branching of the boundary data determines the energy scaling behaviour. We show that the top-order branched laminates scale of the order $\epsilon^{2/5}$ and provide (essentially sharp) scaling results also for the lower-order laminates.
\item provide an example of a ``staircase laminate" in $n$ dimensions with \emph{$n+1$ energy wells} and a ``one-dimensional laminar convex hull" with a scaling of the form $\epsilon^{2/(n+2)}$ for the top order laminates. We provide the full hierarchy of scalings depending on the lamination order of the Dirichlet data.
\end{itemize}

We discuss the corresponding models and results in detail in the following sections.

\begin{figure}[tb]
\begin{center}
\includegraphics{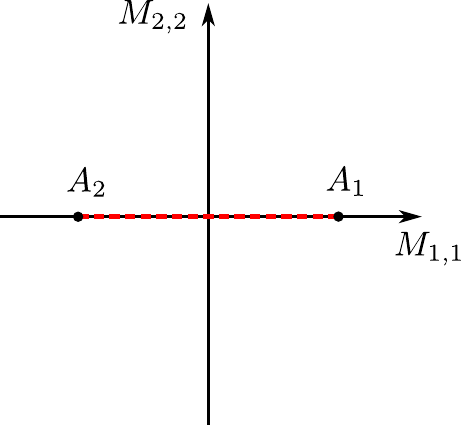}
\end{center}
\caption{Representation of the set $K_2$ on the plane of diagonal matrices (see Section \ref{sec:Lpintro}).
The red hashed segment depicts the set $K_2^{(lc)}$.
}
\label{fig:K2}
\end{figure}

\subsection{First-order laminates and the two-well problem: $L^p$ elastic energies and minimizers}
\label{sec:Lpintro}

We first investigate the problem of Dirichlet boundary data corresponding to first order laminates. This is a well-studied setting \cite{CO,CO1,CC15,C1,CT05,CZ16,CDZ17,SPDBJ20} which builds on the foundational results from \cite{KM1,KM2}. We emphasize that in this setting even finer properties of minimizers such as their asymptotic self-similarity have been studied \cite{C1}. 
Relying on the techniques from \cite{CC15}, we here explore the scaling behaviour of such a setting in which we allow for different choices of $p$ in \eqref{eq:elast_gen}. More precisely, in \eqref{eq:elast_gen} we select $p\in [1,\infty)$ arbitrary but fixed, define the surface energy to be given by $E_{surf}(u):=\|D^2 u\|_{TV}$ and choose the set $K$ to be equal to
\begin{align*}
K_2 = \left\{ \begin{pmatrix} 1 & 0 \\ 0 & 0 \end{pmatrix},  \begin{pmatrix} -1 & 0 \\ 0 & 0 \end{pmatrix}\right\},
\end{align*}
see Figure \ref{fig:K2}.
We consider the first-order laminate $F\in K^{(lc)}_2 \setminus K_2 =\left\{ \begin{pmatrix} \mu & 0 \\ 0 & 0 \end{pmatrix}: \ |\mu|<1 \right\}$ (see Definition \ref{defi:laminates}) as affine boundary condition.
We remark that the choice of $K_2$ is generic, if one requires it to contain a rank-one connection and that this form and the two-dimensional setting may thus be assumed without loss of generality. Further, while -- in view of Taylor approximations of general stored energy functions -- the choice $p=2$ often is the most natural choice, also other values of $p\in [1,\infty)$ have been considered in the literature and arise in applications, e.g. in plasticity. In this context, we prove the following matching lower and upper scaling bounds:

\begin{thm}
\label{thm:p_dependence_two_wells}
Let $\Omega = [0,1]^2$, let $E_{\epsilon}^{(p)}$ be as in \eqref{eq:elast_gen} with $K=K_2$ and $E_{surf}(u) = \|D^2 u\|_{TV}$, and let $F \in K_2^{(lc)}\setminus K_2$.
Let 
\begin{align}\label{eq:en_p_min}
E_{\epsilon}^{(p)}(F):=\inf\limits_{u \in \mathcal{A}_F^p} E_{\epsilon}^{(p)}(u),
\end{align}

where
\begin{align*}
\mathcal{A}_F^p:=\{u\in W^{1,p}_{loc}(\R^2, \R^2): \ u(x) = Fx + b \mbox{ in } \R^2 \setminus \Omega \mbox{ for some } b \in \R^2\}.
\end{align*}
Then, there exist constants $0<c\leq C$ such that for every $\epsilon\in(0,1)$
\begin{align*}
c\epsilon^{\frac{p}{p+1}}\leq E_{\epsilon}^{(p)}(F) \leq C \epsilon^{\frac{p}{p+1}},
\end{align*}
where both $c$ and $C$ depend on $p$, $F$ and $K_2$.
\end{thm}

We emphasize that these types of upper scaling bounds can already be found in the unpublished manuscript \cite{CM97} which was made available to us by the authors. In deducing our scaling bounds, we follow the very robust strategy from \cite{CC15}.

We remark that the restriction to $\Omega = [0,1]^2$ is just for convenience and that any (non-degenerate) sufficiently regular domain yields an analogous result.
Let us further point out several aspects concerning Theorem \ref{thm:p_dependence_two_wells}: Firstly, for $p=2$ we recover the well-known $\epsilon^\frac{2}{3}$ scaling result (see \cite{KM1}). Secondly, for $p=1$ we obtain the $\epsilon^\frac{1}{2}$ behaviour from the discrete model from \cite{CM99} (see \cite{L01} for work relating discrete and continuum scalings). While for $p\in (1,\infty)$ the upper bound constructions are obtained through branched twins and the lower bound constructions distinguish between branched twins and non-branched twins and favour branching (see \cite{KM1,KM2} for a first observation on this), this is not the case for $p=1$. For $p=1$ the scaling behaviour of branched and non-branched twins coincides. 

Also, upper bound constructions for higher order laminates confirm this (however in this case only for second order laminates matching lower bounds are known, see the discussion in Sections \ref{sec:three_wellsinto}, Appendix \ref{sec:second-order} and Remark \ref{rmk:Lorent} below).
\begin{figure}[t]
\begin{center}
\includegraphics{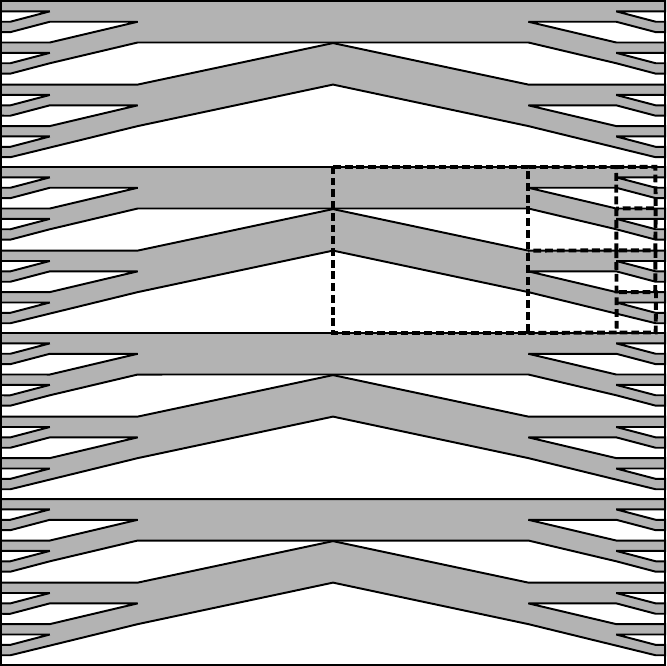}
\quad
\includegraphics{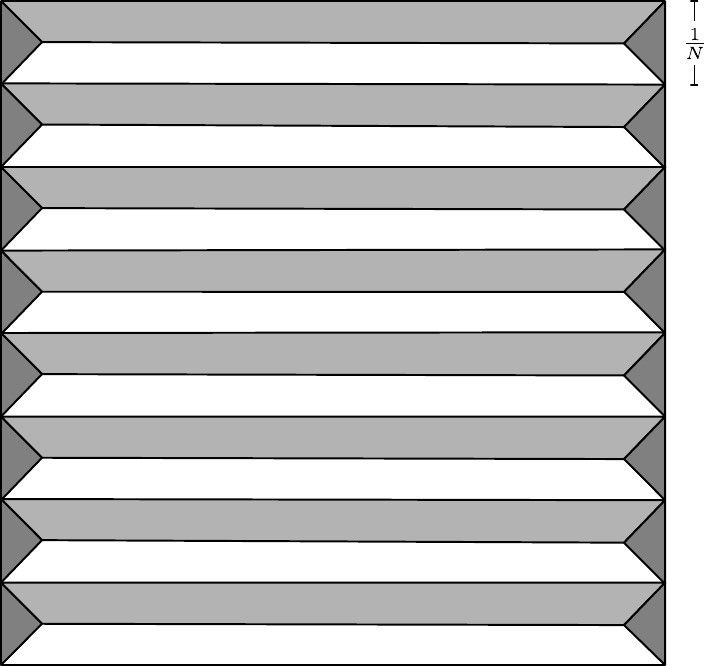}
\end{center}
\caption{An illustration of the microstructures used in the heuristic explanation on the differences between $p=1$ and $p\in (1,\infty)$. On the left a branching construction.
The hashed lines mark the building blocks of different refinement generations.
On the right a simple laminate.
The darker regions represent the boundary layer.
}
\label{fig:intro1}
\end{figure}
\\
Let us give a heuristic back-of-the-envelope argument for this difference between $p \in (1,\infty)$ and $p=1$ by considering upper-bound constructions for $p=2$ and $p=1$. Arguing, for instance, as in the computations in \cite[Lemma 2.4]{CC15}, we obtain the following energy contributions for upper-bound constructions.

\smallskip
\centerline{\emph{$L^2$-based energies:}}
\smallskip

\noindent In the case with branching, one obtains (see Lemma \ref{lem:loc-vert})
\begin{itemize}
\item for the elastic energy per building block a contribution of $(\frac{h}{\ell})^2h\ell\sim\frac{h^3}{\ell}$,
\item for the surface energy $\epsilon \ell$,
\end{itemize}
where $h$ and $\ell$ denote respectively the vertical and horizontal length of the block, being $h \sim \frac{1}{2^jN}$ and $\ell \sim \theta^j$ where $j$ denotes the \emph{generation} of the block of the self-similar construction and $N$ its frequency, see Figure \ref{fig:intro1} (left panel).
Here $\theta$ is a geometric constant smaller then $1$.
There are no essential contributions that come from the interpolation layer which are different from the ones above.
So, after summation and non-dimensionalization, the energy is given by 
\begin{align*}
E_\epsilon^{(2)}(u) \sim N \Big(\frac{1}{N^3} + \epsilon\Big),
\end{align*}
where $u$ is the construction described above
(Lemma \ref{prop:vert-bra}).
This implies the choice $N \sim \epsilon^{-\frac{1}{3}}$ when looking for optimal constructions (in terms of scaling) which leads to an energy scaling of the order $\epsilon^{\frac{2}{3}}$.\\
Compared to this, the scaling without branching would correspond to $\epsilon^{\frac{1}{2}}$ 
which is strictly larger for $\epsilon$ small: Indeed,
\begin{itemize}
\item the elastic energy contribution would consist of a boundary layer that scales as $N^{-1}$,
\item the surface energy originating from fine laminates of size $1\times N^{-1}$ would be given by $E_{surf} \sim N$,
\end{itemize}
(Figure \ref{fig:intro1}, right panel).
Thus, optimization of $N$ in terms of $\epsilon$ would result in an $\epsilon^{\frac{1}{2}}$ scaling behaviour.

\smallskip
\centerline{\emph{$L^1$-based energies:}}
\smallskip

\noindent In this case one has (with branching):
\begin{itemize}
\item elastic energy per building block $\left( \frac{h}{\ell} \right) h \ell \sim h^2$,
\item surface energy $\epsilon \ell$,
\end{itemize}
per building block.
The total energy then becomes after summation and non-dimensionalization
\begin{align*}
E_{\epsilon}^{(1)}(u)\sim N \Big(\frac{1}{N^2} + \epsilon\Big).
\end{align*}
Hence, on the one hand, the choice $N \sim \epsilon^{-\frac{1}{2}}$ is optimal for such constructions, leading to a scaling of order $\epsilon^\frac{1}{2}$ for branching constructions.
On the other hand, the scaling of twins without branching does not change, since the twinning heuristics from above remain true for $p=1$, too. As a consequence, for $p=1$, the scaling of branched and twinned construction is the same.

We remark that in addition to the full scaling result from Theorem \ref{thm:p_dependence_two_wells}, we also provide upper bound constructions for the setting of Dirichlet boundary data $F$ of higher-order lamination and with general $p\in [1,\infty)$ (see Appendix \ref{sec:second-order}).

\begin{figure}[t]
\begin{center}
\includegraphics{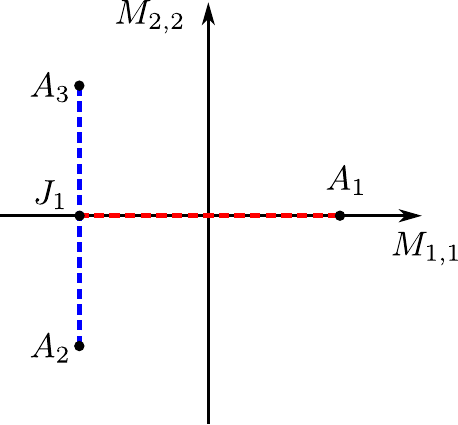}
\end{center}
\caption{The set of matrices $K_3$ from Section \ref{sec:three_wellsinto}.
In blue $K^{(1)}_3$, in red $K^{(lc)}_3\setminus K^{(1)}_3 = K^{(2)}_3\setminus K_3^{(1)}$.
$J_1$ is rank-$1$-connected to all the matrices of $K_3$.}
\label{fig:K3}
\end{figure}

\subsection{Three wells and second-order laminates}
\label{sec:three_wellsinto}

In this section we consider a problem with three wells giving rise to laminations of order up to two. Contrary to the two-well setting, this is a consequence of the matrix-valued differential inclusion which allows for hierarchies of laminations.
For simplicity, we focus on the following explicit model setting in which we make the following choice for the set of wells:
\begin{align*}
K_3:= \left\{A_1, A_2, A_3 \right\}, \mbox{ where }
A_1=\begin{pmatrix}1&0\\0&0\end{pmatrix},\,
A_2=\begin{pmatrix}-1&0\\0&-1\end{pmatrix},\,
A_3=\begin{pmatrix}-1&0\\0&1\end{pmatrix}.
\end{align*}
We however emphasize that the results from below could also be formulated more generally. We observe that $A_2, A_3$ are rank-one connected, that $A_1$ is neither rank-one connected to $A_2$ nor to $A_3$, that the laminates of first order are given by 
\begin{align*}
K_3^{(1)}\setminus K_3=\left\{
\begin{pmatrix} -1 & 0 \\ 0 & \mu \end{pmatrix}: \ |\mu| < 1 \right\},
\end{align*}
those of the second order by
\begin{align*}
K_3^{(2)}\setminus K_3^{(1)}=\left\{ 
\begin{pmatrix} \nu & 0 \\ 0 & 0 \end{pmatrix}: \ |\nu|<1
\right\},
\end{align*}
and that $K_3^{(lc)}= K_3^{(2)}$ (see Definition \ref{defi:laminates} for the definition of the order of lamination).
All of these sets of matrices are depicted in Figure \ref{fig:K3}.
In the sequel, we prove that for the problem with prescribed Dirichlet data, the order of lamination determines the scaling of the singularly perturbed elastic energy. To this end, we consider an energy of a similar type as in Section \ref{sec:Lpintro} with only slight differences which are mainly due to technical considerations. More precisely, in the following discussion, analogously to \cite{CO1,CO,RT21}, we consider
\begin{align}
\label{eq:energyK3}
E_{\epsilon,3}(u,\chi):= E_{el}(u,\chi) + \epsilon E_{surf}(\chi):= \int\limits_{\Omega}|\nabla u - \chi|^2 dx + \epsilon \sum\limits_{j=1}^{3} \|D\chi_j \|_{TV},
\end{align}
where on the left hand side of \eqref{eq:energyK3} we have
\begin{align*}
\chi := \diag(\chi_1-\chi_2-\chi_3, -\chi_2 + \chi_3)=\sum_{j=1}^3 \chi_j A_j: (0,1)^2 \rightarrow \R^{2\times 2},
\end{align*}
and, for $j\in \{1,2,3\}$, $\chi_j \in BV((0,1)^2;\{0,1\})$ such that
\begin{align*}
\sum\limits_{j=1}^{3}\chi_j = 1.
\end{align*}
The functions $\chi_j$, $j\in\{1,2,3\}$, play the role of phase indicators, showing to which of the energy wells the deformation is closest.
It is expected that the energies in \eqref{eq:elast_gen} with $p=2$ and $K=K_3$ behave analogously. Due to the very helpful Fourier characterization (see Section \ref{sec:prelim_lower}), we here focus on the setting outlined in \eqref{eq:energyK3}.

\begin{thm}[Scaling of the three-well problem]
\label{thm:K3}
Let $\Omega  = [0,1]^2$, let $E_{\epsilon,3}(\cdot,\cdot)$ be as \eqref{eq:energyK3} and assume that $\chi$ is as above. For $F\in K^{(lc)}_3\setminus K_3$ let 
\begin{align*}
\mathcal{A}_F:=\{u\in W^{1,2}_{loc}(\R^2;\R^{2}): \ u(x) = Fx+b \mbox{ in } \R^2 \setminus \Omega \mbox{ for some $b\in \R^2$}\}.
\end{align*}
Set
\begin{align}
\label{eq:energyK3_min}
E_{\epsilon,3}(F):= \inf\limits_{\chi}\inf\limits_{u \in \mathcal{A}_F} E_{\epsilon,3}(u,\chi).
\end{align}
\begin{itemize}
\item[(i)] If $F \in K^{(1)}_{3}\setminus K_3$, then there exist constants $0<c\le C$ such that for every $\epsilon\in(0,1)$
\begin{align*}
c\epsilon^\frac{2}{3} \leq E_{\epsilon,3}(F) \leq C \epsilon^\frac{2}{3}.
\end{align*}
\item[(ii)] If $F\in K^{(2)}_{3}\setminus K^{(1)}_3$, then there exist constants $C\geq c>0$, $\epsilon_0 = \epsilon_0(n,F)>0$ such that for every $\epsilon\in(0,\epsilon_0)$
\begin{align*}
c\epsilon^{\frac{1}{2}} \leq E_{\epsilon,3}(F) \leq C \epsilon^\frac{1}{2}.
\end{align*}
\end{itemize}
All constants $c,C$ depend on $F$ and $K_3$.
\end{thm}

This result provides an $L^2$ version of the scaling behaviour captured by Lorent in \cite{L01} where an $L^1$ based elastic energy was considered. The fact that the scaling from Theorem \ref{thm:K3} and the one in \cite{L01} differ (in \cite{L01}, in our notation, a scaling of the order $\epsilon^\frac{1}{3}$ for boundary conditions in the second order lamination convex hull is obtained) is a consequence of the different choice of the $p$-growth condition of the elastic energy. 
Indeed, the result from \cite{L01} corresponds to the choice of $p = 1$ while our result holds for $p=2$. As in Section \ref{sec:Lpintro} this has consequences on the microstructure: While for the case $p=1$ both branched and twinned microstructures lead to the same scaling (see the Appendix \ref{sec:second-order} for general upper bound constructions for second-order laminates with arbitrary $p\in [1,\infty)$), for $p=2$ branched microstructures are favoured for small $\epsilon$. This difference is also strongly reflected in our proofs: While \cite{L01} uses a careful counting argument keeping track of the double laminates, we rely on a Fourier space decomposition in the spirit of Hashin-Strikman estimates \cite{HS63} (see also \cite{CO,CO1,KKO13,KW14,KW16}).

By exploring the model setting from Theorem \ref{thm:K3}, we illustrate that the energy scaling of the simplified three-well problem from \eqref{eq:energyK3} without gauges is determined by the order of lamination of the Dirichlet data. This distinguishes the Dirichlet problem from the setting of periodic boundary data and also the nucleation setting (see Section \ref{subsec:periodic3grad} for a detailed discussion on the periodic setting for the three-well case). As is well-known for simple laminates \cite{KM1,KM2,CO,CO1,CC15,C} boundary data in the first-order lamination convex hull without compensation effects lead to the typical $\epsilon^\frac{2}{3}$ scaling (see the references in Section \ref{sec:lit} below for other, related settings and models from the calculus of variations). Similarly as in \cite{KW16} in which a more complicated model for compliance minimization was considered, in the case of data which are in the second lamination convex hull, this behaviour changes and (essentially) an $\epsilon^\frac{1}{2}$ scaling is obtained.
We remark that in the formulation of the theorem, for technical reason (since we are working with periodic extensions), we have restricted our attention to the square as the underlying domain. It is expected that for other domains which are not extremely elongated in one direction one would obtain the same scaling behaviour but would need to overcome technical challenges in working with zero extensions into $\R^n$ instead of periodic extensions.

\begin{figure}[t]
\begin{center}
\includegraphics{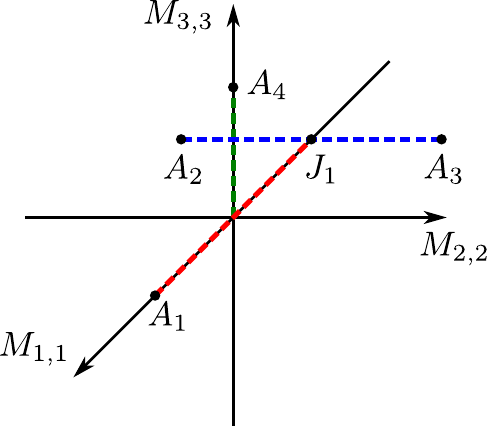}
\end{center}
\caption{The set of matrices $K_4$ in the space of diagonal matrices (see Section \ref{sec:fourthintro}).
In blue $K^{(1)}_4$, in red $K^{(2)}_4\setminus K^{(1)}_4$, in green $K^{(lc)}_4\setminus K_4^{(2)} = K^{(3)}_4\setminus K^{(2)}_4$.}
\label{fig:K4}
\end{figure}

\subsection{Four wells in three dimensions and third-order laminates}
\label{sec:fourthintro}

Building on the observations in the previous section and seeking to study the scaling behaviour of higher-order laminates, we next consider a four-well problem in three dimensions in which laminates up to order three are present. To this end, we consider the set
\begin{align}
\label{eq:K4}
\begin{split}
&K_4:= \{A_1,\dots,A_4\} \subset \R^{3\times 3},\\
& \mbox{with }
A_1=\begin{pmatrix}1&0&0\\0&0&0\\0&0&0\end{pmatrix},\,
A_2=\begin{pmatrix}-1&0&0\\0&-1&0\\0&0&0\end{pmatrix},\,
A_3=\begin{pmatrix}-1&0&0\\0&1&0\\0&0&0\end{pmatrix},\,
A_4=\begin{pmatrix}0&0&0\\0&0&0\\0&0&1\end{pmatrix}.
\end{split}
\end{align}
Again, $A_2$ and $A_3$ are rank-one connected and the lamination convex hulls are given by 
\begin{align*}
&K^{(lc)}_4 = K_4\cup (K^{(1)}_4\setminus K_4) \cup (K^{(2)}_4\setminus K^{(1)}_4)\cup (K^{(3)}_4\setminus K^{(2)}_4); \\
&K^{(1)}_4\setminus K_4:= \left\{
\begin{pmatrix}
-1 & 0 & 0 \\ 0 & \mu & 0\\ 0 & 0& 0
\end{pmatrix}: \ |\mu|<1
\right\},\\
&K^{(2)}_4\setminus K^{(1)}_4:=\left\{
\begin{pmatrix}
\nu & 0 & 0 \\ 0 & 0 & 0\\ 0 & 0& 0
\end{pmatrix}: \ |\nu|<1
\right\},\\
&K^{(3)}_4\setminus K^{(2)}_4:=\left\{
\begin{pmatrix}
0 & 0 & 0 \\ 0 & 0 & 0\\ 0 & 0& \gamma
\end{pmatrix}: \ 0<\gamma<1
\right\},
\end{align*}
with $K_4^{(lc)}=K_4^{(3)}$, see Figure \ref{fig:K4}.
We remark that, while it would be possible to also consider sets $K \subset \R^{2\times 2}$ with three orders of lamination, the set $K_4$ from above has a particularly simple structure, in which $\inte K^{(lc)}_4 = \emptyset$ and $K^{(lc)}_4$ only consists of line segments. This allows for a direct transfer of the ideas from Section \ref{sec:three_wellsinto} to this higher-order lamination setting. In the next section, we generalize this even further, showing that \emph{any} scaling order of the type $\epsilon^{\frac{2}{m+2}}$, $m\in \N$, can be obtained for matrices of a similar structure as the set $K_4$.

As in the previous section, we consider a singularly perturbed variational problem and define
\begin{align}
\label{eq:energyK4}
E_{\epsilon,4}(u,\chi):= E_{el}(u,\chi) + \epsilon E_{surf}(\chi):= \int\limits_{\Omega}|\nabla u - \chi|^2 dx + \epsilon \sum\limits_{j=1}^{4} \|D\chi_j \|_{TV},
\end{align}
where the function $\chi$ on the left hand side of \eqref{eq:energyK4} is defined as
\begin{align*}
\chi := \diag(\chi_1-\chi_2-\chi_3, -\chi_2 + \chi_3,\chi_4) = \sum\limits_{j=1}^{4} \chi_j A_j: (0,1)^3 \rightarrow \R^{3\times 3},
\end{align*}
and, for $j\in \{1,2,3,4\}$, $\chi_j \in BV((0,1)^3;\{0,1\})$ such that
\begin{align*}
\sum\limits_{j=1}^{4}\chi_j = 1.
\end{align*}

With this notation fixed, we prove that again the order of lamination of the Dirichlet data determines the energy scaling behaviour for this four-well problem:

\begin{thm}[Scaling of the four-well problem]
\label{thm:K4}
Let $\Omega  = [0,1]^3$, let $E_{\epsilon,4}(\cdot,\cdot)$ be as in \eqref{eq:energyK4} and assume that $\chi$ is as above. For $F\in K^{(lc)}_4\setminus K_4$ let 
\begin{align*}
\mathcal{A}_F:=\{u\in W^{1,2}_{loc}(\R^3;\R^3): \ u(x) = Fx+b \mbox{ in } \R^3 \setminus \Omega \mbox{ for some $b\in \R^3$}\}.
\end{align*}
Set
\begin{align}
\label{eq:energyK4_min}
E_{\epsilon,4}(F):= \inf\limits_{\chi}\inf\limits_{u \in \mathcal{A}_F} E_{\epsilon,4}(u,\chi).
\end{align}
\begin{itemize}
\item[(i)] If $F\in K^{(1)}_4\setminus K_4$, then there exist two constants $0<c\le C$ such that for every $\epsilon\in(0,1)$
$$
c\epsilon^\frac{2}{3}\le E_{\epsilon,4}(F)\le C\epsilon^\frac{2}{3}.
$$
\item[(ii)] If for $m=2,3$ it holds that $F\in K^{(m)}_4\setminus K^{(m-1)}_4$, then there exist constants $C\geq c>0$, $\epsilon_0>0$ such that for every $\epsilon\in(0, \epsilon_0)$
\begin{align*}
c \epsilon^{\frac{2}{m+2}} \leq E_{\epsilon,4}(F) \leq C \epsilon^{\frac{2}{m+2}}.
\end{align*}
\end{itemize}
All constants $c,C$ depend on $F$ and $K_4$.
\end{thm}

As in the previous section, we here consider a set in matrix space, $K_4 \subset \R^{3\times 3}$, whose lamination convex hull is iteratively built up from a single laminate by adding further lines in each lamination step. The overall lamination convex hull thus consists of a union of lines. In this situation the above result illustrates that the energy scaling behaviour of the associated energy is determined by the order of the lamination of the Dirichlet boundary datum (up to the technical, arbitrarily small loss in the lower bound). It proves that each additional order of lamination requires a fixed quantitative additional amount of energy.

\subsection{Arbitrary order of lamination and $(n+1)$ wells in $n$ dimensions}
\label{sec:n_wells_intro}
Last but not least, we extend the results from the previous subsections to the setting of a set $K_{n+1}$ consisting of $n+1$ matrices in $n$ dimensions. The structure of the $n+1$ matrices are chosen such that only two matrices in $K_{n+1}$ are rank-one connected and such that the lamination convex hull consists of one-dimensional line segments. In this case, laminates up to order $n$ are possible and, as indicated in the previous sections, we again prove that the scaling of the energy minima for prescribed Dirichlet conditions in $K^{(lc)}$ only depends on the order of lamination of the Dirichlet data.

More precisely, as a generalization of the setting from the previous two sections, given parameters $\nu_j\in(0,1)$ for $j=5,\dots,n+1$, we define $A_1,\dots, A_{n+1} \in \R^{n\times n}$ inductively as follows:
\begin{equation*}\label{eq:n_matrices}
\begin{split}
& A_1 := \begin{pmatrix} 1 & 0 & 0  & 0 &\hdots & 0\\
0 & 0 & 0 & 0& \hdots & 0 \\
0 & 0 & 0& 0 & \hdots & 0 \\
\vdots & \vdots & \vdots & \vdots  & \vdots & \vdots\\
0 & 0 & 0 & 0 &0&0 \end{pmatrix},\ 
A_2 := \begin{pmatrix} -1 & 0 & 0  & 0 & \hdots & 0\\
0 & -1 & 0 & 0& \hdots & 0 \\
0 & 0 & 0& 0 & \hdots & 0 \\
\vdots & \vdots & \vdots & \vdots  & \vdots & \vdots\\
0 & 0 & 0 & 0 &0&0 \end{pmatrix},\\ 
& A_3 := \begin{pmatrix} -1 & 0 & 0  & 0 &\hdots & 0\\
0 & 1 & 0 & 0& \hdots & 0 \\
0 & 0 & 0& 0 & \hdots & 0 \\
\vdots & \vdots & \vdots & \vdots  & \vdots & \vdots\\
0 & 0 & 0 & 0 &0&0 \end{pmatrix},\
 A_4  := \begin{pmatrix} 0 & 0 & 0  & 0 &\hdots & 0\\
0 & 0 & 0 & 0& \hdots & 0 \\
0 & 0 & 1 & 0 & \hdots & 0 \\
\vdots & \vdots & \vdots & \vdots  & \vdots & \vdots\\
0 & 0 & 0 & 0 &0&0 \end{pmatrix},\ 
\end{split}
\end{equation*}
and $A_{j+1}=\diag(0,0,\nu_5,\dots,\nu_j,\nu_{j+1},1,0,\dots,0)\in \diag(n,\R)$ for $j\in \{4,\dots,n\}$. 
With this we define
\begin{align}
\label{eq:energyKn}
K_{n+1} := \{A_1, \dots, A_{n+1}\}.
\end{align}

We observe that in this setting, we have that $K^{(lc)}_{n+1}  = K^{(n)}_{n+1}$ and that the laminates of order $3\le j\le n$ consist of the following line segments
\begin{align*}
K^{(j)}_{n+1} \setminus K^{(j-1)}_{n+1} = \left\{ \diag(0,\dots,0,\nu_5,\dots,\nu_{j+1},\nu,0,\dots,0) \in  \R^{n\times n}:  \ \nu \in (0,1) \right\}.
\end{align*}
In particular, there exist genuine laminates of any order up to (and including) the order $n$.

As above, for $\Omega \subset \R^n$ open, bounded we again consider associated elastic and surface energies:
\begin{align}
\label{eq:En}
E_{\epsilon,n+1}(u,\chi):= E_{el}(u,\chi) + \epsilon E_{surf}(\chi),
\end{align}
where 
\begin{align*}
 E_{el}(u,\chi) = \int\limits_{\Omega}|\nabla u - \chi|^2 dx, \quad E_{surf}(\chi):=\sum\limits_{j=1}^{n+1}\|D\chi_j\|_{TV},
\end{align*}
and where the functions $\chi_{j}\in BV((0,1)^n;\{0,1\})$ for $j\in \{1,\dots,n+1\}$ are such that $\sum\limits_{j=1}^{n+1} \chi_j = 1$. We further define
\begin{align*}
\chi = \sum\limits_{j=1}^{n+1} \chi_j A_j:(0,1)^n\to\R^{n\times n}.
\end{align*}
With this notation, as in the previous set-up, we prove the scaling of these singularly perturbed energies in the limit $\epsilon \rightarrow 0$ and illustrate the resulting hierarchy of scales which are determined by the order of lamination of the Dirichlet data.

\begin{thm}[Scaling of the $(n+1)$-well problem]
\label{thm:Kn}
Let $\Omega  = [0,1]^n$, let $E_{\epsilon,n+1}(\cdot,\cdot)$ be as in \eqref{eq:En} and assume that $\chi$ is as above. For $F\in K^{(lc)}_{n+1}\setminus K_{n+1}$ let 
\begin{align*}
\mathcal{A}_F:=\{u\in W^{1,2}_{loc}(\R^n;\R^{n}): \ u(x) = Fx+b \mbox{ in } \R^n \setminus \Omega \mbox{ for some $b\in \R^n$}\}.
\end{align*}
Set
\begin{align}
\label{eq:energyKn_min}
E_{\epsilon,n+1}(F):= \inf\limits_{\chi}\inf\limits_{u \in \mathcal{A}_F} E_{\epsilon,n+1}(u,\chi).
\end{align}
\begin{itemize}
\item[(i)] If $F\in K^{(1)}_{n+1}\setminus K_{n+1}$, then there exist constants $0<c\le C$ such that for every $\epsilon\in(0,1)$ there holds
$$
c\epsilon^\frac{2}{3}\le E_{\epsilon,n+1}(F)\le C\epsilon^\frac{2}{3}.
$$
\item[(ii)]
If for $m\in \{2,\dots,n\}$, it holds that $F\in K^{(m)}_{n+1}\setminus K^{(m-1)}_{n+1}$, then there exist constants $C\geq c>0$, $\epsilon_0>0$ such that for every $\epsilon\in(0, \epsilon_0)$
\begin{align*}
c \epsilon^{\frac{2}{m+2}} \leq E_{\epsilon,n+1}(F) \leq C \epsilon^{\frac{2}{m+2}}.
\end{align*}
\end{itemize}
All constants $c,C$ depend on $F$ and $K_{n+1}$.
\end{thm}

We emphasize that in the specific setting in which the lamination convex hull consists only of line segments Theorem \ref{thm:Kn} yields the dependence of the minimal energy scaling purely in terms of the boundary data.

\subsection{Structure of the sets $K_m$}
\label{sec:structureKm}
Let us stress that the sets $K_m$ for which we study differential inclusions in this article are constructed such that only a single rank-one connection is present in the set $K_m$ and that the lamination convex hull is obtained iteratively in a hierarchical procedure adding one new one-dimensional line segment in each lamination step. In particular, the overall lamination convex hull consists of a finite union of one-dimensional line segments. With slight abuse of notation compared to \cite{CFM05}, we refer to these boundary conditions and wells $K_m$ as a ``staircase laminates''. It is this structure which allows us to apply the commutator arguments from Section \ref{sec:prelim_lower} in deducing lower bounds for the energy scaling laws. We emphasize that with this observation it is possible to construct many further examples of sets $K$ having a similar property and that thus our sets $K_m$ from above should be regarded as prototypical model cases which can be generalized substantially. We however caution that while this analysis provides a systematic treatment of scaling laws for such ``staircase laminates'' and while this procedure can also be applied to certain sets $K \subset \R^{n\times n}$ whose lamination convex hulls are not one-dimensional, it is expected that with more complicated sets $K$, new, interesting behaviour can be observed.

\subsection{Relation to the literature}
\label{sec:lit}
Let us embed the outlined results into the literature: While not including a gauge invariance, our model is strongly motivated by the investigation of singularly perturbed problems in the phenomenological theory of martensite and, more generally, by vector-valued phase transition problems \cite{B04,D,BJ92,M1,P}. These have in common that they lead to highly non-convex variational problems and are associated with vector-valued convexity notions. The analysis of these models and their (potentially very complex) minimizers by means of singularly perturbed models has a long tradition (see \cite{KM1,KM2} and \cite{M1} and the references therein) and has been studied for a rich set of physical applications. Closest to our setting are the articles on the phenomenological theory of martensite \cite{CO,CO1,CC15,KK,KKO13,KO19,KM1,KM2,R16,C,CZ16,CDZ17,CDMZ20,BG15,BK14,SPDBJ20,Lorent06,L01,TS,TS21}, compliance minimization problems \cite{KW14,KW16}, various scaling results on micromagnetism \cite{CKO99,OV10, KN18}, the more abstract matrix-space analysis in \cite{KMS03,RTZ19} and the discrete model problems from \cite{CK99,C91,C95,CM99,C99}. In particular, as in the latter class of articles, we seek to investigate the role of the order of lamination on the energy scaling and have simplified the models from applications by, for instance, not taking into account gauges. It is our objective and one of the main novelties of this article to contribute to ``closing the gap'' between the celebrated $\epsilon^{\frac{2}{3}}$ branching scaling from \cite{KM1,KM2} and the extremely rigid scaling of the $T_4$ structures from \cite{RT21} and to thus also capture the scaling behaviour and properties of higher order laminar structures.

\subsection{Outline of the article}
The remainder of the article is structured as follows: 
\begin{itemize}
\item In Section \ref{sec:not_pre} we recall our notation, central definitions and some auxiliary results.
\item Building on \cite{CC15}, in Section \ref{sec:Lp} we first study the $p$-dependence of the energy scaling for simple laminate boundary conditions by proving Theorem \ref{thm:p_dependence_two_wells}. 
\item In Section \ref{sec:prelim_lower} we present a number of technical auxiliary results which will play a crucial role in the proof of the lower bounds of Theorems \ref{thm:K3}-\ref{thm:Kn}. In contrast to the results from Section \ref{sec:Lp} these strongly rely on $L^2$-based Fourier techniques.
\item Building on the discussion from Section \ref{sec:prelim_lower}, the lower and upper bounds of Theorems \ref{thm:K3}-\ref{thm:Kn} are then presented in Sections \ref{sec:3well_lower}-\ref{sec:low_n}.
\item Last but not least, in Appendix \ref{sec:second-order} we provide higher-order two-dimensional branching constructions for the general energies in \eqref{eq:elast_gen} and deduce corresponding upper scaling bounds.
\end{itemize}

\subsection*{Acknowledgements}
Both authors would like to thank Prof.~M.~Chipot and Prof.~S.~Müller for sharing their thoughts on the $p$-dependence from Section \ref{sec:Lp} with them and sending us their unpublished, draft manuscript \cite{CM97}, containing upper bound constructions for a general, finite element $m$-well problem with wells consisting of vectors $v_j \in \R^n$ (which correspond to the upper bound scaling results presented in Section \ref{sec:Lpintro}).

This work was funded by the Deutsche Forschungsgemeinschaft (DFG, German Research Foundation) through SPP 2256, project ID 441068247.
A.R.~is a member of the Heidelberg STRUCTURES Excellence Cluster, which is funded by the Deutsche Forschungsgemeinschaft (DFG, German Research Foundation) under Germany’s Excellence Strategy EXC2181/1-390900948.

\section{Notation and Auxiliary Results}

\label{sec:not_pre}

In this section we collect some of the notational conventions and background results which we will use in the following sections.

Throughout the article, when writing $a\sim b$ we mean that $c^{-1} a\le b\le c a$ where $c>0$ is a fixed constant independent of $\epsilon$ (but which may depend on other quantities such as the dimension $n$, the choice of $p \in [1,\infty)$ etc).
Analogously, $a\lesssim b$ and $a\gtrsim b$ stand for $a\le cb$ and $a\ge cb$, respectively.

Given $n\in\N$, we use the notation $\diag(n,\R)$ to denote the $n\times n$ diagonal matrices with real entries.
With $e_j\in\R^n$, $j\in\{1,\dots,n\}$, we denote the $j$-th vector of the canonical basis of $\R^n$.

Given $d,m\in\N$, $\Omega \subset \R^n$ open, bounded and a summable function $f:\Omega\subset\R^n\to\R^{m\times d}$ we denote the \emph{average of $f$ on $\Omega$} by
$$
\langle f\rangle:=\frac{1}{|\Omega|}\int_\Omega f(x)dx.
$$
In the case that this notation is ambiguous, we will write $\langle f\rangle_\Omega$.
Let $\Omega=\prod_{j'=1}^n[a_{j'},b_{j'}]$, given $j\in\{1,\dots,n\}$, let $\Omega':=\prod_{j'\neq j}^n[a_{j'},b_{j'}]$ we will denote by $\langle f\rangle_j:\Omega'\subset\R^{n-1}\to\R^{m\times d}$
the function defined as follows
$$
\langle f\rangle_j:= \frac{1}{b_j-a_j}\int_{a_j}^{b_j} f(x)dx_j
$$
which is well-defined almost everywhere (and summable) in $\Omega'$ thanks to Fubini's Theorem.

Given a set $M\subset \R^n$ and a positive integer $h$, we use the notation
$S_{h} M:= \{\sum\limits_{j=1}^{h} m_j:\ m_j \in M\}$ to denote the $h$-fold Minkowski sum of $M$ with itself.

We collect our notation on the energies defined in the introduction and in the remainder of the article.
Given $p\in[1,\infty)$ and $\Omega \subset \R^n$ open, bounded, we consider
\begin{itemize}
\item the following $L^p$-based elastic and surface energy 

\begin{align*}
E^{(p)}_{\epsilon}(u) = \int\limits_{\Omega}\dist^p(\nabla u, K) dx + \epsilon \|D^2 u\|_{TV}
\end{align*} 
in Section \ref{sec:Lp}, together with their localizations $E_{\epsilon}^{(p)}(u,R)$ for $R\subset \Omega \subset \R^{2}$,
\item the $L^2$-based, phase indicator type energies  

\begin{align*}
E_{\epsilon,n+1}(u,\chi) = E_{el}(u,\chi) + \epsilon E_{surf}(\chi) := \int\limits_{\Omega}|\nabla u - \chi|^2 dx + \epsilon \sum\limits_{j=1}^n \|D \chi_j\|_{TV}
\end{align*}

 with $\chi$ as in the introduction and $n\in \N$. We also use a number of related quantities in which we minimize over $\chi$ or $u$ and emphasize the role of the boundary condition.
\end{itemize}

Further, for the characteristic functions from the introduction, we set
\begin{align}
\label{eq:char_func_norm}
\tilde{\chi}_{j,j}:= \chi_{j,j}- F_{j,j},
\end{align}
and often view these as functions on $\T^n$ (instead of as functions on $\Omega$) by considering the periodic extension.

Moreover, since our lower bound estimates will rely on arguments in Fourier space we introduce our notation of Fourier multipliers. A function $m: \R^n \rightarrow \R $ with $|m(k)| \leq C(1+|k|^2)^s$ for some $s\in \R$ (and in particular for $s=0$) gives rise to a \emph{Fourier multiplier} $m(D)$ defined as 
\begin{align}
\label{eq:mult}
\F(m(D)u)(k):= m(k) \F u(k).
\end{align}
Here given $u\in L^1(\T^n)$
$$
\F u(k):= \int_{\T^n} e^{- 2\pi i k \cdot x} u(x) dx, \quad k\in\Z^n,
$$
denotes the \emph{$k$-th Fourier coefficient} of $u$.
If there is no ambiguity we will also use the notation $\hat u=\F u$.
Also, for clarity of exposition, we will denote with $x$ the space variable and with $k$ the frequency variable.

For further use, we recall a corollary of the Marcinkiewicz multiplier theorem on $\R^n$ (see, for instance, \cite[Corollary 6.2.5]{Grafakos}) which, in combination with the transference principle (see e.g.~\cite[Theorem 4.3.7]{Grafakos}), provides $L^p$-$L^p$ bounds of $C^n$ regular Fourier multipliers provided a suitable decay of their derivatives holds.
\begin{prop}\label{prop:grafakos}
Let $m$ be a bounded $C^\infty$ function.
Assume that for all $h\in\{1,\dots,n\}$, all distinct $j_1,\dots, j_{h}\in\{1,\dots,n\}$, and all $k_j\in\R\setminus\{0\}$ with $j\in\{j_1,\dots,j_h\}$ we have
\begin{equation}\label{eq:decay-multiplier}
|\p_{j_1}\dots\p_{j_{h}} m(k)|\le A|k_{j_1}|^{-1}\dots|k_{j_{h}}|^{-1}
\end{equation}
for some $A>0$.
Then for all $p\in(1,\infty)$, there exists a constant $C_n>0$ depending on the dimension such that for every $u\in L^p(\T^n;\C)$ there holds
$$
\|m(D) u\|_{L^p}\le\|u\|_{L^p} C_n(A+\|m\|_{L^\infty})\max\Big\{p,\frac{1}{p-1}\Big\}^{6n}.
$$
\end{prop}

We end this section by recalling the notion of \emph{lamination-convex hull} (see for instance \cite{D04,M1}).

\begin{defi}[The lamination convex hull]
\label{defi:laminates}
Let $K \subset \R^{n\times m}$ for $m,n \geq 2$. The \emph{lamination convex hull of the set $K$}, $K^{(lc)}$, is defined as
\begin{align*}
K^{(lc)
}:= \bigcup\limits_{j=0}^{\infty} K^{(j)}, \ \mbox{ with }  K^{(0)}:= K,
\end{align*}
and
\begin{align*}
K^{(j+1)}:=K^{(j)}\cup \{\lambda A+ (1-\lambda) B: \ A, B \in K^{(j)}; \ \text{rank}(A-B)=1, \ \lambda \in [0,1]\}.
\end{align*}
The elements in $K^{(j+1)}\setminus K^{(j)}$ are called \emph{laminates of order $j+1$}.
\end{defi}

\section{On the Two-Well Problem with Elastic Energy of $p$-Growth: Proof of Theorem \ref{thm:p_dependence_two_wells}}
\label{sec:Lp}

In this section we consider elastic energies of $p$-growth of the form
$$
E_{el}^{(p)}(u)=\int_{(0,1)^2}\dist^p(\nabla u,\{A,B\})dx,
$$
where $A,B\in \diag(2,\R)$ are rank-1-connected and where $u$ attains affine Dirichlet data in $\R^2 \setminus [0,1]^2$ which are in the lamination convex hull of $\{A,B\}$.
Without loss of generality, by rotation and scaling, we may assume that
\begin{equation}\label{eq:AB}
A=\begin{pmatrix} 1-\lambda & 0 \\ 0 & 0 \end{pmatrix}, \quad B=\begin{pmatrix} -\lambda & 0 \\ 0 & 0 \end{pmatrix}, \quad \text{for some }\lambda\in(0,1),
\end{equation}
with $u$ attaining zero boundary conditions.
We will consider perturbations of these elastic energies by a higher order, ``surface energy'' contribution of the form
$$
E_{surf}(u)=\|D^2 u\|_{TV((0,1)^2)}
$$
in order to detect the energy scaling of microstructures via minimization  and to prove Theorem \ref{thm:p_dependence_two_wells}. 
We thus study upper and lower scaling-bounds of the total energy
\begin{equation}\label{eq:tot_p_en}
E_\epsilon^{(p)}(u):=E_{el}^{(p)}(u)+\epsilon E_{surf}(u)
\end{equation}
where $p\in[1,\infty)$.
We also denote the \emph{localized total energy} by
\begin{equation}\label{eq:tot_p_en_loc}
E_\epsilon^{(p)}(u;\Omega)=\int_{\Omega}\dist^p(\nabla u,\{A,B\})dx+\|D^2 u\|_{TV(\Omega)}
\end{equation}
for every $\Omega\subset (0,1)^2$.
Both for the upper and lower bounds we borrow the techniques used in \cite{CC15}, as they can be adapted to general $p$-growth conditions without much effort.  

\subsection{Upper bounds: first-order branching}
We obtain the upper scaling bounds of energies of the form \eqref{eq:tot_p_en} by means of a branching construction. These constructions originate from the work \cite{KM1, KM2}, similar constructions are present in the unpublished manuscript \cite{CM97} which was made available to us.
The result below, proved in the case of $p=2$ in \cite[Lemma 2.1]{CC15}, quantifies the total energy contribution of a building block of a branching construction.
We here repeat the proof for the reader's convenience.

\begin{figure}[t]
\begin{center}
\includegraphics{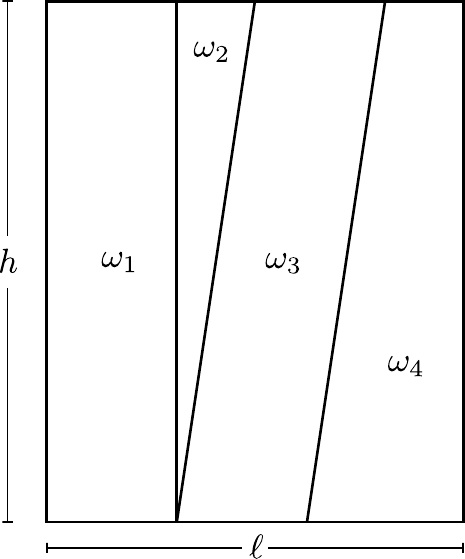}
\end{center}
\caption{Example of a building block of a (vertical) branching construction as given in Lemma \ref{lem:loc-vert}.}
\label{fig:cell2}
\end{figure}

\begin{lem}\label{lem:loc-vert}
Let $0<\ell<h\le1$, $\omega=[0,\ell]\times[0,h]$, $A,B$ as in \eqref{eq:AB} and let $E_\epsilon^{(p)}$ be defined by \eqref{eq:tot_p_en_loc}.
Then there exists a piecewise affine function $v\in W^{1,\infty}(\omega;\R^2)$ such that
\begin{equation}\label{eq:lem-loc-vert-BC}
v(\ell,x_2)=v(0,x_2)=0, \quad \text{for every } x_2\in[0,h],
\end{equation}
\begin{equation}\label{eq:lem-loc-vert-self}
v\Big(\frac{x_1}{2},h\Big)=v\Big(\frac{\ell}{2}+\frac{x_1}{2},h\Big)=\frac{1}{2} v(x_1,0), \quad \text{for every } x_1\in[0,\ell],
\end{equation}
and
\begin{equation}\label{eq:lem-loc-vert-e}
E_\epsilon^{(p)} (v;\omega) +Per(\omega) \le (1-\lambda)^p \frac{\ell^{p+1}}{h^{p-1}}+8\epsilon h,
\end{equation}
for every $p\in[1,\infty)$ and $\epsilon>0$.
\end{lem}
\begin{proof}
We subdivide $\omega$ into the following subsets:
\begin{align*}
\omega_1 &= \Big\{(x_1,x_2)\in\omega \,:\, 0<x_1<\frac{\lambda\ell}{2}\Big\},\\
\omega_2 &= \Big\{(x_1,x_2)\in\omega \,:\, \frac{\lambda\ell}{2}<x_1<\frac{\lambda\ell}{2}+\frac{(1-\lambda)\ell\,x_2}{2h}\Big\},\\
\omega_3 &= \Big\{(x_1,x_2)\in\omega \,:\, \frac{\lambda\ell}{2}+\frac{(1-\lambda)\ell\,x_2}{2h}<x_1<\lambda\ell+\frac{(1-\lambda)\ell\,x_2}{2h}\Big\},\\
\omega_4 &= \Big\{(x_1,x_2)\in\omega \,:\, \lambda\ell+\frac{(1-\lambda)\ell\,x_2}{2h}<x_1<\ell\Big\},
\end{align*}
(see Figure \ref{fig:cell2})
and define $v$ with $v_2(x_1,x_2)=0$ by integration of $\p_1 v_1(\cdot,x_2)$ on $(x_1,x_2) \in  \omega_j$, where $\p_1v_1=1-\lambda$ on $\omega_1\cup\omega_3$ and $\p_1v_1=-\lambda$ on $\omega_2\cup\omega_4$ ;
$$
v_1(x_1,x_2)=\begin{cases}
(1-\lambda)x_1 & (x_1,x_2)\in\omega_1, \\
-\lambda x_1+\frac{\lambda\ell}{2} & (x_1,x_2)\in\omega_2 ,\\
(1-\lambda)x_1-\frac{(1-\lambda) \ell x_2}{2h} & (x_1,x_2)\in\omega_3, \\
-\lambda x_1+\lambda\ell & (x_1,x_2)\in\omega_4.
\end{cases}
$$
Here we have chosen the initial value of the integration to be zero at $x_1 =0$.
Thus, we have that $v\in W^{1,\infty}(\omega;\R^2)$,  \eqref{eq:lem-loc-vert-BC} and \eqref{eq:lem-loc-vert-self} are satisfied and
\eqref{eq:lem-loc-vert-e} follows by noticing that $\nabla v=A$ in $\omega_1$, $\nabla v=B$ in $\omega_2\cup\omega_4$ and
$$
\nabla v=\begin{pmatrix}1-\lambda&-\frac{(1-\lambda)\ell}{2h}\\0&0\end{pmatrix} \quad \text{in }\omega_3.
$$
\end{proof}

Using the building block construction from Lemma \ref{lem:loc-vert},
we now quantify the total energy of a periodic branching construction on rectangles of dimension $L\times H$ (see Figure \ref{fig:bra-proof}).
We mimic the proof of \cite[Lemma 2.3]{CC15}. Relying on this lemma in many further upper bound constructions in the following sections, we here formulate the result in general rectangles and thus slightly more general than needed for the upper bound in Theorem \ref{thm:p_dependence_two_wells}.

\begin{figure}[t]
\begin{center}
\includegraphics{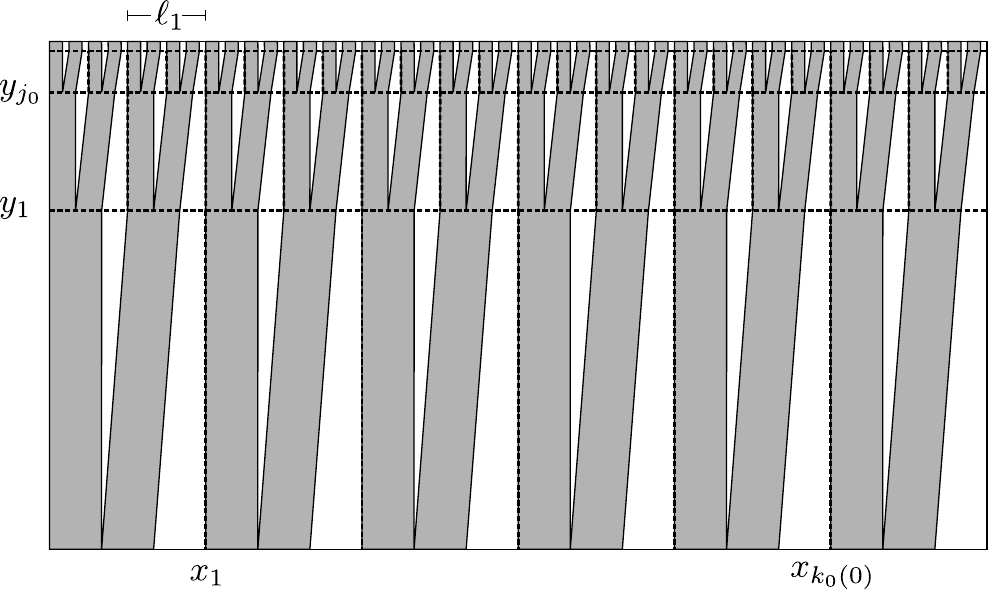}
\end{center}
\caption{Branching construction on $R'$ as described in the proof of Lemma \ref{prop:vert-bra}.}
\label{fig:bra-proof}
\end{figure}

\begin{lem}\label{prop:vert-bra}
Let $0<L, H\le1$, $x_0\in[0,1-L]\times[0,1-H]$, $R=x_0+[0,L]\times[0,H]$ and $N\in\N$ such that $N>\frac{4L}{H}$ and let $E_\epsilon^{(p)}$ be as in \eqref{eq:tot_p_en_loc}.
Let $A,B\in\diag(2,\R)$ be such that $A-B=|A-B|e_1\otimes e_1$, $F=\lambda A+(1-\lambda)B$ for some $0<\lambda<1$.
Then there exists $u\in W^{1,\infty}(R;\R^2)$ with $\nabla u\in BV(R;\R^{2\times 2})$ and with $u(x)=Fx+b$ for every $x\in\partial R$, $b\in\R^2$ such that
\begin{equation}\label{eq:prop-vbra}
E_\epsilon^{(p)}(u;R) \lesssim \dist^p(F,\{A,B\}) \frac{L^{p+1}}{N^p H^{p-1}}+\epsilon HN
\end{equation}
for every $p\in[1,\infty)$ and $\epsilon>0$.
\end{lem}
\begin{proof}
We first consider $A$ and $B$ as in \eqref{eq:AB}, $F=0$, $b=0$ and $x_0=0$.
We will then recover the general case at the end of the proof.

Let $\theta$ be a geometric constant such that
\begin{equation}\label{eq:theta}
\theta\in\begin{cases}
(2^{-\frac{p}{p-1}},\frac{1}{2}) & p>1, \\
(0,\frac{1}{2}) & p=1.
\end{cases}
\end{equation}
Consider the upper half of the rectangle $R$, that is $R':=[0,L]\times[\frac{H}{2},H]$.
Set
\begin{equation*}\label{eq:cut-i}
y_j=H-\frac{H}{2}\theta^j,
\end{equation*}
we cut $R'$ horizontally at height $y_j$ obtaining the intervals $[0,L]\times[y_j,y_{j+1}]$ with $j\in\{0,\dots,j_0+1\}$ where $j_0$ is chosen below.
We further subdivide each of these horizontal rectangles in $N2^j$ smaller intervals (of the same width), by cutting $[0,L]\times[y_j,y_{j+1}]$ vertically.
Set
\begin{equation}\label{eq:dim-i}
\ell_j=\frac{L}{2^j N}
\quad \text{and} \quad
h_j=y_{j+1}-y_j=\theta^j\frac{H(1-\theta)}{2},
\end{equation}
we have reduced $R'$ to a union of rectangles of dimensions $\ell_i\times h_i$ (see Figure \ref{fig:bra-proof}).
In order to use Lemma \ref{lem:loc-vert} to define a self-similar construction, we stop this subdivision when $j=j_0$, with $j_0$ the last index for which $\ell_{j_0}<h_{j_0}$.
Notice that on the one hand by the condition $N>\frac{4L}{H}$, $\ell_0<h_0$, and on the other hand, by the condition that $\theta<\frac{1}{2}$, for all $j>0$ large enough, it holds that $\ell_j>h_j$. As a consequence such a $j_0$ exists.
We define the rectangles $\omega_{j,k}$ described in the lines above as follows
\begin{equation}\label{eq:cells-bra}
\omega_{j,k}:=\begin{cases}\displaystyle
(k\ell_j,y_j)+[0,\ell_j]\times[0,h_j] & j=0,\dots,j_0,\quad k=0,\dots, N2^j-1, \\ 
\\ \displaystyle
(k\ell_{j_0},y_{j_0+1})+[0,\ell_{j_0}]\times\Big[0,\frac{H}{2}\theta^{j_0+1}\Big] & j=j_0+1,\quad k=0,\dots, N2^{j_0}-1.
\end{cases}
\end{equation}
From Lemma \ref{lem:loc-vert}, for every $j\in\{0,\dots,j_0\}$, we can find on each rectangle $[0,\ell_j]\times[0,h_j]$ a test function $v_j:[0,\ell_j]\times [0,h_j] \rightarrow \R^2$ satisfying \eqref{eq:lem-loc-vert-BC}, \eqref{eq:lem-loc-vert-self} and \eqref{eq:lem-loc-vert-e}.
We thus define $v\in W^{1,\infty}(R';\R^2)$ by setting
$$
v(x):=\begin{cases}
v_j(x-(k\ell_j,y_j)) &x\in\omega_{j,k},\\ 
\\ \displaystyle
\frac{H-x_2}{H-y_{j_0+1}}v_{j_0}(x_1-k\ell_{j_0},h_{j_0}) &x\in\omega_{j_0+1,k},
\end{cases}
$$
for every $j\in\{0,\dots,j_0+1\}$, $k\in\{0,\dots,k_0(j)\}$ with $k_0(j):=N2^j$ if $j\le j_0$ and $k_0(j_0+1):=N2^{j_0}$. 
Reasoning symmetrically on $[0,L]\times[0,\frac{H}{2}]$, we obtain $v\in W_0^{1,\infty}(R;\R^2)$ thanks to \eqref{eq:lem-loc-vert-BC} and \eqref{eq:lem-loc-vert-self}.
By construction and \eqref{eq:lem-loc-vert-e} we have
\begin{align*}
E_\epsilon^{(p)}(v;R) &\le \sum_{j=0}^{j_0+1}\sum_{k=0}^{k_0(j)}E_\epsilon^{(p)}(v;\omega_{j,k}) +Per(\omega_{j,k}) \lesssim \sum_{j=0}^{j_0+1}N2^j\Big((1-\lambda)^p\frac{\ell_j^{p+1}}{h_j^{p-1}}+8\epsilon h_j\Big).
\end{align*}
In the inequality above we have also used the fact that the contribution of the cut-off term on $\omega_{j_0+1,k}$ is of the order $h_{j_0}\ell_{j_0}+\epsilon h_{j_0}$ which is comparable with the $(j_0+1)$-th term of the sum above for the definition of $j_0$.
Substituting \eqref{eq:dim-i} into this expression, we obtain
$$
E_\epsilon^{(p)}(v;R) \lesssim \sum_{j=0}^\infty\Big((1-\lambda)^p\frac{L^{p+1}}{N^p H^{p-1}}\Big(\frac{1}{2^p\theta^{p-1}}\Big)^j+\epsilon (2\theta)^j HN\Big).
$$
This is a converging series thanks to the condition \eqref{eq:theta} and eventually, we obtain
\begin{equation}\label{eq:lem-bra-simp}
E_\epsilon^{(p)}(v;R)\lesssim (1-\lambda)^p\frac{L^{p+1}}{N^p H^{p-1}}+\epsilon HN.
\end{equation}

We now return to the general case.
Then considering $u(x)=|A-B|v(x+x_0)+Fx+b$  with $v$ as constructed in the particular case from above, the following relation holds
$$
\dist(\nabla u,\{A,B\})=|A-B|\dist\Big(\nabla v,\Big\{\begin{pmatrix}1-\lambda&0\\0&0\end{pmatrix},\begin{pmatrix}-\lambda&0\\0&0\end{pmatrix}\Big\}\Big)
$$
and \eqref{eq:prop-vbra} comes from \eqref{eq:lem-bra-simp} and by possibly switching the roles of $A$ and $B$ in the construction of Lemma \ref{lem:loc-vert} so that $|A-B|(1-\lambda)=\dist(F,\{A,B\})$.
\end{proof}

We refer to the construction in the normalized setting (for which $A,B$ are as in \eqref{eq:AB}) of the proof of Lemma \ref{prop:vert-bra} as a branching construction with \emph{oscillation in the direction $e_1$} and \emph{branching in the direction $e_2$}.

An immediate consequence of the previous lemma is the following upper scaling bound for the two-well problem for elastic energies of $p$-growth.

\begin{prop}\label{prop:first-order}
Let $A,B\in\diag(2,\R)$ be such that $\rank(A-B)=1$, let $\lambda$ and $F$ be as in the statement of Lemma \ref{prop:vert-bra} and let $E_\epsilon^{(p)}$ be defined as in \eqref{eq:tot_p_en}.
Then, for every $\epsilon\in(0,1)$ sufficiently small there exists $u_\epsilon\in W^{1,\infty}((0,1)^2;\R^2)$ with $\nabla u_\epsilon\in BV(\R^2;\R^{2\times2})$ and with $u_\epsilon(x)=Fx+b$ for every $x\in\p(0,1)^2$ and $b\in\R^2$ such that
$$
E_\epsilon^{(p)}(u_\epsilon)\lesssim\dist^p(F,\{A,B\})\epsilon^\frac{p}{p+1}.
$$
\end{prop}
\begin{proof}
Up to switching the roles of $x_1$ and $x_2$ we can reduce to the case $A-B=|A-B|e_1\otimes e_1$. 
We then apply Lemma \ref{prop:vert-bra} with $H=L=1$ and $N=\frac{4}{r}$ for some $0<r<1$ obtaining $u$ such that
$$
E_\epsilon^{(p)}(u)\lesssim\dist^p(F,\{A,B\})\Big(r^p+\frac{\epsilon}{r}\Big).
$$
Optimizing in $r$, we infer that $r\sim\epsilon^\frac{1}{p+1}$, and the result follows.
\end{proof}

\subsection{Lower bounds through a localization argument}
In order to obtain lower scaling bounds for every $p\in[1,\infty)$, we proceed along the lines of \cite{CC15} which in contrast to other, e.g. Fourier-based techniques, generalizes easily to arbitrary values of $p\in[1,\infty)$.

We premise the following well-known result which will be useful in the sequel.
This is a Poincar\'e inequality for $BV$ functions stated in \cite{CC15}, compounding the results from \cite[Theorem 3.44 and Remark 3.45]{AFP}.

\begin{lem}\label{lem:poincBV}
Let $A \subset \R^n$ be a bounded, connected, Lipschitz regular domain.
Then for every $v\in BV_{loc}(\R^n;\R^m)$ there holds
$$
\int_{A}|v-\langle v\rangle_{A}|dx \le C \|Dv\|_{TV}(A)
$$
where $C\le\diam(A)$ is a positive constant.
\end{lem}

Exploiting this together with the ideas from \cite{CC15}, we obtain the following lower bound scaling estimate:

\begin{prop}\label{prop:lb_p}
Let $A,B\in\diag(2,\R)$ be such that $\rank(A-B)=1$, $F=\lambda A+(1-\lambda)B$ for some $0<\lambda<1$ and let $E_\epsilon^{(p)}$ be defined as in \eqref{eq:en_p_min} with $K=\{A,B\}$.
Then, for every $\epsilon>0$ there holds
\begin{equation*}
E_\epsilon^{(p)}(F) \gtrsim \min\{\dist^p(F,K),\dist(F,K)\} \epsilon^\frac{p}{p+1}.
\end{equation*}
\end{prop}

\begin{proof}
From the translation invariance of the energy we can reduce to consider functions $u=Fx$ in $\R^2\setminus[0,1]^2$ (i.e.~taking $b=0$ in \eqref{eq:en_p_min}).
Let $u\in W^{1,p}_{loc}(\R^2;\R^2)$ be given such that $u(x)=Fx$ on $\R^2\setminus[0,1]^2$. Reasoning as in the proof of \cite[Lemma 3.1]{CC15}, for every $0<\mu<1$ small enough we can find a vertical stripe $S=[s,s+\mu]\times[0,1]\subset[0,1]^2$ such that
\begin{align}\label{eq:loc_p1}
E_\epsilon^{(p)}(u;S) \lesssim \mu E_\epsilon^{(p)}(u),
\end{align}
where $E_\epsilon^{(p)}(u;\cdot)$ is as in \eqref{eq:tot_p_en_loc}.
Consider now
$$
\tilde{F}(x_2):=\frac{1}{\mu}\int_{s}^{s+\mu}\nabla u(t,x_2)dt
$$
and define $f_1:[0,1]\to\R$ as the first entry of a projection of $\tilde{F}$ onto the set $K$, e.g.,

$$
f_1(x_2):=\begin{cases}
A_{1,1} & \text{if } |\tilde F(x_2)-A|\le|\tilde F(x_2)-B|, \\
B_{1,1} & \text{if } |\tilde F(x_2)-B|<|\tilde F(x_2)-A|.
\end{cases}
$$
From the triangle inequality we have
\begin{align*}
\|\p_1 u_1-f_1\|_{L^1(S)} &\le \|\p_1 u_1-\tilde{F}_{1,1}\|_{L^1(S)}+\|\dist(\tilde{F}_{1,1},\{A_{1,1},B_{1,1}\})\|_{L^1(S)} \\
&\le 2\|\p_1 u_1-\tilde{F}_{1,1}\|_{L^1(S)}+\|\dist(\nabla u,K)\|_{L^1(S)}.
\end{align*}
We apply 
Lemma \ref{lem:poincBV} on the right-hand-side above with $\p_1 u_1(\cdot,x_2)-\tilde{F}_{1,1}(x_2)$ in the interval $(s,s+\mu)$ for every $x_2\in(0,1)$.
This yields
\begin{align*}
\|\p_1 u_1-f_1\|_{L^1(S)} \lesssim \mu \|\p_1^2 u_1\|_{TV}(S)+\|\dist(\nabla u,K)\|_{L^1(S)}.
\end{align*}
From \eqref{eq:loc_p1} and H\"older's inequality we obtain
\begin{align*}
\|\p_1 u_1-f_1\|_{L^1(S)} \le \frac{\mu^2}{\epsilon}E_\epsilon^{(p)}(u)+\mu E_\epsilon^{(p)}(u)^\frac{1}{p}.
\end{align*}
The (standard) Poincar\'e-Wirtinger inequality in the $x_1$ variable implies
\begin{equation}\label{eq:imp1}
\int_{S}|u_1(x_1,x_2)-f_1(x_2)x_1+a(x_2)|dx_1dx_2 \le \frac{\mu^3}{\epsilon}E_\epsilon^{(p)}(u)+\mu^2 E_\epsilon^{(p)}(u)^\frac{1}{p}
\end{equation}
for some function $a$ depending only on $x_2$.
Moreover, (standard) Poincar\'e's inequality in the $x_1$ variable, the diagonal structure of the matrices in $K$ and \eqref{eq:loc_p1} give
\begin{equation}\label{eq:imp2}
\|u_1(x)-F_{1,1}x_1\|_{L^1(S)}\le\|\p_2 u_1\|_{L^1(S)}\le \mu E_\epsilon^{(p)}(u)^\frac{1}{p}.
\end{equation}
Notice that in the application of Poincar\'e inequality above we used that $u(x_1,0)-F_{1,1}x_1\equiv 0$.
Combining \eqref{eq:imp1} and \eqref{eq:imp2}, the triangle inequality yields
\begin{align*}
\int_{S}\big|(F_{1,1}-f_1(x_2))x_1+a(x_2)\big|dx_1dx_2 &\le \frac{\mu^3}{\epsilon}E_\epsilon^{(p)}(u)+(\mu+\mu^2) E_\epsilon^{(p)}(u)^\frac{1}{p} \\
& \lesssim \frac{\mu^3}{\epsilon}E_\epsilon^{(p)}(u)+\mu E_\epsilon^{(p)}(u)^\frac{1}{p}.
\end{align*}
Recalling the definition of $f_1$, we notice that by linearity in $x_1$ the left-hand-side above is bounded from below by $\dist(F,K)\mu^2$; hence
$$
\dist(F,K)\mu^2\lesssim \frac{\mu^3}{\epsilon}E_\epsilon^{(p)}(u)+\mu E_\epsilon^{(p)}(u)^\frac{1}{p}.
$$
From this we obtain that
$$
\dist(F,K)\mu^2\lesssim\max\Big\{\frac{\mu^3}{\epsilon}E_\epsilon^{(p)}(u),\mu E_\epsilon^{(p)}(u)^\frac{1}{p}\Big\},
$$
which in particular yields that either $E_\epsilon^{(p)}(u)\gtrsim\frac{\epsilon}{\mu}\dist(F,K)$ or $E_\epsilon^{(p)}(u)\gtrsim\mu^p\dist^p(F,K)$ and therefore
$$
E_\epsilon^{(p)}(u)\gtrsim\min\Big\{\dist(F,K)\frac{\epsilon}{\mu},\dist^p(F,K)\mu^p\Big\}\gtrsim \min\{\dist(F,K),\dist^p(F,K)\}\epsilon^\frac{p}{p+1},
$$
where the last step comes from the optimization $\mu\sim\epsilon^\frac{1}{p+1}$.
\end{proof}

Eventually, the combination of Propositions \ref{prop:first-order} and \ref{prop:lb_p} conclude the proof of Theorem \ref{thm:p_dependence_two_wells}.

\section{Preliminaries for the $L^2$-Based Lower Bound Estimates for Higher Order Laminates}

\label{sec:prelim_lower}

Heading towards the proofs of Theorems \ref{thm:K3}-\ref{thm:Kn}, in this section, we present some auxiliary results which will be used in the derivation of the lower bounds. Many of the arguments follow similar ideas as in \cite{RT21} and \cite{KW16} with only smaller changes necessary due to the $n$-dimensional setting in this article. They strongly rely on $L^2$-based Fourier techniques.

\subsection{The elastic energy and its Fourier multiplier} 

We begin by deducing the multiplier formulation of the elastic energy for which we work with periodic extensions of the function $v:= u-Fx$.

\begin{lem}[Fourier characterization of the elastic energy]
\label{lem:Fourier}
Let $\Omega =(0,1)^n \subset \R^n$.
Let $F\in \diag(\R,n)$ and let $\chi_j\in \{0,1\}$, $j\in\{1,\dots,n+1\}$, denote characteristic functions such that $\chi_j \in BV(\Omega,\{0,1\})$, and $\sum\limits_{j=1}^{n+1}\chi_j = 1$. Let further 
\begin{align}
\label{eq:energy_aux_sec}
E_{el,n+1}(\chi,F):= \inf\limits_{u: \ \nabla u = F \text{ in } \R^n\setminus \overline{\Omega}} \int\limits_{\Omega}|\nabla u - \chi|^2 dx,
\end{align}
where $\chi= \sum\limits_{j=1}^{n+1} \chi_j A_j$ and define $\tilde{\chi}:= \sum\limits_{j=1}^{n+1} \chi_j (A_j-F)= \chi - F: \Omega \rightarrow \R^{n\times n}$.
Then, identifying $\tilde{\chi}$ with a function on the torus $\T^n$, we obtain
\begin{align*}
E_{el,n+1}(\chi,F) \geq \sum\limits_{j=1}^{n}  \sum\limits_{\ell\neq j}\sum\limits_{k\in\Z^n}  \frac{k_{\ell}^2}{|k|^2} |\mathcal{F}{\tilde{\chi}}_{j,j}(k)|^2 ,
\end{align*}
with the convention that the multiplier is equal to one for $k=0$.
\end{lem}

\begin{proof}
We first note that in what follows, we view $\tilde{\chi}$ as a function on $\T^n$. Next, we observe that by the assumptions for the characteristic functions $\chi_j$ it holds that
\begin{align*}
E_{el,n+1}(\chi,F)&:= \inf\limits_{u \in W^{1,2}_{loc}(\R^n, \R^n): \ \nabla u = F \text{ in } \R^n\setminus \overline{\Omega}} \int\limits_{\Omega}|\nabla u - \chi|^2 dx\\
&= \inf\limits_{v \in W^{1,2}_{loc}(\R^n, \R^n): \ \nabla v = 0 \text{ in } \R^n\setminus \overline{\Omega}} \int\limits_{\Omega}|\nabla v - \tilde{\chi}|^2 dx\\
&= \inf\limits_{v\in W^{1,2}_0(\Omega, \R^n)} \int\limits_{\Omega}|\nabla v - \tilde{\chi}|^2 dx\\
&\geq \inf\limits_{v\in W^{1,2}(\T^n, \R^n): \ \langle v\rangle = 0 } \int\limits_{\T^n}|\nabla v - \tilde{\chi}|^2 dx.
\end{align*}
Here, in the last inequality, we have viewed $W^{1,2}_0(\Omega,\R^n)$ functions as a subset of the periodic functions and have changed the integration to an integration over $\T^n$.
Hence, an application of the Fourier transform yields that 
\begin{align*}
 \int\limits_{\T^n}|\nabla v - \tilde{\chi}|^2 dx = \sum\limits_{k\in\Z^n}|2\pi \F{v}\otimes ik - \F{\tilde{\chi}}|^2 .
\end{align*}
Minimizing this leads to the Euler Lagrange equation
\begin{align*}
2\pi(\F v \otimes ik)k = \F \tilde{\chi} k \mbox{ for all } k \in \R^n \setminus \{0\}.
\end{align*}
Thus, by the diagonal structure of the matrices $\tilde{\chi}$, we obtain $\F v_j = - \frac{i}{2\pi} \frac{k_j}{|k|^2} \F \tilde{\chi}_{j,j}$ for $j\in\{1,\dots,n\}$  and $k\neq 0$. Inserting this into the expression for $E_{el,n+1}$ and taking into account the case $k=0$ separately, implies the desired result. We note that the above function $v$ yields a minimizer for the periodic elastic energy given a fixed matrix-valued function $\chi$.
\end{proof}

\subsection{Low frequency bounds}
As a direct consequence of the identity relating $v$ and $\chi$, we observe the following estimate which we will use in controlling low frequencies.

\begin{lem}
\label{lem:axes_estimates}
Let $\Omega = (0,1)^n$ and let 
$\tilde{\chi}_{j,j}$ be as in Lemma \ref{lem:Fourier} (viewed as a function on $\T^n$). Then, for every $\mu\in(0,1)$ and for every $j\in \{1,\dots,n\}$ there holds
\begin{align*}
\sum\limits_{\{k_j \in \Z: \ |k_j|\leq \mu^{-1}\}}  |\F \tilde{\chi}_{j,j}(0,\dots,0,k_j,0,\dots,0)|^2 \leq C \mu^{-2} E_{el,n+1}(\chi,F).
\end{align*}
\end{lem}

\begin{rmk}
In the result above the boundary conditions play a very important role.
This becomes clear when looking at the stress-free case.
Indeed, in case of affine boundary conditions, purely one-dimensional solutions must be constants, whereas in the presence of periodic conditions we have nontrivial laminate solutions.

With this point of view, one can see that in order to rule out (nontrivial) one-dimensional solutions, it is sufficient to consider affine data only on two (couples of) facets of $\p\Omega$ that are orthogonal to two independent directions, see also Remark \ref{rmk:mixed-BC} below.

We emphasize that the proof of this lemma fails in general for the ``periodic setting'' in which the elastic energy $E_{el}(u,\chi)$ is minimized among $u\in W^{1,2}_{loc}(\R^n;\R^n)$ such that $\nabla u$ is periodic with $\langle \nabla u \rangle = F$. This however is not only a technical artifact of the method of proof; in fact, in the periodic case, a different (shifted) scaling behaviour is expected. We refer to Section \ref{subsec:periodic3grad} for an example of such a result.
\end{rmk}

\begin{proof}

Let $v\in W^{1,2}_0(\Omega;\R^n)$.
We first estimate $\F v_j$. Assuming first that $j=1$, Plancherel's formula, the vanishing Dirichlet conditions and the fundamental theorem of calculus give
\begin{align*}
\sum\limits_{|k_1| \leq \mu^{-1}} |\F v_1(k_1,0,\dots,0)|^2
&\leq \sum\limits_{k_1 \in \Z} |\F v_1(k_1,0,\dots,0)|^2 \\
&= \int_0^1\Big( \int_{[0,1]^{n-1}}  v_1(x_1,x_2,\dots,x_n) dx_2\dots dx_n  \Big)^2  dx_1\\
&= \int_0^1\Big( \int_{[0,1]^{n-1}}\int_0^{x_2} \p_2 v_1(x_1,s,x_3,\dots,x_n) ds dx_2 \dots dx_n \Big)^2  dx_1 \\
& \le \int_0^1\Big( \int_{[0,1]^{n-1}} |\p_2 v_1(x_1,s,x_3,\dots,x_n)| ds dx_3 \dots dx_n \Big)^2  dx_1.
\end{align*}
Next, we specify $v \in W^{1,2}_0(\Omega, \R^n)$ to be such that $\int\limits_{\Omega}|\nabla v - \tilde{\chi}|^2 dx \leq 2 E_{el,n+1}(\chi,F)$.
This is always possible since $E_{el,n+1}(\chi,F)=0$ if and only if $\chi\equiv F$, as a consequence of Lemma \ref{lem:Fourier}.
With this specification, we obtain by
Jensen's inequality and the diagonal structure of $\tilde\chi$
\begin{align}
\label{eq:v11}
\sum\limits_{|k_1| \leq \mu^{-1}} |\F v_1(k_1,0,\dots,0)|^2 \leq \|\p_2 v_1\|_{L^2(\T^2)}^2 \leq 2 E_{el,n+1}(\chi,F).
\end{align}
Moreover,
\begin{align*}
\sum\limits_{|k_1|\leq \mu^{-1}} |\F \tilde{\chi}_{1,1}(k_1,0,\dots,0)|^2 
& \leq 2\sum\limits_{|k_1|\leq \mu^{-1}} |\F(\p_1 v_1) (k_1,0,\dots,0)- \F \tilde{\chi}_{1,1}(k_1,0,\dots,0)|^2 \\
& \quad + 2\sum\limits_{|k_1|\leq \mu^{-1}} |\F(\p_1 v_1)(k_1,0,\dots,0) |^2 \\
& \leq 4 E_{el,n+1}(\chi,F) + 8\pi^2 \mu^{-2} \sum\limits_{|k_1| \leq \mu^{-1}} |\F v_1(k_1,0,\dots,0)|^2\\
& \leq C \mu^{-2} E_{el,n+1}(\chi,F).
\end{align*}
In the last step, we have used the estimate \eqref{eq:v11}.

Let now $j>1$. With an analogous reasoning we obtain
\begin{multline*}
\sum_{|k_j|<\mu^{-1}}|\F v_j(0,\dots,0,k_j,0,\dots,0)|^2 \\
\le \int_0^1\Big(\int_{[0,1]^{n-1}}|\p_1 v_j(s,x_2,\dots,x_n)|ds dx_2\dots dx_{j-1} dx_{j+1}\dots dx_n\Big)^2dx_j.
\end{multline*}

As in the case $j=1$, specifying $v \in W^{1,2}_0(\Omega,\R^n)$ to be such that $\int\limits_{\Omega}|\nabla v - \tilde{\chi}|^2 dx \leq 2 E_{el,n+1}(\chi,F)$ then yields
\begin{align*}
\sum\limits_{\{|k_j| \leq\mu^{-1}\}} |\F v_j(0,\dots,k_j,0,\dots,0)|^2 \leq \|\p_1 v_j\|_{L^2(\T^2)}^2 \leq C E_{el,n+1}(\chi,F).
\end{align*}
Combining the above bounds concludes the proof.
\end{proof}

\begin{rmk}\label{rmk:mixed-BC}
We highlight that, in the proof of Lemma \ref{lem:axes_estimates}, the homogeneous boundary condition on $v$ can be relaxed, requiring $v$ to be periodic and vanishing only on two couples of (opposite) facets of the cube $\Omega$.
Namely, given $j',j''\in\{1,\dots,n\}$, $j'\neq j''$ we let $v:\T^n\to\R$ be such that
$$v(x_1,\dots,x_{h-1},0,x_{h+1},\dots,x_n)=0 =v(x_1,\dots,x_{h-1},1,x_{h+1},\dots,x_n) \text{ with } h=j',j''.$$
Indeed, in applying the fundamental theorem of calculus in the proof, we have only used that $v_1(x_1,0,x_3,\dots,x_n)=0=v_1(x_1,1,x_3,\dots,x_n)$ for all $x_1,x_3,\dots,x_n \in (0,1)$ when $j=1$, and that $v_j(0,x_2,\dots,x_n)=0=v_j(1,x_2,\dots,x_n)$ for all $x_2,\dots,x_n \in (0,1)$ when $j>1$.

In particular, this means that the result of Lemma \ref{lem:axes_estimates} still holds for elastic energies
$$
E_{el,n+1}(\chi,F):=\inf\Big\{\int_\Omega|\nabla u-\chi|^2dx \,:\, \nabla u\in\T^n,\, \langle\nabla u\rangle=F,\, u=Fx+b \text{ on } \Gamma_{j'}\cup\Gamma_{j''}\Big\}
$$
where $\Gamma_{j}=\{x\in\p\Omega\,:\, x_j=0,1\}$.

For such boundary conditions one can find constructions (simpler than those shown in Sections \ref{sec:proof4_wells_lower} and \ref{sec:low_n}) matching the upper bounds of Theorem \ref{thm:Kn}.
These are the most general boundary conditions that ``force'' the highest order of lamination possible, see also the discussions in 
Section \ref{subsec:periodic3grad}.
\end{rmk}

\subsection{High frequency estimates}

Next, we recall the phase-space bounds resulting from the interplay between the Fourier multiplier for the elastic energy and the surface energy. This leads to low frequency localization results in certain truncated cones. The argument for this follows along the lines of \cite{RT21} but is formulated for $n$ dimensions here. 

For arbitrary but fixed $\mu,\mu_2>0$, we start by introducing the notation for truncated cones which we will be using in the sequel:
\begin{align}\label{eq:cones-def}
C_{j,\mu, \mu_2}:= \{k\in \Z^n: \ \sum\limits_{\ell \neq j}|k_{\ell}|^2 \leq \mu^2 |k|^2, \ |k|\leq \mu_2\}, \ j \in \{1,\dots,n\}.
\end{align}
Corresponding to these truncated cones, we also assign associated Fourier multipliers $\chi_{j,\mu,\mu_2}(D)$. These are determined by non-negative $C^{\infty}(\R^n)$ functions such that $\chi_{j,\mu,\mu_2}(k)=1$ for $k\in C_{j,\mu,\mu_2}$ and $\chi_{j,\mu,\mu_2}(k)=0$ for $k\in\Z^n\setminus C_{j,2\mu,2\mu_2}$. For $|k| \leq \mu_2 $ they are chosen to be essentially zero-homogeneous (due to technical reasons in the regularity conditions required for applying the transference principle, a small neighbourhood of the zero frequency has to be treated separately). A possible explicit choice would for instance be
$\chi_{j,\mu,\mu_2}(k)=(1-\psi(k))\varphi(\frac{k_j}{\mu|k|})\varphi(\frac{|k|}{\mu_2})+\psi(k)$ where $\varphi\in C^\infty(\R)$ is a positive function which equals $1$ on $[-1,1]$ and vanishes outside $[-2,2]$ and $\psi\in C^\infty(\R^n)$ which equals $1$ on $B_\frac{1}{2}$ and vanishes on $\R^n\setminus B_1$. The multipier $\chi_{j,\mu,\mu_2}(D)$ is then defined as in \eqref{eq:mult}.
In this article we will almost exclusively restrict to the choice $\mu= \epsilon^{\alpha}$ for some $\alpha \in (0,1)$ which will be determined in the respective settings below. The role of $\mu_2$ will vary depending on the iteration step in our argument.

Further, we recall the notation $\tilde{\chi}_{j,j}:= \chi_{j,j}-F_{j,j}$ from \eqref{eq:char_func_norm}, where the functions $\chi_{j,j}$ denote the components of the diagonal matrix ${\chi}$ from above and where $F\in \diag(\R,n)$.

\begin{lem}[A first localization result to cones]
\label{lem:first_loc}
Let $\mu\in(0,1), \mu_2>0$ and let $E_{el,n+1}(\cdot,\cdot)$  be as in \eqref{eq:energy_aux_sec} and $E_{surf}(\chi):= \sum\limits_{j=1}^{n+1}\|D\chi_j\|_{TV}(\Omega)$. Then,
\begin{align*}
\sum\limits_{j=1}^{n}\|\F \tilde\chi_{j,j}- \chi_{j,\mu,\mu_2}(D)\F \tilde \chi_{j,j}\|_{L^2}^2 \lesssim \mu_2^{-1} (E_{surf}(\chi) + \Per(\Omega)) + \mu^{-2}E_{el,n+1}(\chi,F) .
\end{align*}
\end{lem}

\begin{proof}
We argue in two steps: First, using the surface energy, we show that high frequency contribution can be cut off. Secondly, using the multiplier for the elastic energy, we show that it suffices to restrict to certain cones.

\emph{Step 1: High-frequency cut-off.} We argue as in the proof of \cite[Lemma 4.3]{KKO13}. Using that $|\chi_j| \leq 1$, we obtain that for any $c\in \R^n$ it holds that
\begin{align*}
\Per(\Omega)+\|\nabla \chi_{j}\|_{TV((0,1)^n)} &\geq \|\nabla \chi_j\|_{TV(\T^n)} \\
&\geq \frac{1}{|c|} \int\limits_{\T^n} |\chi_j - \chi_j(\cdot + c)|^2 dx \geq \frac{1}{|c|} \sum\limits_{k\in \Z^n} |(1-e^{ic\cdot k}) \F \chi_{j}|^2.
\end{align*}
Here, as above, we have viewed $\chi_j$ as a function on $\T^n$ after extending it periodically.
Integrating over $\partial B_{|c|}$ we hence deduce that
\begin{align*}
|c|^2 \|\nabla \chi_{j}\|_{TV(\T^n)}  \geq C |c| \sum\limits_{\{k \in \Z^n: \ |k|\geq \frac{1}{|c|}\}}|\F \chi_{j}|^2 . 
\end{align*}
Choosing $|c|=\mu_2^{-1}$ and summing over $j\in \{1,\dots, n+1\}$ yields the following bound
\begin{align}
\label{eq:bd1high}
C'\sum_{j=1}^n\sum_{\{|k|\ge\mu_2\}}|\F \chi_{j,j}|^2\le\sum\limits_{j=1}^{n+1} \sum\limits_{\{ |k|\geq \mu_2 \}} |\F \chi_j|^2 dk \leq C \mu_2^{-1}(E_{surf}(\chi) + \Per(\Omega)).
\end{align}

\emph{Step 2: Conical cut-off.}
Using the multiplier from Lemma \ref{lem:Fourier} in combination with the definition of $\tilde{\chi}_{j,j}$, we obtain
\begin{align}
\label{eq:bd2high}
\begin{split}
E_{el,n+1}(\chi,F)
&= \sum\limits_{j=1}^{n}  \sum\limits_{\ell\neq j}\sum\limits_{k\in \Z^n}  \frac{k_{\ell}^2}{|k|^2} |\mathcal{F}{\tilde{\chi}}_{j,j}(k)|^2 
\geq \mu^2 \sum\limits_{j=1}^{n}  \sum\limits_{\{\sum_{\ell\neq j}|k_\ell|^2>\mu^2|k|^2\}}  |\mathcal{F}{\tilde{\chi}}_{j,j}(k)|^2 
\end{split}
\end{align}

Combining the two cut-off bounds from \eqref{eq:bd1high}, \eqref{eq:bd2high} by a triangle inequality argument then yields the claim.
\end{proof}

Next, we formulate the following nonlinear commutation result:

\begin{lem}[A commutation result]
\label{lem:comm}
Let $g$ be a polynomial of degree $d\in \N \cup \{0\}$. Let $\lambda_j \in \R$ and assume that for some $\ell\in \{1,\dots,n\}$
\begin{align*}
\tilde{\chi}_{\ell, \ell} = g\left( \sum\limits_{j\neq \ell}\lambda_j \tilde{\chi}_{j,j} \right).
\end{align*}
Then, for any $\gamma \in (0,1)$ there exists a constant $C>0$ depending on $\gamma,g,F$ such that
\begin{align*}
&\|g(\sum\limits_{j\neq \ell} \lambda_j {\chi}_{j,\mu,\mu_2}(D)\tilde\chi_{j,j}) - \chi_{\ell, \mu, \mu_2}(D) g(\sum\limits_{j\neq \ell} \lambda_j \tilde \chi_{j,j})\|_{L^2} \\
&\leq C\sum\limits_{j\neq \ell} \psi_{d}\big(|\lambda_j|\big)  \psi_{1-\gamma}\big(\|\tilde\chi_{j,j}-\chi_{j,\mu,\mu_2}(D)\tilde\chi_{j,j}\|_{L^2}\big)+\|\tilde\chi_{\ell \ell}-\chi_{\ell,\mu,\mu_2}(D)\tilde \chi_{\ell \ell}\|_{L^2},
\end{align*}
where $\psi_\beta(z)=\max\{|z|,|z|^\beta\}$ for $\beta>0$.
\end{lem}

\begin{proof}
\emph{Step 1: Dealing with the lack of global Lipschitz bounds for $g$.}
Using the convention that $\lambda_{\ell} =0$, we begin by proving the following auxiliary result: For any $\gamma \in (0,1)$ there exists a constant $C>0$ depending on $g, \gamma$ such that
\begin{align}
\label{eq:aux_aim}
\|g(\sum\limits_{j=1}^{n} \lambda_j\tilde{\chi}_{j,j}) - g(\sum\limits_{j=1}^{n} \lambda_j {\chi}_{j,\mu,\mu_2}(D)\tilde\chi_{j,j})\|_{L^2} \leq C \sum\limits_{j=1}^{n}\max\{|\lambda_j|^{d},|\lambda_j| \}  \|\tilde{\chi}_{j,j}- {\chi}_{j,\mu,\mu_2}(D)\tilde\chi_{j,j} \|_{L^2}^{1-\gamma}.
\end{align}
By the triangle inequality, we may assume without loss of generality, that $g(t)=t^d$ for some $d\in \N \cup \{0\}$. We now argue as in \cite[Lemma 4.5]{RT21}. Expanding $a^d-b^d = (a-b)G(a,b)$ for $d\geq 1$ (and noticing that the differences cancel completely for $d=0$), where $G$ is a polynomial in $a,b$ of degree $d-1$, we obtain by H\"older's inequality for any $\gamma \in (0,1)$ that
\begin{align}
\label{eq:poly}
\begin{split}
&\|g(\sum\limits_{j=1}^{n} \lambda_j \tilde\chi_{j,j}) - g(\sum\limits_{j=1}^{n} \lambda_j {\chi}_{j,\mu,\mu_2}(D) \tilde \chi_{j,j})\|_{L^2}\\
&\leq \sum\limits_{j=1}^{n} |\lambda_j | \|(\tilde\chi_{j,j} -  {\chi}_{j,\mu,\mu_2}(D) \tilde{\chi}_{j,j} ) G(\sum\limits_{j=1}^{n} \lambda_j \tilde\chi_{j,j}, \sum\limits_{j=1}^{n} \lambda_j {\chi}_{j,\mu,\mu_2}(D) \tilde \chi_{j,j}) \|_{L^2}\\
&\leq \sum\limits_{j=1}^{n} |\lambda_j | \|\tilde\chi_{j,j} -  {\chi}_{j,\mu,\mu_2}(D) \tilde{\chi}_{j,j} \|_{L^{2+2\gamma}} \| G(\sum\limits_{j=1}^{n} \lambda_j \tilde\chi_{j,j}, \sum\limits_{j=1}^{n} \lambda_j {\chi}_{j,\mu,\mu_2}(D) \tilde \chi_{j,j}) \|_{L^{\frac{2+2\gamma}{\gamma}}}.
\end{split}
\end{align}
By interpolation, the $L^p$-$L^p$ boundedness of Fourier multipliers given by Proposition \ref{prop:grafakos} (which results in uniform in $\mu$ and $\mu_2$ bounds for the operator norms of $\chi_{j,\mu,\mu_2}(D)$ and $1-\chi_{j,\mu,\mu_2}(D)$) and the boundedness of the functions $\tilde{\chi}_{j,j}$, we further obtain
\begin{align}
\label{eq:interpol}
\begin{split}
&\|\tilde\chi_{j,j} -  {\chi}_{j,\mu,\mu_2}(D) \tilde{\chi}_{j,j} \|_{L^{2+2\gamma}}\\
&\leq \|\tilde\chi_{j,j} -  {\chi}_{j,\mu,\mu_2}(D) \tilde{\chi}_{j,j} \|_{L^2}^{1-\gamma} \|\tilde\chi_{j,j} -  {\chi}_{j,\mu,\mu_2}(D) \tilde{\chi}_{j,j} \|_{L^{\frac{2+2\gamma}{\gamma}}}^{\gamma}
\leq C \|\tilde\chi_{j,j} -  {\chi}_{j,\mu,\mu_2}(D) \tilde{\chi}_{j,j} \|_{L^2}^{1-\gamma},
\end{split}
\end{align}
where $C>0$ is a constant depending on $\gamma$ and $F$.
We deal with the remaining nonlinear contribution as in \cite{RT21} and estimate by H\"older's inequality
\begin{align}
\label{eq:nonlinear}
\begin{split}
&\| G(\sum\limits_{j=1}^{n} \lambda_j\tilde\chi_{j,j}, \sum\limits_{j=1}^{n} \lambda_j {\chi}_{j,\mu,\mu_2}(D) \tilde \chi_{j,j})\|_{L^{\frac{2+2\gamma}{\gamma}}}\\
&\leq \sum\limits_{h=0}^{d-1} \|\sum\limits_{j=1}^{n} \lambda_j\tilde\chi_{j,j}\|_{L^{\frac{(2+2\gamma)(d-1)}{\gamma}}}^{h} \|\sum\limits_{j=1}^{n} \lambda_j {\chi}_{j,\mu,\mu_2}(D) \tilde \chi_{j,j}\|_{L^{\frac{(2+2\gamma)(d-1)}{\gamma}}}^{d-1-h}.
\end{split}
\end{align}
Using again the (uniform) $L^p$-$L^p$ boundedness of our Fourier multipliers, the boundedness of the functions $\tilde{\chi}_{j,j}$ and inserting \eqref{eq:interpol}, \eqref{eq:nonlinear} into \eqref{eq:poly} thus finally yields \eqref{eq:aux_aim}.

\emph{Step 2: Conclusion.} With the result of Step 1 in hand, the conclusion now follows from the triangle inequality:
 \begin{align*}
 &\|g(\sum\limits_{j\neq \ell} \lambda_j {\chi}_{j,\mu,\mu_2}(D)\tilde\chi_{j,j}) - \chi_{\ell,\mu,\mu_2}(D) g( \sum\limits_{j\neq \ell} \lambda_j\tilde \chi_{j,j})\|_{L^2}\\
 & \leq 
   \|g(\sum\limits_{j=1}^{n} \lambda_j {\chi}_{j,\mu,\mu_2}(D)\tilde\chi_{j,j}) -   g(\sum\limits_{j=1}^{n} \lambda_j \tilde \chi_{j,j})\|_{L^2} +  \| g(\sum\limits_{j=1}^{n} \lambda_j \tilde \chi_{j,j})- \chi_{\ell,\mu,\mu_2}(D)g(\sum\limits_{j=1}^{n} \lambda_j \tilde \chi_{j,j})\|_{L^2}\\
& \leq C \sum\limits_{j\neq \ell}  \psi_{d}\big(|\lambda_j|\big) \psi_{1-\gamma}\big(\|\tilde\chi_{j,j}-\chi_{j,\mu,\mu_2}(D)\tilde\chi_{j,j}\|_{L^2}\big) + \|\tilde\chi_{\ell, \ell}-\chi_{\ell,\mu,\mu_2}(D)\tilde \chi_{\ell, \ell}\|_{L^2}.
 \end{align*}
In the second inequality we have used the assumption that $\tilde{\chi}_{\ell, \ell} = g\left( \sum\limits_{j \neq \ell} \lambda_j \tilde{\chi}_{j,j} \right)$.
\end{proof}

\begin{rmk}[The linear case]
The commutation result of Lemma \ref{lem:comm} is designed for nonlinear (polynomial) relations $g$, that is $d\ge 2$.
While the case $g$ constant (i.e. $d=0$) is trivial, it is worth commenting on the linear case $d=1$.
Indeed, take for simplicity $g(t)=t$, by simply adding and subtracting $\tilde\chi_{\ell,\ell}$ and applying the triangle inequality we have
\begin{multline*}
\|\sum_{j\neq\ell}\lambda_j\chi_{j,\mu,\mu_2}(D)\tilde\chi_{j,j}-\chi_{\ell,\mu,\mu_2}(D)\tilde\chi_{\ell,\ell}\|_{L^2} \\
\le \sum_{j\neq\ell}|\lambda_j|\|\tilde\chi_{j,j}-\chi_{j,\mu,\mu_2}(D)\tilde\chi_{j,j}\|_{L^2}+\|\tilde\chi_{\ell,\ell}-\chi_{\ell,\mu,\mu_2}(D)\tilde\chi_{\ell,\ell}\|_{L^2}.
\end{multline*}
We, in particular, highlight that there are no losses (manifested in the presence of the parameter $\gamma>0$) occurring due to the interpolation in the linear case. 
\end{rmk}

Finally, we combine the information from the previous two auxiliary results into an improved conical localization statement.

\begin{prop}[Exploiting the nonlinearity]\label{prop:cone-red-nonl}
Let  $\lambda_j \in \R$ for $j\in \{1,\dots,n\}$ and assume that for some $\ell\in \{1,\dots,n\}$ it holds that $\tilde{\chi}_{\ell, \ell} = g\left( \sum\limits_{j\neq \ell} \lambda_j \tilde\chi_{j,j} \right)$, where $g$ is a polynomial of degree $d\in \N\cup \{0\}$.
Let $\mu_3:=M\mu \mu_2$ for some $M>0$ depending on $d$.
Then, for any $\gamma \in (0,1)$ there exists a constant $C>0$ (depending on $g,\gamma, F$) such that
\begin{align*}
& \|\tilde\chi_{\ell, \ell}- \chi_{\ell,\mu,\mu_3}(D)\tilde \chi_{\ell, \ell}\|_{L^2}\\
& \leq C \sum\limits_{j\neq \ell} \psi_{d}\big(|\lambda_j|\big)  \psi_{1-\gamma}\big(\|\tilde\chi_{j,j}-\chi_{j,\mu,\mu_2}(D)\tilde\chi_{j,j}\|_{L^2}\big) + C\|\tilde{\chi}_{\ell, \ell} - \chi_{\ell, \mu, \mu_2}(D) \tilde{\chi}_{\ell, \ell}\|_{L^2}.
\end{align*}
\end{prop}

\begin{proof}
We argue as in \cite[Lemma 4.7]{RT21}. 
We first note that by the construction of the cones, we have that
\begin{align}
\label{eq:vanish}
\max\limits_{k\in C_{j,2\mu,2\mu_2}}|k_{\ell}| \leq 4\mu_2 \mu, \ j \neq \ell.
\end{align}
Hence, since $g$ is a polynomial and due to the fact that multiplication is turned into convolution by the Fourier transform and recalling that $\supp(\chi_{j,\mu,\mu_2})\cap\Z^n\subset C_{j,2\mu,2\mu_2}$, we obtain that
\begin{align*}
\F ( g(\sum\limits_{j \neq \ell}\lambda_j \chi_{j,\mu,\mu_2}(D) \tilde{\chi}_{j,j}))(k) = 0 \mbox{ for } |k_{\ell}|\geq M \mu \mu_2,
\end{align*}
where $M=M(d)>0$ depends on the degree of $g$.
Setting $\chi_{\ell,\mu \mu_2}$ to be the characteristic function of the set $\{k\in \R^n: \ |k_{\ell}|> M \mu \mu_2\}$ and recalling \eqref{eq:vanish} as well as Lemma \ref{lem:comm}, we infer that
\begin{align*}
&\|\chi_{\ell,\mu \mu_2}(D) \chi_{\ell,\mu,\mu_2}(D)g(\sum\limits_{j\neq \ell} \lambda_j \tilde\chi_{j,j})\|_{L^2}\\
&= \|\chi_{\ell,\mu \mu_2}(D) \big(\chi_{\ell,\mu,\mu_2}(D)g(\sum\limits_{j\neq \ell} \lambda_j \tilde\chi_{j,j}) - g(\sum\limits_{j\neq \ell} \lambda_j \chi_{j,\mu,\mu_2}(D) \tilde{\chi}_{j,j} )\big)\|_{L^2}\\
& \leq \|\chi_{\ell,\mu,\mu_2}(D)g(\sum\limits_{j\neq \ell} \lambda_j \tilde\chi_{j,j}) - g(\sum\limits_{j\neq \ell} \lambda_j \chi_{j,\mu,\mu_2}(D) \tilde{\chi}_{j,j} )\|_{L^2}\\
& \leq C \sum\limits_{j\neq \ell}  \psi_{d}\big(|\lambda_j|\big)  \psi_{1-\gamma}\big(\|\tilde\chi_{j,j}-\chi_{j,\mu,\mu_2}(D)\tilde\chi_{j,j}\|_{L^2}\big)
 +  \|\tilde\chi_{\ell, \ell}-\chi_{\ell,\mu,\mu_2}(D)\tilde \chi_{\ell, \ell}\|_{L^2}.
\end{align*}
In particular, by virtue of the pointwise bound $|\chi_{\ell,\mu,\mu_2}-\chi_{\ell,\mu,\mu_3}|\leq \chi_{\ell,\mu\mu_2} \chi_{\ell,\mu,\mu_2}$, we obtain
\begin{align*}
&\|\chi_{\ell,\mu,\mu_2}(D)g(\sum\limits_{j\neq \ell} \lambda_j \tilde\chi_{j,j}) - \chi_{\ell,\mu,\mu_3}(D)g(\sum\limits_{j\neq \ell} \lambda_j \tilde\chi_{j,j})\|_{L^2} \\
& \leq \|\chi_{\ell,\mu \mu_2}(D) \chi_{\ell,\mu,\mu_2}(D)g(\sum\limits_{j\neq \ell} \lambda_j \tilde\chi_{j,j})\|_{L^2}\\
& \leq  C \sum\limits_{j\neq \ell}  \psi_{d}\big(|\lambda_j|\big)  \psi_{1-\gamma}\big(\|\tilde\chi_{j,j}-\chi_{j,\mu,\mu_2}(D)\tilde\chi_{j,j}\|_{L^2}\big)
 +  \|\tilde\chi_{\ell, \ell}-\chi_{\ell,\mu,\mu_2}(D)\tilde \chi_{\ell, \ell}\|_{L^2} .
\end{align*}
As a consequence, by the triangle and by the bounds from above
\begin{align*}
&\|g(\sum\limits_{j\neq \ell} \lambda_j \tilde\chi_{j,j}) - \chi_{\ell,\mu,\mu_3}(D)g(\sum\limits_{j\neq \ell} \lambda_j \tilde\chi_{j,j})\|_{L^2}\\
&\leq \|g(\sum\limits_{j\neq \ell} \lambda_j \tilde\chi_{j,j})- \chi_{\ell,\mu,\mu_2}(D)g(\sum\limits_{j\neq \ell} \lambda_j \tilde\chi_{j,j})\|_{L^2}\\
& \quad  + \|\chi_{\ell,\mu,\mu_2}(D)g(\sum\limits_{j\neq \ell} \lambda_j \tilde\chi_{j,j}) - \chi_{\ell,\mu,\mu_3}(D)g(\sum\limits_{j\neq \ell} \lambda_j \tilde\chi_{j,j})\|_{L^2}\\
& \leq C \sum\limits_{j\neq \ell} \psi_{d}\big(|\lambda_j|\big)  \psi_{1-\gamma}\big(\|\tilde\chi_{j,j}-\chi_{j,\mu,\mu_2}(D)\tilde\chi_{j,j}\|_{L^2}\big)
 +  \|\tilde\chi_{\ell, \ell}-\chi_{\ell,\mu,\mu_2}(D)\tilde \chi_{\ell, \ell}\|_{L^2} . 
\end{align*}
Here, in the last line, we used that $\tilde{\chi}_{\ell, \ell} = g\left( \sum\limits_{j\neq \ell} \lambda_j \tilde\chi_{j,j} \right)$.
\end{proof}

For technical reasons we will need a slight modification of the previous result that lets us deal with ``non-symmetric'' frequency localization on cones.

\begin{cor}\label{cor:cone-red-nonl}
Let  $\lambda_j \in \R$ for $j\in \{1,\dots,n\}$ and assume that for some $\ell\in \{1,\dots,n\}$ it holds that $\tilde{\chi}_{\ell, \ell} = g\left( \sum\limits_{j\neq \ell} \lambda_j \tilde\chi_{j,j} \right)$, where $g$ is a polynomial of degree $d\in \N\cup \{0\}$.
Let $\mu_3:=M\mu \mu_2$ and $\mu_4:=M'\mu \mu_3$ for some $M,M'>0$ depending on $d$ such that $0<\mu_4<\mu_3<\mu_2$.
Then, for any $\gamma \in (0,1)$ there exists a constant $C>0$ (depending on $g,\gamma, F$) such that
\begin{align*}
& \|\tilde\chi_{\ell, \ell}- \chi_{\ell,\mu,\mu_4}(D)\tilde \chi_{\ell, \ell}\|_{L^2}\\
& \leq C \sum\limits_{j\neq \ell}  \psi_{d}\big(|\lambda_j|\big)  \psi_{1-\gamma}\big(\|\tilde\chi_{j,j}-\chi_{j,\mu,\mu_3}(D)\tilde\chi_{j,j}\|_{L^2}\big)
 +  \|\tilde\chi_{\ell, \ell}-\chi_{\ell,\mu,\mu_3}(D)\tilde \chi_{\ell, \ell}\|_{L^2}.
\end{align*}
\end{cor}

\begin{proof}
Working analogously as in the proof of Proposition \ref{prop:cone-red-nonl} we have that
$$
\max\limits_{k\in C_{j,2\mu,2\mu_3}}|k_{\ell}| \leq 4\mu_3 \mu, \ j \neq \ell,
$$
and since $g$ is a polynomial we obtain
\begin{align*}
\F ( g(\sum\limits_{j \neq \ell}\lambda_j \chi_{j,\mu,\mu_3}(D) \tilde{\chi}_{j,j}))(k) = 0 \mbox{ for } |k_{\ell}|\geq M' \mu \mu_3.
\end{align*}
Setting $\chi_{\ell,\mu \mu_3}$ to be the characteristic function of the set $\{k\in \R^n: \ |k_{\ell}|> M' \mu \mu_3\}$, from equation \eqref{eq:aux_aim} (exploited with $\mu_3$ in place of $\mu_2$) we infer that
\begin{align*}
&\|\chi_{\ell,\mu \mu_3}(D) \chi_{\ell,\mu,\mu_2}(D)g(\sum\limits_{j\neq \ell} \lambda_j \tilde\chi_{j,j})\|_{L^2}\\
&= \|\chi_{\ell,\mu \mu_3}(D) \big(\chi_{\ell,\mu,\mu_2}(D)g(\sum\limits_{j\neq \ell} \lambda_j \tilde\chi_{j,j}) - g(\sum\limits_{j\neq \ell} \lambda_j \chi_{j,\mu,\mu_3}(D) \tilde{\chi}_{j,j} )\big)\|_{L^2}\\
& \leq \|\chi_{\ell,\mu,\mu_2}(D)g(\sum\limits_{j\neq \ell} \lambda_j \tilde\chi_{j,j}) - g(\sum\limits_{j\neq \ell} \lambda_j \chi_{j,\mu,\mu_3}(D) \tilde{\chi}_{j,j} )\|_{L^2}\\
& \leq \|\chi_{\ell,\mu,\mu_2}(D)\tilde\chi_{\ell,\ell}-\tilde\chi_{\ell,\ell}\|_{L^2}+\|g(\sum\limits_{j\neq \ell} \lambda_j \tilde{\chi}_{j,j} )- g(\sum\limits_{j\neq \ell} \lambda_j \chi_{j,\mu,\mu_3}(D) \tilde{\chi}_{j,j} )\|_{L^2}\\
& \leq C \sum\limits_{j\neq \ell} \psi_{d}\big(|\lambda_j|\big)  \psi_{1-\gamma}\big(\|\tilde\chi_{j,j}-\chi_{j,\mu,\mu_3}(D)\tilde\chi_{j,j}\|_{L^2}\big)
 +  \|\tilde\chi_{\ell, \ell}-\chi_{\ell,\mu,\mu_3}(D)\tilde \chi_{\ell, \ell}\|_{L^2}
\end{align*}
Using the bound $|\chi_{\ell,\mu,\mu_2}-\chi_{\ell,\mu,\mu_4}|\leq \chi_{\ell,\mu\mu_3} \chi_{\ell,\mu,\mu_2}$ and the triangle inequality as done at the end of the proof of Proposition \ref{prop:cone-red-nonl} we obtain the claim.
\end{proof}

\section{Three Wells: Proof of Theorem \ref{thm:K3}}
\label{sec:3well_lower}
In this section, we study a model problem involving three wells without gauges which gives rise to laminates of order up to two. As the main result of this section, we provide the proof of Theorem \ref{thm:K3} illustrating that the order of lamination of the displacement boundary data determines the energy scaling of the problem. We split the proof into two parts: First, in Section \ref{sec:3up} we discuss the upper bound construction (which essentially follows \cite{CC15} or \cite{KW16}) and then combine the ideas from \cite{RT21} with the ones from \cite{KW16} to deduce the (essentially matching) lower bound scaling behaviour in Section \ref{sec:3low2}.
Finally, in Section \ref{subsec:periodic3grad} we illustrate the difference between the periodic and Dirichlet settings by proving scaling bounds in the periodic setting.

\subsection{Upper bound}
\label{sec:3up}
The upper bound in Theorem \ref{thm:K3} (i) directly follows from Proposition \ref{prop:first-order} with $A=A_2, B=A_3$ and $p=2$.

To obtain the bound of Theorem \ref{thm:K3} (ii) we perform a second-order branching construction.
We therefore consider the auxiliary matrix 

\begin{align}
\label{eq:aux_matr}
J_1:=\frac{1}{2}A_2+\frac{1}{2}A_3=-A_1
\end{align}
which is rank-1-connected to all the stress-free states (see Figure \ref{fig:K3}).
We remark that in this case $F\in K_3^{(lc)}\setminus K_3^{(1)}$, that is a second-order laminate of the set $K_3$, can be written as $F=\lambda A_1+(1-\lambda)J_1$ for some $\lambda\in(0,1)$.
\begin{proof}[Proof of the upper bound from Theorem \ref{thm:K3} (ii)]
Due to translation invariance we assume $b=0$.
We work in several steps for the sake of clarity of exposition.

\emph{Step 1.}
Consider $u^{(1)}\in W^{1,\infty}((0,1)^2;\R^2)$ defined by Lemma \ref{prop:vert-bra} with $A=J_1$, $B=A_1$, $p=2$, $H=L=1$ and $N=\frac{4}{r}$, where $0<r<1$ is an arbitrarily small length scale.
We set $\chi^{(1)}$ to be the projection of $\nabla u^{(1)}$ onto $\{A_1,J_1\}$.
Note that such a projection function is well-defined almost everywhere by construction of $u^{(1)}$.

\emph{Step 2.}
We will define $u^{(2)}$ by replacing $u^{(1)}$ with a finely (branched) oscillation between the states $A_2$ and $A_3$ on $\{\chi^{(1)}=J_1\}$ (outside of the cut-off region) with boundary datum $u^{(1)}$.
We make this substitution inside each cell $\omega_{j,k}$ with $j\in\{0,\dots,j_0\}$ as defined in 
\eqref{eq:cells-bra}.
Here we have $\{\chi^{(1)}=J_1\}\cap\omega_{j,k}=\omega_{j,k}^{(1)}\cup\omega_{j,k}^{(3)}$ with
\begin{equation}\label{eq:cells-2order}
\begin{split}
\omega_{j,k}^{(1)} &:=\Big[k\ell_j,k\ell_j+\frac{\lambda\ell_j}{2}\Big]\times[y_j,y_{j+1}] \\
\omega_{j,k}^{(3)} &:=\Big\{\frac{\lambda\ell_j}{2}\le x_1-\frac{(1-\lambda)\ell_j(x_2-y_j)}{2h_j}-k\ell_j\le\lambda\ell_j,\, x_2\in[y_j,y_{j+1}]\Big\}.
\end{split}
\end{equation}
We recall that $\ell_j=\frac{r}{2^j}$, $y_j=1-\frac{\theta^j}{2}$ and $h_j=y_{j+1}-y_j=\frac{1-\theta}{2}\theta^j$ as defined in \eqref{eq:dim-i}.
Notice that the sets above correspond to $\omega_1$ and $\omega_3$ in Figure \ref{fig:cell2}. 

\begin{figure}[t]
\begin{center}
\includegraphics{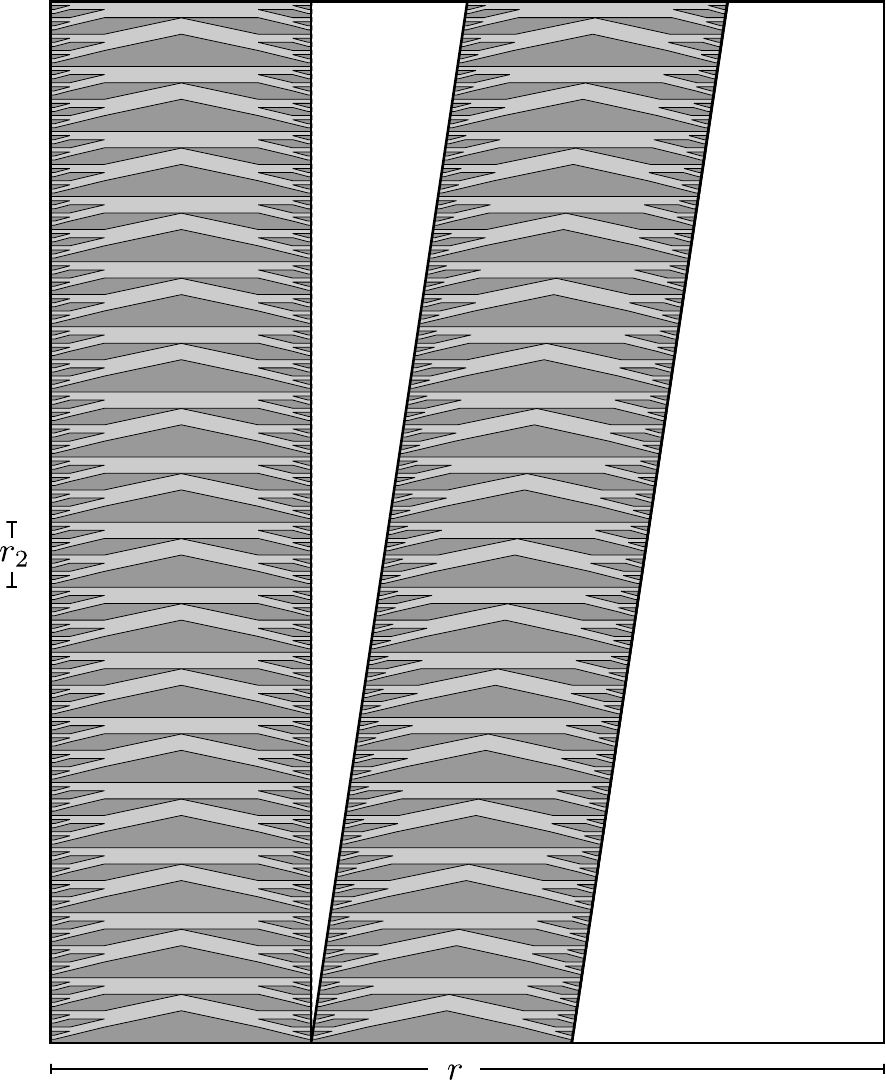}
\end{center}
\caption{An example of the second order laminate construction $u^{(2)}$ used in the proof of the upper bound for Theorem \ref{thm:K3}(ii) in a cell of generation zero, e.g. $\omega_{0,0}$.
Each color corresponds to a well of $K_3$.}
\label{fig:cell-second}
\end{figure}

\emph{Step 3.}
We fix a second small length scale $0<r_2<\frac{r}{4}$.
We apply Lemma \ref{prop:vert-bra} (with switched roles between $x_1$ and $x_2$) on $\omega_{j,k}^{(1)}$ with $F=J_1$, $A=A_2$, $B=A_3$ and $N=N_j:=\frac{(2\theta)^j}{\lambda r_2}$.
Notice that the hypothesis $N>\frac{4L}{H}$ in Lemma \ref{prop:vert-bra} is satisfied for such choices.
Thus, we can find a function $v_{j,k}^{(1)}\in W^{1,\infty}(\omega_{j,k}^{(1)};\R^2)$ such that $v_{j,k}^{(1)}=u^{(1)}=J_1x+b$ on $\p\omega_{j,k}^{(1)}$ and
\begin{equation}\label{eq:2nd-ord-cell}
\int_{\omega_{j,k}^{(1)}}\dist^2(\nabla v_{j,k}^{(1)},K_3)dx+\|D^2v_{j,k}^{(1)}\|_{TV(\omega_{j,k}^{(1)})}+Per(\omega_{j,k}^{(1)})\lesssim \frac{h_j^3}{N^2\ell_j}+\epsilon\ell_j N+\epsilon h_j.
\end{equation}
Here the quantities $\ell_j, h_j$ refer to the ones from Step 2 which determine the length scales in the inner branching construction, see Figure \ref{fig:cell-second}.
The construction inside $\omega_{j,k}^{(3)}$ is obtained from the one above, thanks to an affine change of variables, i.e.
\begin{equation}\label{eq:shear2D}
v_{j,k}^{(3)}(x_1,x_2)=v_{j,k}^{(1)}\Big(x_1-\frac{(1-\lambda)\ell_j(x_2-y_j)}{2h_j}-\frac{\lambda\ell_j}{2},x_2\Big)+\Big(-\frac{\lambda\ell_j}{2}+\frac{(1-\lambda)\ell_j (x_2-y_j)}{2h_j},0\Big).
\end{equation}
We thus get
\begin{equation}\label{eq:2nd-ord-cell2}
\begin{split}
\int_{\omega_{j,k}^{(3)}}\dist^2(\nabla v_{j,k}^{(3)},K_3)dx &=\int_{\omega_{j,k}^{(1)}}\dist^2\Big(\nabla v_{j,k}^{(1)}\begin{pmatrix}1&\frac{(\lambda-1)\ell_j}{2h_j}\\0&1\end{pmatrix}+\begin{pmatrix}0&\frac{(1-\lambda)\ell_j}{2h_j}\\0&0\end{pmatrix},K_3\Big)dx \\
&\lesssim \int_{\omega_{j,k}^{(1)}}\dist^2(\nabla v_{j,k}^{(1)},K_3)dx+\frac{\ell_j^3}{h_j}.
\end{split}
\end{equation}
We have now everything in place to define the function $u^{(2)}\in W^{1,\infty}((0,1)^2;\R^2)$ to be equal to $v_{j,k}^{(1)}(x)$ and $v_{j,k}^{(3)}(x)$ on $\omega_{j,k}^{(1)}$ and $\omega_{j,k}^{(3)}$ respectively and to $u^{(1)}$ otherwise.
We recall that in $\omega_{j,k}':=\omega_{j,k}\setminus(\omega_{j,k}^{(1)}\cup\omega_{j,k}^{(3)})$ only surface energy is present and thus
$$
\int_{\omega_{j,k}'}\dist^2(\nabla u^{(1)},K_3)dx+\|D^2u^{(1)}\|_{TV(\omega_{j,k}')}+Per(\omega_{j,k}')\lesssim\epsilon h_j.
$$
Combining the inequality above with \eqref{eq:2nd-ord-cell} and \eqref{eq:2nd-ord-cell2}, for every $j\in\{0,\dots,j_0\}$ and $k \in \{0,\dots,\frac{2^j}{r}-1\}$, we have that
\begin{equation*}\label{eq:2nd-ord-cell3}
\int_{\omega_{j,k}}\dist^2(\nabla u^{(2)},K_3)dx+\|D^2u^{(2)}\|_{TV(\omega_{j,k})}+Per(\omega_{j,k})
\lesssim \frac{h_j^3}{N^2\ell_j}+\epsilon\ell_j N+\frac{\ell_j^3}{h_j}+\epsilon h_j.
\end{equation*}
Eventually, summing over $j$ and $k$, and controlling the $(j_0+1)$-th term as in the proof of Lemma \ref{prop:vert-bra}, we get
\begin{equation}\label{eq:2D-2order}
\begin{split}
E_{\epsilon,3}(u^{(2)},\chi^{(2)}) &\lesssim \sum_{j=0}^{j_0+1} \frac{2^j}{r}\Big(\frac{\theta^j}{2^j}\frac{r_2^2}{r}+\Big(\frac{1}{8\theta}\Big)^j r^3+\epsilon\theta^j\frac{r}{r_2}\Big)\\
&\lesssim \sum_{j=0}^\infty \Big(\theta^j\Big(\frac{r_2}{r}\Big)^2+\Big(\frac{1}{4\theta}\Big)^j r^2+(2\theta)^j\frac{\epsilon}{r_2}\Big) \lesssim \Big(\frac{r_2}{r}\Big)^2+r^2+\frac{\epsilon}{r_2}
\end{split}
\end{equation}
for some $\frac{1}{4}<\theta<\frac{1}{2}$, where $\chi^{(2)}$ is the projection of $\nabla u^{(2)}$ on $K_3$.
In the first inequality in \eqref{eq:2D-2order} we have also used that $\|D\chi^{(2)}\|_{TV} \leq C \|D^2 u^{(2)}\|_{TV}$.
Optimizing the expression above we get $r_2\sim r^2$ and thus $r\sim\epsilon^\frac{1}{4}$.
We can therefore define $u_\epsilon\in W^{1,\infty}((0,1)^2;\R^2)$ with $u_\epsilon(x)=Fx$ for every $x\in\p[0,1]^2$ and 
\begin{equation*}
E_{\epsilon,3}(u_\epsilon,\chi_\epsilon)\lesssim \epsilon^\frac{1}{2},
\end{equation*}
where $u_\epsilon=u^{(2)}$ and $\chi_\epsilon=\chi^{(2)}$.
\end{proof}

\begin{rmk}
The choice $N=\frac{(2\theta)^j}{\lambda r_2}$ in Step 2 is motivated by the fact that it preserves the ``self-similarity'' of the construction.
This seems the most natural one among all the choices that keep the summability of the quantity in \eqref{eq:2D-2order}, which is possible. 
\end{rmk}

We also refer to Appendix \ref{sec:second-order} for two-dimensional constructions with general $p\in [1,\infty)$ and an arbitrary number of wells.

\subsection{Proof of the lower bounds}
\label{sec:3low2}

The proof of the lower bounds in Theorem \ref{thm:K3} in part mimics the argument from the stress-free case. We thus begin by recalling the rigidity proof in this setting.

\subsubsection{Motivation for the lower bound: The stress-free case}
\label{sec:3low1}

Before turning to the proof of the lower bounds in Theorem \ref{thm:K3}, we give an argument for the proof of the corresponding stress-free rigidity result. We will mimic parts of this in our proof of the lower bounds in Theorem \ref{thm:K3}.

In the stress-free case, we observe that the problem reads
\begin{align}
\label{eq:stress_free3}
\nabla u = \begin{pmatrix} \chi_1 - \chi_2 -\chi_3 & 0 \\ 0 & -\chi_2 + \chi_3 \end{pmatrix} \mbox{ a.e. in } \Omega.
\end{align}

We claim the following rigidity result for this problem:

\begin{prop}
\label{prop:stress_free_rigid}
Let $u\in W^{1,\infty}(\Omega, \R^2)$ be a solution to \eqref{eq:stress_free3}. Then, the following dichotomy holds:
\begin{itemize}
\item Either $\nabla u = A_1$ or $\nabla u \in \{A_2, A_3\}$ a.e. in $\Omega$.
\item If $\nabla u \in \{A_2, A_3\}$ a.e. in $\Omega$, then $u$ is (locally) a simple laminate.
\end{itemize}
\end{prop}

\begin{proof}
The differential inclusion \eqref{eq:stress_free3} implies that $\p_2 u_1 = 0$ and $\p_1 u_2 = 0$ a.e. in $\Omega$. Hence, there exist functions $f_1$ and $f_2$ of a single variable each, such that
$$
\p_1 u_1=f_1(x_1)\in\{\pm1\}, \quad \p_2 u_2=f_2(x_2)\in\{0,\pm1\}.
$$
Now, since $\chi_{1,1}=1-2\chi_{2,2}^2$, we obtain that $f_1=1-2f_2^2$. Hence, due to the different $x_1$ and $x_2$ dependences, $f_1$ is constant.
Thus,
\begin{align*}
&\text{if } f_1=1 \quad\Rightarrow\quad f_2\equiv 0;\\
&\text{if } f_1=-1 \quad\Rightarrow\quad f_2\in\{\pm1\}.
\end{align*}
This concludes the proof.
\end{proof}

\subsubsection{Preliminaries for the lower bound}

We seek to deduce the lower bounds by mimicking the argument from the stress-free setting. However, we caution that the dichotomy arising from the differential inclusion yields important \emph{structural} information but \emph{not} directly the \emph{full} information. Only in combination with the Dirichlet boundary data, the desired scaling behaviour is obtained. This matches the observation that there is a \emph{different} scaling for the problem with imposed Dirichlet data and the periodic problem (with prescribed mean) if one considers data from the second lamination convex hull, see Section \ref{subsec:periodic3grad}.

As indicated our proof combines
\begin{itemize}
\item the Fourier multiplier ideas from \cite{RT21} (see also \cite{CO1,CO,KKO13}) in order to deduce structure on the Fourier support of the characteristic functions $\chi_j$, $j\in\{1,2,3\}$,
\item the bound from Lemma \ref{lem:axes_estimates} used to bound the correspondingly matching range of low frequencies.
\end{itemize}

We begin by recalling the notation from Section \ref{sec:prelim_lower}.
In the three-well setting from Theorem \ref{thm:K3}, the diagonal components of $\chi$ read
\begin{align*}
\chi_{1,1}:= \chi_1-\chi_2-\chi_3, \quad \chi_{2,2}:=-\chi_2+\chi_3
\end{align*}
and, denoting by $F$ the affine boundary datum, we write
$$
\tilde\chi_{1,1}=\chi_{1,1}-F_{1,1},
\quad
\tilde\chi_{2,2}=\chi_{2,2}-F_{2,2}.
$$
Moreover, the conical domains defined in \eqref{eq:cones-def} now are given by:
\begin{align*}
C_{1,\mu, \mu_2}:= \{k \in \Z^2: \ |k_2| \leq \mu|k|, \ |k|\leq \mu_2\}, \ C_{2,\mu,\mu_2}:= \{k\in \Z^2: \ |k_1| \leq \mu|k|, \ |k|\leq \mu_2\}.
\end{align*}
As above, to these cones we associate Fourier multipliers $\chi_{j,\mu,\mu_2}(D)$ which are determined as in \eqref{eq:mult} by $C^{\infty}(\R^2)$ functions, (essentially, in the same sense as above) zero-ho\-mo\-ge\-ne\-ous for $|k|\leq \mu_2$, non-negative, such that $\chi_{j,\mu,\mu_2}(k) = 1$ for $k\in C_{j,\mu,\mu_2}$ and  $\chi_{j,\mu,\mu_2}(k) = 0$ for $k\in \Z^2\setminus C_{j,2\mu,2\mu_2}$.

In the following, we will always presuppose the first localization result in the frequency space which is provided by Lemma \ref{lem:first_loc} and which in our three-well setting reads
\begin{equation}\label{eq:3grad1}
\|\tilde\chi_{1,1}-\chi_{1,\mu,\mu_2}(D)\tilde\chi_{1,1}\|_{L^2}^2+\|\tilde\chi_{2,2}-\chi_{2,\mu,\mu_2}(D)\tilde\chi_{2,2}\|_{L^2}^2 \lesssim \epsilon^{-2\alpha} (E_{\epsilon,3}(\chi,F) + \epsilon\Per(\Omega)),
\end{equation}
where $E_{\epsilon,3}(\chi,F):= \epsilon E_{surf}(\chi) + E_{el,3}(\chi,F)$ and where $\mu \sim \epsilon^{\alpha}$ and $\mu_2 \sim \epsilon^{2\alpha -1}$ (the choice of the prefactors will be made in the following steps).
In the sequel, we combine this localization with further reduction steps originating from the choice of the boundary data (and their order of lamination) and further nonlinear dependences.
We deal with the case of first order laminates and second order laminates separately, although the case of first order laminates (with affine boundary data) could essentially be reduced to the discussion in Section \ref{sec:Lp} with $p=2$.

\subsubsection{One order of lamination}

As an illustration of the tools and ideas which we are employing in this section in a simplified setting, we first present a Fourier-based proof of the lower bounds for boundary data $F$ involving first order laminates, i.e. of the lower bounds in Theorem \ref{thm:K3} (i).
We recall that the elements in $K^{(1)}_{3}\setminus K_{3}$ are of the form
$$
F=\begin{pmatrix}-1&0\\0&\mu\end{pmatrix}
$$
where $|\mu|<1$.

\begin{proof}[Proof of the lower bound in Theorem \ref{thm:K3} (i)]
In order to apply the low-frequency control of Lemma \ref{lem:axes_estimates}, we choose $\mu=\epsilon^
\alpha$, $\mu_2=\frac{1}{8}\epsilon^{2\alpha-1}$ and $\alpha=\frac{1}{3}$.
With these choices of parameters the truncated cones $C_{j,\mu,\mu_2}$, that are of thickness $8\mu\mu_2=1$, reduce to be one dimensional, namely $C_{j,2\mu,2\mu_2}=\{k_j e_j : |k_j|\le\frac{1}{4}\epsilon^{-\frac{1}{3}}\}$.
Lemma \ref{lem:axes_estimates} thus yields
\begin{equation}\label{eq:3grad2}
\sum_{|k_1|\le\frac{1}{4}\epsilon^{-\frac{1}{3}}}|\F\tilde\chi_{1,1}(k_1,0)|^2+\sum_{|k_2|\le\frac{1}{4}\epsilon^{-\frac{1}{3}}}|\F\tilde\chi_{2,2}(0,k_2)|^2\lesssim\epsilon^{-\frac{2}{3}}E_{el,3}(\chi,F).
\end{equation}

Hence \eqref{eq:3grad1} and \eqref{eq:3grad2} give

\begin{align}\label{eq:3grad3}
\begin{split}
\|\chi_{1,1}+1\|_{L^2}^2+\|\chi_{2,2}-\mu\|_{L^2}^2
& =|\F\tilde{\chi}_{1,1}(0,\dots,0)|^2 + |\F\tilde\chi_{2,2}(0,\dots,0)|^2\\
&\leq 2 \|\tilde\chi_{1,1}- \chi_{1,\mu,\mu_2}(D) \tilde{\chi}_{1,1}\|_{L^2}^2 + 2 \|\tilde\chi_{1,1}- \chi_{1,\mu,\mu_2}(D) \tilde{\chi}_{1,1}\|_{L^2}^2\\
& \quad + 2\sum_{|k_1|\le\frac{1}{4}\epsilon^{-\frac{1}{3}}}|\F\tilde\chi_{1,1}(k_1,0)|^2+2\sum_{|k_2|\le\frac{1}{4}\epsilon^{-\frac{1}{3}}}|\F\tilde\chi_{2,2}(0,k_2)|^2 \\
& \lesssim \epsilon^{-\frac{2}{3}} E_{el,3}(\chi,F) +\epsilon^{-\frac{2}{3}}(E_{\epsilon,3}(\chi,F) + \epsilon \Per(\Omega) ).
\end{split}
\end{align}
Exploiting the left-hand-side above, we obtain
\begin{align*}
\|\chi_{1,1}+1\|_{L^2}^2+\|\chi_{2,2}-\mu\|_{L^2}^2 
\ge\min\{(1-\mu)^2,(1+\mu)^2\}=\dist^2(F,K).
\end{align*}
Since for $F \in K^{(1)}\setminus K$ it holds that $|\mu|<1$,  we arrive at
$$
0<C(F)\leq \epsilon^{-\frac{2}{3}}(E_{\epsilon,3}(\chi,F) + \epsilon \Per(\Omega) ),
$$
which is the desired inequality after absorbing the perimeter contribution $\epsilon \Per(\Omega) $ into the left-hand-side of the above inequality by considering $\epsilon \in (0,\epsilon_0)$ with $\epsilon_0 = \epsilon_0(F,n)>0$ sufficiently small.
\end{proof}

\subsubsection{Two orders of lamination}
Next we consider the setting from Theorem \ref{thm:K3} (ii) and thus focus on boundary data $F$ of the form
$$
F=\begin{pmatrix}\mu&0\\0&0\end{pmatrix}
$$
with $|\mu|<1$.
For such $F$ we have $\tilde\chi_{1,1}=\chi_{1,1}-\mu$ and $\tilde\chi_{2,2}=\chi_{2,2}$.

In this setting, we combine one iteration of the nonlinear bootstrap argument from \cite{RT21} with the boundary data argument from Lemma \ref{lem:axes_estimates}. 
We make use of the analysis performed in Section \ref{sec:prelim_lower}.

\begin{proof}[Proof of the lower bound in Theorem \ref{thm:K3} (ii)]
As in the proof of Proposition \ref{prop:stress_free_rigid}, the crucial remark is that
\begin{equation}
\tilde\chi_{1,1}=1-\mu-2\tilde\chi_{2,2}^2=:g(\tilde\chi_{2,2}),
\end{equation}
that is the second (diagonal) component of $\tilde\chi$ determines the first one through a (nonlinear) polynomial relation $g$.
We can therefore improve the arguments used in the proof of Theorem \ref{thm:K3} (i) by reducing the conical Fourier multiplier $C_{2,2\mu,2\mu_2}$ on which $\F\tilde\chi_{1,1}$ concentrates its $L^2$ mass.
This is achieved by virtue of Proposition \ref{prop:cone-red-nonl}: for any $\gamma\in(0,1)$ there exists a constant $C>0$ depending on $\gamma$ and $g$ and a constant $M>0$ depending on the degree of $g$ such that
\begin{equation}\label{eq:3grad-cone-red}
\|\tilde\chi_{1,1}-\chi_{1,\mu,\mu_3}(D)\tilde\chi_{1,1}\|_{L^2}^2 \le \|\tilde\chi_{1,1}-\chi_{1,\mu,\mu_2}(D)\chi_{1,1}\|_{L^2}^2+ C\psi_{1-\gamma}\big(\|\tilde\chi_{2,2}-\chi_{2,\mu,\mu_2}(D)\chi_{2,2}\|_{L^2}^2\big),
\end{equation}
where $\mu_3=M\mu\mu_2$.
Here we take $\mu=\frac{1}{2}\epsilon^\alpha$ and $\mu_2=\frac{1}{4M}\epsilon^{2\alpha-1}$.
Collecting \eqref{eq:3grad1} and \eqref{eq:3grad-cone-red}, we obtain
\begin{align}\label{eq:3grad-cone-red2}
\|\tilde\chi_{1,1}-\chi_{1,\mu,\mu_3}(D)\tilde\chi_{1,1}\|_{L^2}^2 \le C \psi_{1-\gamma}\big(\epsilon^{-2\alpha}(E_{\epsilon,3}(\chi,F) + \epsilon \Per(\Omega) )\big).
\end{align}
We control the low-frequencies thanks to Lemma \ref{lem:axes_estimates}, that is
$$
\sum_{|k_1|\le 2M\epsilon^{-\alpha}}|\F\tilde\chi_{1,1}(k_1,0)|^2\lesssim \epsilon^{-2\alpha}E_{el,3}(\chi,F),
$$
which combined with \eqref{eq:3grad-cone-red2}, choosing $\alpha=\frac{1}{4}$ so that $8\mu\mu_3=1$ and thus $C_{1,2\mu,2\mu_3}=\{(k_1,0) : |k_1|\le2\mu_3=2M\epsilon^{-\frac{1}{4}}\}$, gives

\begin{align*}
\|\tilde\chi_{1,1}\|_{L^2}^2 
& \leq 2 \|\chi_{1,\mu,\mu_2}(D)\tilde{\chi_{1,1}}\|_{L^2}^2 + 2 \|\chi_{1,1}-\chi_{1,\mu,\mu_2}(D)\tilde{\chi_{1,1}}\|_{L^2}^2\\
&\le C \psi_{1-\gamma}\big(\epsilon^{-2\alpha}(E_{\epsilon,3}(\chi,F) + \epsilon \Per(\Omega) )\big)+C\epsilon^{-2\alpha}E_{el,3}(u,\chi).
\end{align*}
Eventually, inserting the definition of $\tilde{\chi}_{1,1}$ and recalling the definition of $J_1$ from \eqref{eq:aux_matr}, we infer
\begin{align*}
\dist^2(F,\{A_1,J_1\}) &=\min\{(1-\mu)^2,(1+\mu)^2\}\le\|\chi_{1,1}-\mu\|_{L^2}^2 = \|\tilde{\chi}_{1,1}\|_{L^2}^2 \\
&\le   C \psi_{1-\gamma}(\epsilon^{-\frac{1}{2}}(E_{\epsilon,3}(\chi,F) + \epsilon \Per(\Omega))) .
\end{align*}

Next, we distinguish two cases: If $\epsilon^{-\frac{1}{2}}(E_{\epsilon,3}(\chi,F) + \epsilon \Per(\Omega)))\geq 1$, then the last inequality turns into
\begin{align*}
\dist^2(F,\{A_1,J_1\}) &=\min\{(1-\mu)^2,(1+\mu)^2\}\le\|\chi_{1,1}-\mu\|_{L^2}^2 = \|\tilde{\chi}_{1,1}\|_{L^2}^2 \\
&\leq   C (\epsilon^{-\frac{1}{2}}(E_{\epsilon,3}(\chi,F) + \epsilon \Per(\Omega) ) .
\end{align*}
If on the other hand, $\epsilon^{-\frac{1}{2}}(E_{\epsilon,3}(\chi,F) + \epsilon \Per(\Omega)))< 1$, then, we obtain
\begin{align*}
\dist^{\frac{2}{1-\gamma}}(F,\{A_1,J_1\}) &=\min\{(1-\mu)^2,(1+\mu)^2\}\le\|\chi_{1,1}-\mu\|_{L^2}^2 = \|\tilde{\chi}_{1,1}\|_{L^2}^2 \\
&\le   C (\epsilon^{-\frac{1}{2}}(E_{\epsilon,3}(\chi,F) + \epsilon \Per(\Omega)) .
\end{align*}

Choosing $\epsilon \in (0,\epsilon_0)$ for $\epsilon_0=\epsilon_0(F,n,\gamma)>0$ small enough and absorbing the perimeter contribution into the left-hand-side, then concludes the proof of Theorem \ref{thm:K3}(ii).
\end{proof}

\subsection{The periodic setting}
\label{subsec:periodic3grad}

The situation is different if we allow for a weaker constraint on $\chi$, such as an imposed mean value in the periodic setting.
Specifically, consider $F\in K^{(2)}_{3}\setminus K^{(1)}_{3}$.
In the Dirichlet case minimizers involve two orders of laminations:
\begin{itemize}
\item (finest) inside the domain near the interface between incompatible states, i.e. $\chi_1$ and a mixture of $\chi_2$ and $\chi_3$;
\item (coarsest) at the boundary to attain boundary condition $F$.
\end{itemize}
Both oscillations occur at small length scales in order to minimize the energy.
In the periodic case, a mean value constraint does \emph{not} force the minimizers to match the constant gradient $F$ at the boundary. 
Therefore, fine oscillations appear only at the interface of different states and the minimal energy behaviour corresponds to that of the two-well problem with affine data which we make precise in the following results.

Let $E_{\epsilon,3}(u,\chi)$ be defined as in \eqref{eq:energyK3} and set
$$
E_\epsilon(\chi):=\inf\big\{ E_{\epsilon,3}(u,\chi) \,:\, u\in W_{loc}^{1,2}(\R^2;\R^2),\, \nabla u \mbox{ is } [0,1]^2\text{-periodic} \big\}
$$
for every $\chi_j\in BV(\T^2;\{0,1\})$ with $\chi_1+\chi_2+\chi_3\equiv1$.
Define also
$$
\theta_j=\int_{[0,1]^2}\chi_j, \quad j=1,2,3.
$$
We first show a quantification of the stress-free case from Proposition \ref{prop:stress_free_rigid} in the spirit of the articles \cite{CO,CO1,Rue16b}.

\begin{lem}\label{lem:periodic3grad}
For every $\chi$ as above and $\epsilon>0$ there holds
\begin{equation}\label{eq:lem_per}
\|\chi_{1,1}-\langle\chi_{1,1}\rangle\|_{L^2([0,1]^2)}^2+\|\chi_{2,2}-f_2\|_{L^2([0,1]^2)}^2 \lesssim \epsilon^{-\frac{2}{3}}E_\epsilon(\chi)
\end{equation}
where $f_2:[0,1]\to\{0,\pm1\}$ is a function of $x_2$ only satisfying the following dichotomy:
\begin{itemize}
\item[(i)] if $1-\theta_1\lesssim \epsilon^{-\frac{2}{3}}E_\epsilon(\chi)$, then  it is possible to choose $f_2\equiv0$;
\item[(ii)] if $\theta_1\lesssim \epsilon^{-\frac{2}{3}}E_\epsilon(\chi)$, then it is possible to choose $f_2\in\{\pm1\}$.
\end{itemize}
\end{lem}
\begin{proof}
We argue in two steps.

\emph{Step 1.}
Reasoning as in Lemma \ref{lem:Fourier} without the mean-value constraint we have
$$
\inf\big\{E_{el,3}(u,\chi) : u\in W^{1,2}_{loc}(\R^2;\R^2), \nabla u \text{ is } [0,1]^2\text{-periodic}\big\}\gtrsim \sum_{k\in\Z^2}\frac{|k_2|^2}{|k|^2}|\hat\chi(k)|^2+\frac{|k_2|^2}{|k|^2}|\hat\chi(k)|^2.
$$
Hence, we can rework Lemma \ref{lem:first_loc} to result in
\begin{equation}\label{eq:lem_per1}
\|\chi_{1,1}-\chi_{1,\mu,\mu_2}(D)\chi_{1,1}\|_{L^2}^2+\|\chi_{2,2}-\chi_{2,\mu,\mu_2}(D)\chi_{2,2}\|_{L^2}^2\lesssim \epsilon^{-2\alpha}E_\epsilon(\chi)
\end{equation}
with $\mu=\frac{1}{2}\epsilon^{\alpha}$ and $\mu_2=\frac{1}{4}\epsilon^{2\alpha-1}$ (where we note that in the periodic setting no additional perimeter contribution is needed in the high frequency bounds).
Since
$$
\sup_{k\in C_{1,2\mu,2\mu_2}}|k_{2}|=4\mu\mu_2=\frac{1}{2}\epsilon^{3\alpha-1},
\quad
\sup_{k\in C_{2,2\mu,2\mu_2}}|k_{1}|=4\mu\mu_2=\frac{1}{2}\epsilon^{3\alpha-1}
$$
with the choice $\alpha=\frac{1}{3}$ we have
$$
C_{1,2\mu,2\mu_2}\subset\Z\times\{0\},
\quad
C_{2,2\mu,2\mu_2}\subset\{0\}\times\Z.
$$
Hence, \eqref{eq:lem_per1} implies that
\begin{equation}\label{eq:lem_per1.5}
\|\chi_{1,1}-\langle\chi_{1,1}\rangle_2\|_{L^2}^2+\|\chi_{2,2}-\langle\chi_{2,2}\rangle_1\|_{L^2}^2\lesssim\epsilon^{-\frac{2}{3}}E_\epsilon(\chi).
\end{equation}

We can improve this estimate by exploiting the fact that $\chi_{1,1}=g(\chi_{2,2})$, where e.g. $g(t)=1-2t^2$. Indeed, notice that
\begin{align}
\label{eq:comm_a1}
\begin{split}
\|\langle\chi_{1,1}\rangle_2-\langle\chi_{1,1}\rangle\|_{L^2}^2 &= \|\langle\chi_{1,1}\rangle_2-\langle\langle\chi_{1,1}\rangle_2\rangle_1\|_{L^2}^2 \\
&\le \|\langle\chi_{1,1}\rangle_2-g(\langle\chi_{2,2}\rangle_1)-\langle\langle\chi_{1,1}\rangle_2-g(\langle\chi_{2,2}\rangle_1)\rangle_1\|_{L^2}^2 \\
&\lesssim \|\langle\chi_{1,1}\rangle_2-g(\langle\chi_{2,2}\rangle_1)\|_{L^2}^2,
\end{split}
\end{align}
where in the last step we have used the general fact that $\|v-\langle v\rangle\|_{L^2}^2\le4\|v\|_{L^2}^2$ (recalling that all the integrals are in $[0,1]^2$).
Then by the triangle inequality, \eqref{eq:lem_per1.5}, \eqref{eq:comm_a1} and the local Lipschitz regularity of $g$ yield
\begin{align*}
\|\chi_{1,1}-\langle\chi_{1,1}\rangle\|_{L^2}^2 &\lesssim \|\chi_{1,1}-\langle\chi_{1,1}\rangle_2\|_{L^2}^2+\|\langle\chi_{1,1}\rangle_2-\langle\chi_{1,1}\rangle\|_{L^2}^2 \\
&\lesssim \|\chi_{1,1}-\langle\chi_{1,1}\rangle_2\|_{L^2}^2+\|g(\chi_{2,2})-g(\langle\chi_{2,2}\rangle_1)\|_{L^2}^2 \lesssim \epsilon^{-\frac{2}{3}} E_\epsilon(\chi).
\end{align*}

Since for every integrable function we can find a point in which it is lower than its mean we can find $\tilde x_1,\tilde x_2\in(0,1)$ for which
\begin{align*}
\|\chi_{2,2}(\tilde x_1,\cdot)-\langle\chi_{2,2}\rangle_1\|_{L^2}^2&\le\|\chi_{2,2}-\langle\chi_{2,2}\rangle_1\|_{L^2}^2.
\end{align*}
Using the notation $f_2=\chi_{2,2}(\tilde x_1,\cdot)$ and applying the triangle inequality, we get
\begin{equation}\label{eq:lem_per2}
\|\chi_{1,1}-\langle\chi_{1,1}\rangle\|_{L^2}^2+\|\chi_{2,2}-f_2\|_{L^2}^2\lesssim\epsilon^{-\frac{2}{3}}E_\epsilon(\chi).
\end{equation}
We remark that $f_2\in\{0,\pm1\}$ and thus \eqref{eq:lem_per} is proved.

\emph{Step 2.}
It remains to prove the dichotomy on $f_2$.
On the one hand, point (i) is immediate, once we recall that $1-\theta_1=\theta_2+\theta_3=\|\chi_{2,2}\|_{L^2}^2$.
On the other hand, in order to prove (ii), assume that $\theta_1\lesssim\epsilon^{-\frac{2}{3}}E_\epsilon(\chi)$.
We show that if $|\{f_2=0\}|>0$, on this set, we can substitute $f_2$ with some non-zero $\tilde f_2$ in the estimate \eqref{eq:lem_per}.
Indeed, let
$$
\tilde f_2=\begin{cases}
f_2 & \mbox{ if } f_2\neq 0, \\
1 & \mbox{ if } f_2=0.
\end{cases}
$$
Since $\theta_1=|\{\chi_{2,2}=0\}|$, denoting with $\Omega'$ the support of $\chi_{2,2}$, we obtain
$$
\|\chi_{2,2}-\tilde f_2\|_{L^2}^2
\lesssim\|\chi_{2,2}-\tilde f_2\|_{L^2(\Omega')}^2+\|\chi_{2,2}-\tilde{f}_2\|_{L^2(\Omega\setminus  \Omega')}^2 \\
\lesssim \|\chi_{2,2}-\tilde f_2\|_{L^2(\Omega')}^2 + 
\epsilon^{-\frac{2}{3}}E_\epsilon(\chi).
$$
Here, in the last estimate, we have used that, by assumption, 
\begin{align*}
\|\chi_{2,2}-\tilde{f}_2\|_{L^2(\Omega\setminus  \Omega')}^2 = \|\tilde{f}_2\|_{L^2(\Omega\setminus  \Omega')}^2 \leq C |\Omega\setminus  \Omega'| = C \theta_1 \leq \epsilon^{-\frac{2}{3}}E_\epsilon(\chi).
\end{align*} 
Now, from the definition of $\tilde f_2$ we get
\begin{align*}
\|\chi_{2,2}-\tilde f_2\|_{L^2(\Omega')}^2 &=\|\chi_{2,2}-\tilde f_2\|_{L^2(\Omega'\cap\{f_2\neq0\})}^2+\|\chi_{2,2}-\tilde f_2\|_{L^2(\Omega'\cap\{f_2=0\})}^2 \\
&=\|\chi_{2,2}-f_2\|_{L^2(\Omega'\cap\{f_2\neq0\})}^2+\|\chi_{2,2}-1\|_{L^2(\Omega'\cap\{f_2=0\})}^2\le4\|\chi_{2,2}-f_2\|_{L^2(\Omega')}^2.
\end{align*}
Here, in order to bound the last contribution, we have used that 
\begin{align*}
\|\chi_{2,2}-1\|_{L^2(\Omega'\cap\{f_2=0\})} \leq 4 |\Omega'\cap\{f_2=0\}| = 4 \|\chi_{2,2} - f_2\|_{L^2(\Omega'\cap\{f_2=0\})}.
\end{align*}
Combining the two inequalities above the result is proven.
\end{proof}

\begin{rmk}
Lemma \ref{lem:periodic3grad} quantifies the stress-free case in the sense that if $E_\epsilon=o(\epsilon^\frac{2}{3})$, let $u_\epsilon$ be a minimizing sequence, then we have either $\|\nabla u_\epsilon-A_1\|_{L^2}=o_\epsilon(1)$ or
$$
\Big\|\nabla u_\epsilon-\begin{pmatrix}-1&0\\0&f_2\end{pmatrix}\Big\|_{L^2}=o_\epsilon(1).
$$
\end{rmk}

We now have the tools to prove the following energy scaling behaviour result under mean value constraint.

\begin{prop}
Let $F\in K^{(2)}_{3}\setminus K_3^{(1)}$ and let
$$
E_\epsilon(F)=\inf_\chi \inf\big\{ E_{\epsilon,3}(u,\chi) \,:\, u\in W^{1,2}(\T^2;\R^2),\, \nabla u \mbox{ is } [0,1]^2\text{-periodic},\, \langle\nabla u\rangle=F \big\},
$$
then
$$
\epsilon^\frac{2}{3}\lesssim E_\epsilon(F)\lesssim\epsilon^\frac{2}{3}.
$$
\end{prop}
\begin{proof}
Jensen's inequality implies
\begin{equation}\label{eq:prop_per}
|\langle\chi_{1,1}\rangle-F_{1,1}|^2\le E_{el}(\chi).
\end{equation}
This combined with \eqref{eq:lem_per} and the triangle inequality yields
$$
\min\{|F_{1,1}+1|^2,|F_{1,1}-1|^2\}\lesssim \epsilon^{-\frac{2}{3}}E_\epsilon(\chi)+E_{el}(\chi),
$$
which proves the lower bound.

Regarding the upper bound, let $F=\lambda A_1+(1-\lambda)J_1$ with $J_1$ as in \eqref{eq:aux_matr}.
We consider 
$$
u_\epsilon(x)=\begin{cases}
A_1x & x\in[0,\lambda]\times[0,1], \\
v_\epsilon(x) & x\in[\lambda,1]\times[0,1],
\end{cases}
$$
where $v_\epsilon$ is a branching construction given by Lemma \ref{prop:vert-bra} with $A=A_2, B=A_3, F=J_1$, $p=2$ and $R=[\lambda,1]\times[0,1]$.
\end{proof}

Note that if $F\in K^{(1)}_{3}\setminus K_3$ it is well known that $
E_\epsilon(F) \sim \epsilon$ since in this case only a single twin is necessary.

\section{Four Wells: Proof of Theorem \ref{thm:K4}}
\label{sec:proof4_wells_lower}
In this section, we discuss the three-dimensional problem involving four gradients which had been introduced in Section \ref{sec:fourthintro}. 

Using the observation from Lemma \ref{lem:Fourier}, we note that in this setting we have the following Fourier characterization of the elastic energy:
For every $F\in \diag(3,\R)$ we have
$$
E_{el,3}(\chi,F)=\sum_{k\in\Z^3} \frac{k_2^2+k_3^2}{|k|^2}|\F \tilde \chi_{1,1}|^2+\frac{k_1^2+k_3^2}{|k|^2}|\F \tilde \chi_{2,2}|^2+\frac{k_1^2+k_2^2}{|k|^2}|\F \tilde \chi_4|^2.
$$
We emphasize that for the specific matrices in $K_4$ we have that $\chi_{3,3}=\chi_4$ and that in analogy with the notation from \eqref{eq:char_func_norm} we will also set $\tilde{\chi}_4 = \tilde{\chi}_{3,3} = \chi_{3,3}-F_{3,3}$ (often also viewed as a function on $\T^3$ by considering the periodic extension).

The section again consists of two parts: In Section \ref{sec:4wells_lower} we first prove the lower bounds from Theorem \ref{thm:K4}, then in Section \ref{sec:proof4_wells_upper}, we deduce the corresponding upper bounds.

\subsection{Lower bound}
\label{sec:4wells_lower}
In this section we prove the lower bounds from Theorem \ref{thm:K4}. We begin by discussing boundary data which are three-fold laminates.

As in Section \ref{sec:prelim_lower}, for $j\in\{1,\dots,n\}$, $\mu\in (0,1)$, and $\mu_m>0$ we define
$$
C_{j,\mu,\mu_m}=\{k\in \Z^3 \,:\, |k|^2-k_j^2\le\mu^2|k|^2,\, |k|\le\mu_m\}, \quad j=1,2,3.
$$
As a consequence of Lemma \ref{lem:first_loc} with $\mu_2 \sim \epsilon^{2\alpha-1}$ and $\mu \sim \epsilon^{\alpha}$, we obtain that
\begin{align}
\label{eq:4_initial}
\begin{split}
&\|\tilde\chi_{1,1}-\chi_{1,\mu,\mu_2}(D)\tilde{\chi}_{1,1}\|_{L^2}^2+\|\tilde\chi_{2,2}-\chi_{2,\mu,\mu_2}(D)\tilde{\chi}_{2,2}\|_{L^2}^2+\|\tilde\chi_4-\chi_{3,\mu,\mu_2}(D)\tilde{\chi}_4\|_{L^2}^2 \\
&\lesssim \epsilon^{-2\alpha} (E_{\epsilon,4}(\chi,F) + \epsilon \Per(\Omega)).
\end{split}
\end{align}
With this in hand, we present the proof of the lower bound in Theorem \ref{thm:K4} in the case that $F \in K^{(3)}_{4}\setminus K^{(2)}_4$:

\begin{proof}[Proof of the lower bound in Theorem \ref{thm:K4} for $F\in K^{(3)}_{4}\setminus K^{(2)}_4$]

We argue in three steps.\\

\emph{Step 1: $\tilde{\chi}_{2,2} + 2\tilde{\chi}_{4}$ determines $\tilde{\chi}_{1,1}$.} We first observe that $\tilde{\chi}_{1,1} = f(\tilde{\chi}_{2,2} + 2 \tilde{\chi}_4)$, where $f$ is a polynomial of degree three.
Now $\supp(\F \tilde\chi_{2,2}+2\F\tilde\chi_4)$ concentrates on
$$
S_3(C_{2,\tilde{\mu},\tilde{\mu}_2}\cup C_{3,\mu,\mu_2})\subset\{k\in \Z^3\,:\, |k_1|\lesssim \epsilon^{3\alpha-1}\}.
$$
Thus, Proposition \ref{prop:cone-red-nonl} and \eqref{eq:4_initial} with the choices of the prefactors $\mu_2 =\frac{1}{8\tilde MM} \epsilon^{2\alpha-1}$ (where $\tilde M$ and $M$ are the constants defined in the following lines) and $\mu = \epsilon^{\alpha}$ imply
\begin{align*}
&\|\tilde{\chi}_{1,1}-{\chi}_{1,\mu,\mu_3}(D)\tilde{\chi}_{1,1}\|_{L^2}^2+\|\tilde{\chi}_{2,2}-{\chi}_{2,\mu,\mu_2}(D)\tilde{\chi}_{2,2}\|_{L^2}^2+\|\tilde{\chi}_4-{\chi}_{3,\mu,\mu_2}(D)\tilde{\chi}_4\|_{L^2}^2 \\
&\le C \psi_{1-\gamma}(\epsilon^{-2\alpha} (E_{\epsilon,4}(\chi,F) + \epsilon \Per(\Omega))),
\end{align*}
where $C>0$ is a constant depending on $\gamma$ and $\mu_3=M\mu\mu_2$ for some constant $M>1$. 

\emph{Step 2: ${\tilde{\chi}}_{1,1}$ determines ${\tilde{\chi}}_{4}$.} Using that $\tilde{\chi}_4=g(\tilde{\chi}_{1,1})$ (with $g(t)=1-t^2$) and invoking Corollary \ref{cor:cone-red-nonl}, we obtain
\begin{equation}
\label{eq:reduction_almost}
\|\tilde{\chi}_4-{\chi}_{3,\mu,\mu_4}(D)\tilde{\chi}_4\|_{L^2}^2 \le C \psi_{(1-\gamma)^2}(\epsilon^{-2\alpha} (E_{\epsilon,4}(\chi,F) + \epsilon \Per(\Omega)) ),
\end{equation}
with $\mu_4=\tilde M\mu\mu_3=\tilde M M\mu^2\mu_2$ for some constant $\tilde M>1$.
We note that the width of the cone $C_{3,2\mu,2\mu_4}$ equals $2\max_{k\in C_{3,2\mu,2\mu_4}}\{|k_1|,|k_2|\}\le8\mu\mu_4=\frac{1}{2}\epsilon^{5\alpha-1}$.
Hence, choosing $\alpha=\frac{1}{5}$, we have
$$
C_{3,2\mu,2\mu_4}=\Big\{(0,0,k_3) \,:\, k_3\in\Z, |k_3|<2\mu_4=\frac{1}{4}\epsilon^{-\frac{1}{5}}\Big\}.
$$
In other words, $\tilde{\chi}_4$ is essentially one-dimensional.

\emph{Step 3: Conclusion.}
We conclude the argument by invoking the three-dimensional version of Lemma \ref{lem:axes_estimates}. This yields that
\begin{align*}
\|{\chi}_{3,\mu,\mu_4}(D)\tilde{\chi}_4\|_{L^2}^2 = \sum_{|k_3|\le\frac{1}{4}\epsilon^{-\frac{1}{5}}}|\F\tilde\chi_4(0,0,k_3)|^2 \lesssim \epsilon^{-\frac{2}{5}} E_{\epsilon,4}(\chi,F).
\end{align*}
As a consequence, \eqref{eq:reduction_almost} reduces to
\begin{align*}
\| \tilde{\chi}_4 \|_{L^2}^2 \leq C \psi_{(1-\gamma)^2}(\epsilon^{-\frac{2}{5}}(E_{\epsilon,4}(\chi,F) + \epsilon \Per(\Omega))).
\end{align*}
Last but not least, we observe that $\tilde{\chi}_4 = \chi_4 - F_{3,3}$ and that $0<F_{3,3}<1$ (since $F \in  K^{(3)}_{4}\setminus K^{(2)}_4$) and thus
\begin{align*}
0< \min\{|F_{3,3}|^2,|1-F_{3,3}|^2\} \leq \| \chi_4 - F_{3,3} \|_{L^2}^2 \leq C \psi_{(1-\gamma)^2}(\epsilon^{-\frac{2}{5}}(E_{\epsilon,4}(\chi,F) + \epsilon \Per(\Omega))).
\end{align*}
Together with a case distinction, considering the cases $(E_{\epsilon,4}(\chi,F) + \epsilon \Per(\Omega))\geq 1$ and $(E_{\epsilon,4}(\chi,F) + \epsilon \Per(\Omega))<1$ separately, and an absorption argument for the perimeter contribution as in the previous sections, this implies the claim.
\end{proof}

The lower order lamination bounds follow analogously. We only discuss the proof in the case of second order laminates (for the first order case this is essentially identical as in the argument in the previous section).

\begin{proof}[Proof of the lower bound in Theorem \ref{thm:K4} for $F\in K^{(2)}_{4}\setminus K^{(1)}_4$]
We first note that for $F \in F\in K^{(2)}_{4}\setminus K^{(1)}_4$ we have that $F_{1,1}\in (-1,1)$. With this we argue as in the previous proof: Again using that $\tilde{\chi}_{2,2}+2\tilde\chi_{3,3}$ determines $\tilde{\chi}_{1,1}$, from Proposition \ref{prop:cone-red-nonl}, we deduce that 
\begin{align*}
\|\tilde{\chi}_{1,1}-{\chi}_{1,\mu,\mu_3}(D)\tilde{\chi}_{1,1}\|_{L^2}^2 \le C \psi_{1-\gamma}(\epsilon^{-2\alpha} (E_{\epsilon,4}(\chi,F) + \epsilon \Per(\Omega))),
\end{align*}
for $\mu_3 = M\mu\mu_2$.
With the correct choice of the prefactors, the cone $C_{1,2\mu,2\mu_3}$ has width $\epsilon^{4\alpha-1}$.
Thus we choose $\alpha = \frac{1}{4}$ and hence, by Lemma \ref{lem:axes_estimates} we obtain that 
\begin{align*}
\|{\chi}_{1,\mu,\mu_3}(D)\tilde{\chi}_{1,1}\|_{L^2}^2 \lesssim \epsilon^{-\frac{1}{2}}E_{\epsilon,4}(\chi,F).
\end{align*}
Therefore, since for $F\in \inte(K^{(2)}_{4})$ we have $-1<F_{1,1}<1$,
\begin{align*}
0&< \min\{|F_{1,1}-1|^2,|F_{1,1}+1|^2\}\leq \|\chi_{1,1}-F_{1,1}\|_{L^2}^2 \leq \|\tilde{\chi}_{1,1}\|_{L^2}^2\\
& \le C \psi_{1-\gamma}(\epsilon^{-\frac{1}{2}} (E_{\epsilon,4}(\chi,F) + \epsilon \Per(\Omega))).
\end{align*}
Rearranging then concludes the proof.
\end{proof}

\subsection{Upper bounds}
\label{sec:proof4_wells_upper}
Here we define a construction which consists of three levels of branching related to the set $K_4$ defined in \eqref{eq:K4}, with boundary datum $F=\lambda A_4$, $\lambda\in(0,1)$.
This will prove the upper bounds stated in Theorem \ref{thm:K4}.

\subsubsection{Lower-order constructions}

We first premise some partial results that are needed to produce a three-dimensional branching construction of the third order.
We begin by defining (branched) simple laminate constructions of controlled energy and a fixed direction of lamination. We will refer to these constructions as ``one-dimensional branching constructions''.

\begin{lem}\label{lem:3D-1O-app}
Let $\tilde A,\tilde B\in \diag(3,\R)$ be such that $\tilde A-\tilde B=c e_m\otimes e_m$ with $m\in \{1,2\}$. Let $F=\lambda \tilde A+(1-\lambda)\tilde B$ for some $0<\lambda<1$ and let
$$
\omega=[0,\tilde{\ell}_1]\times[0,\tilde{\ell}_2]\times[0,\tilde{\ell}_3], \quad V=\tilde{\ell}_1 \tilde{\ell}_2 \tilde{\ell}_3,
$$
with $0<\tilde{\ell}_1,\tilde{\ell}_2,\tilde{\ell}_3\le1$.
Then for every $0<\rho<\frac{\tilde{\ell}_3}{4}$ if $m=1$, $0<\rho<\frac{\tilde{\ell}_1}{4}$ if $m=2$, and $\epsilon\in(0,1)$ there exists $v\in W^{1,\infty}(\omega;\R^3)$ such that $\nabla v\in BV(\omega;\R^{3\times 3})$, $v(x)=Fx$ for every $x\in\partial\omega$ and satisfying, if $m=1$
\begin{equation}\label{eq:3D-1O-app1}
\int_\omega\dist^2(\nabla v,\{\tilde A,\tilde B\})dx+\epsilon\|D^2v\|_{TV(\omega)} \lesssim |c|^2V \Big( \frac{1}{\tilde{\ell}_3^2} \rho^2+\frac{1}{\tilde{\ell}_2}\rho+\frac{\epsilon}{\rho}\Big),
\end{equation}
if $m=2$
\begin{equation}\label{eq:3D-1O-app2}
\int_\omega\dist^2(\nabla v,\{\tilde A,\tilde B\})dx+\epsilon\|D^2v\|_{TV(\omega)} \lesssim |c|^2V \Big( \frac{1}{\tilde{\ell}_1^2} \rho^2+\frac{1}{\tilde{\ell}_3}\rho+\frac{\epsilon}{\rho}\Big).
\end{equation}
\end{lem}
\begin{proof}
Without loss of generality, we consider the case $m=1$.
Let $\tilde v$ be given by Lemma \ref{prop:vert-bra} with $L=\tilde{\ell}_1$, $H=\tilde{\ell}_3$ and $N=\frac{\tilde{\ell}_1}{\rho}$, $A=\diag(\tilde A_{1,1},\tilde A_{3,3})$, $B=\diag(\tilde B_{1,1},\tilde B_{3,3})$.
Then we define $v\in W^{1,\infty}(\omega, \R^3)$ by setting $v_2(x)=\tilde A_{2,2}x_2$ and $(v_1,v_3)(x_1,x_2,x_3)=\tilde v(x_1,x_3)$ up to a cut-off of scale $\rho$ to attain the boundary conditions on $[0,\tilde{\ell}_1]\times\{0,\tilde{\ell}_2\}\times[0,\tilde{\ell}_3].$

Hence, from Fubini's theorem we get
$$
\int_\omega\dist^2(\nabla v,\{\tilde A,\tilde B\})dx \lesssim (1-2\rho)\tilde\ell_2\int_{[0,\tilde\ell_1]\times[0,\tilde\ell_3]}\dist^2(\nabla\tilde v,\{A,B\})^2dx+\rho\tilde\ell_1\tilde\ell_3
$$
and from the coarea formula $\|D^2v\|_{TV(\omega)}\lesssim\tilde\ell_2\|D^2\tilde v\|_{TV([0,\tilde\ell_1]\times[0,\tilde\ell_3])}$.
Then formula \eqref{eq:prop-vbra} gives \eqref{eq:3D-1O-app1}.
The case $m=2$ is completely analogous.
\end{proof}

The function $v$ from Lemma \ref{lem:3D-1O-app} above corresponding to $m=1$ \emph{branches in the direction $e_3$} and \emph{laminates in direction $e_1$}, whereas when $m=2$, $v$ branches in the $e_1$-direction and laminates in $e_2$.

Next, we proceed to construct branched double laminates in three dimensions. In what follows the matrices $A_j \in \R^{3\times 3}$ correspond to the ones from the well $K_4$ defined in \eqref{eq:K4}.

\begin{lem}\label{lem:3D-2O}
Let $K_4$ be as in \eqref{eq:K4} and let
$$
R=[0,L_1]\times[0,L_2]\times[0,L_3], \quad V=L_1 L_2 L_3,
$$
with $0<L_1,L_2,L_3\le1$.
Then for every $0<\rho_2<\rho_1<\frac{L_3}{4}$ and $\epsilon\in(0,1)$ there exists $v\in W^{1,\infty}(R;\R^3)$ such that $v(x)=0$ for every $x\in\partial R$ and
\begin{equation}\label{eq:3D-2O}
\int_R\dist^2(\nabla v,\{A_1,A_2,A_3\})dx+\epsilon\|D^2v\|_{TV(R)} \lesssim V \Big(\Big(\frac{\rho_2}{\rho_1}\Big)^2+\frac{1}{L_3^2} \rho_1^2+\frac{1}{L_3}\rho_2+\frac{1}{L_2}{\rho_1}+\frac{\epsilon}{\rho_2}\Big).
\end{equation}
\end{lem}

\begin{proof}
Let $v^{(1)}$ be given by Lemma \ref{lem:3D-1O-app} with $m=1$, $\tilde{\ell}_l=L_l$, $\rho=\rho_1$ and $\tilde{A}=A_1$, $\tilde{B}=-A_1$ and $F=0$.
Denote with $\chi^{(1)}$ the projection of $\nabla v^{(1)}$ onto $\{\pm A_1\}$ (which is almost everywhere well-defined).
The application of Lemma \ref{prop:vert-bra} in the proof of Lemma \ref{lem:3D-1O-app} implies that the region $\{\chi^{(1)}=-A_1\}\cap ([0,L_1]\times[\rho_1 L_2,(1-\rho_1)L_2]\times[0,L_3])$ consists of the union of the following cells
\begin{equation}
\begin{split}
R_{j,k}^{(1)} &= \big\{(x_1,x_3)\in\omega_{j,k}^{(1)}, x_2\in[\rho_1 L_2,(1-\rho_1)L_2]\big\}, \\
R_{j,k}^{(3)} &= \big\{(x_1,x_3)\in\omega_{j,k}^{(3)}, x_2\in[\rho_1 L_2,(1-\rho_1)L_2]\big\},
\end{split}
\end{equation}
where $\omega_{j,k}^{(1)}$ and $\omega_{j,k}^{(3)}$ are defined by \eqref{eq:cells-2order} with $\ell_j=\frac{\rho_1}{2^j}$, $y_j=L_3-L_3\frac{\theta^j}{2}$, $h_j:= y_{j+1}-y_j$ and $\lambda=\frac{1}{2}$ with $j\in\{0,\dots,j_0\}$ and $k\in\{0,\dots,k_0(j)\}$ where $j_0$ and $k_0(j)$ are defined as in the proof of Lemma \ref{prop:vert-bra}.

Now we apply Lemma \ref{lem:3D-1O-app} (up to suitable translations) on every $R_{j,k}^{(1)}$ with $m=2$, $\tilde{A}=A_2$, $\tilde{B}=A_3$, $F=-A_1$ and $\rho=\frac{\rho_2}{2^{j+3}}$.

Notice that this corresponds to $\tilde{\ell}_1=\frac{\rho_1}{2^{j+1}}$, $\tilde{\ell}_2=(1-2\rho_1)L_2$, $\tilde{\ell}_3=L_3\frac{(1-\theta)}{2}\theta^j$.
Hence, we obtain $v^{(1)}_{j,k}$ attaining boundary conditions $v^{(1)}$ on $R_{j,k}^{(1)}$ and in this situation the bounds from  equation \eqref{eq:3D-1O-app2} read
\begin{multline}\label{eq:3D-2O-app1}
\int_{R_{j,k}^{(1)}}\dist^2(\nabla v^{(1)}_{j,k},\{A_2,A_3\})dx+\epsilon\|D^2 v^{(1)}_{j,k}\|_{TV(R_{j,k}^{(1)})} \\
\lesssim L_2 L_3 \theta^j \frac{\rho_1}{2^j}\Big(\Big(\frac{\rho_2}{\rho_1}\Big)^2+\frac{1}{(2\theta)^jL_3}\rho_2+2^j\frac{\epsilon}{\rho_2}\Big).
\end{multline}
The construction inside $R_{j,k}^{(3)}$ is obtained by a shear of the construction in $R_{j,k}^{(1)}${, that is}
$$
v_{j,k}^{(3)}(x_1,x_2,x_3)=v_{j,k}^{(1)}\Big(x_1-\frac{\ell_j(x_{3}-y_j)}{4h_j}-\frac{\ell_j}{4},x_2,x_3\Big)+\Big(-\frac{\ell_j}{4}+\frac{\ell_j(x_3-y_j)}{4h_j},0,0\Big).
$$
Thus, analogously as in \eqref{eq:2nd-ord-cell2} we get
\begin{equation}\label{eq:3D-2O-app2}
\int_{R_{j,k}^{(3)}}\dist^2(\nabla v_{j,k}^{(3)},\{A_1,A_2,A_3\})dx \lesssim \int_{R_{j,k}^{(1)}}\dist^2(\nabla v_{j,k}^{(1)},\{A_1,A_2,A_3\})dx+\frac{\ell_j^3}{h_j}L_2.
\end{equation}
The function $v$ is defined as $v_{j,k}^{(1)}$ and $v_{j,k}^{(3)}$ respectively on $R_{j,k}^{(1)}$ and $R_{j,k}^{(3)}$ and equals $v^{(1)}$ elsewhere.
We also recall (for a better understanding) that there are no elastic energy contributions outside the union of $R_{j,k}^{(1)}$, $R_{j,k}^{(3)}$ (apart from the cut-off region) since in there $\nabla v^{(1)}=A_1$ by construction.
The total energy amount is therefore obtained by adding the contributions on $R_{j,k}^{(1)}\cup R_{j,k}^{(3)}$.
By \eqref{eq:3D-2O-app1}, \eqref{eq:3D-2O-app2}, summing over $j$ and $k$ and adding the contribution of the cut-off on $[0,L_1]\times([0,\rho_1 L_2]\cup[(1-\rho_1)L_2,L_2])\times[0,L_3]$ we have
\begin{equation*}
\begin{split}
& \int_R\dist^2(\nabla v,\{A_1,A_2,A_3\})dx+\epsilon\|D^2 v\|_{TV(R)} \\
& \qquad \lesssim \sum_{j=0}^{j_0+1}\sum_{k=0}^{k_0(j)}\int_{R_{j,k}^{(1)}\cup R_{j,k}^{(3)}}\dist^2(\nabla v,\{A_1,A_2,A_3\})dx+\epsilon\|D^2 v\|_{TV(R_{j,k}^{(1)}\cup R_{j,k}^{(3)})}+L_1L_3\rho_1
\\
& \qquad \lesssim V \sum_{j=0}^\infty \Big(\theta^j\Big(\frac{\rho_2}{\rho_1}\Big)^2+\frac{2^{-j}}{L_3}\rho_2+(2\theta)^j\frac{\epsilon}{\rho_2}\Big)+(4\theta)^{-j}\frac{1}{L_3^2}\rho_1^2\Big)+L_1L_3\rho_1,
\end{split}
\end{equation*}
and the result follows by the summubility of $\theta^j$, $(2\theta)^j$, $2^{-j}$ and $(4\theta)^{-j}$ recalling that $\theta\in(\frac{1}{4},\frac{1}{2})$.
\end{proof}

\subsubsection{Third-order construction}
With the previous auxiliary results in hand, we approach the proof of the upper bound from Theorem \ref{thm:K4} in the presence of a third order laminate as Dirichlet boundary condition. Here the strategy is the following:
\begin{itemize}
\item In the outer-most, coarsest branching construction, we do not use a purely ``one-dimensional'' branching construction, but a more refined construction which is ``two-dimensional'' in its branching behaviour. More precisely, we use a construction which oscillates in the $e_3$-direction and branches both towards the $e_1$ and $e_2$ boundaries, see Figures \ref{fig:3Dbra}, \ref{fig:3D-cells1}, \ref{fig:3D-cells2}. It is this outer-most improvement which leads to the improved scaling behaviour (compared to a simple twinning construction and purely ``two-dimensional branching constructions'').
\item All inner branching constructions are essentially ``one-dimensional branching constructions'' (as in Lemmas \ref{lem:3D-1O-app} and \ref{lem:3D-2O}).
\end{itemize}

\begin{prop}\label{prop:3D-3order}
Let $F$ be as above and $E_{\epsilon,4}$ as in \eqref{eq:energyK4}.
For every $\epsilon\in(0,1)$ there exist $u_\epsilon\in W^{1,\infty}((0,1)^3;\R^3)$ and $\chi_\epsilon\in BV((0,1)^3;K_4)$ such that $u_\epsilon(x)=Fx$ for every $x\in\R^3\setminus(0,1)^3$ and
\begin{equation}
E_{\epsilon,4}(u_\epsilon,\chi_\epsilon) \lesssim \epsilon^\frac{2}{5}.
\end{equation}
\end{prop}
\begin{proof}
The proof is divided into different steps.

\emph{Step 1.}
We define a first order branching construction $u^{(1)}$ oscillating in the direction $e_3$ corresponding to a convex combination of the zero matrix and $A_4$.
Let $\tilde u$ be defined by Lemma \ref{prop:vert-bra} with $L=H=$1, $N=\frac{1}{r}$, $A=0$, $B=\diag(1,0)$, $F=\diag(\lambda,0)$ and $p=2$, where $0<r<1$ is an small parameter which is to be fixed below.
Then we define $u^{(1)}$ such that $u^{(1)}_1=u^{(1)}_2=0$ and
$$
u^{(1)}_3(x_1,x_2,x_3)=\tilde u_1(x_3,\rho(x_1,x_2)),
\quad
\text{where } \rho(x_1,x_2):=\max\Big\{\Big|x_1-\frac{1}{2}\Big|,\Big|x_2-\frac{1}{2}\Big|\Big\}+\frac{1}{2}
$$
which is Lipschitz continuous and attains the boundary conditions.
\begin{figure}[t]
\begin{center}
\includegraphics[width=0.3\linewidth]{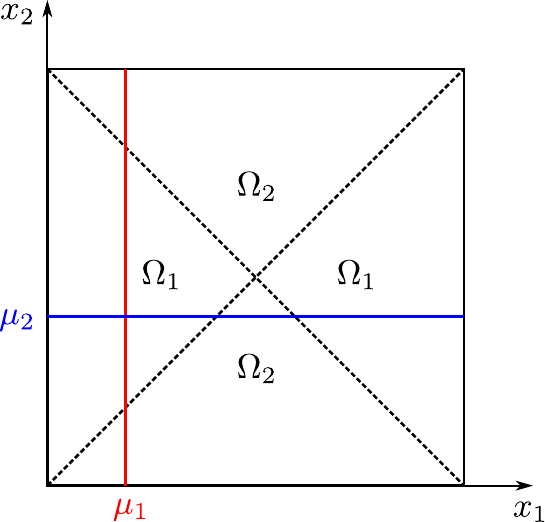}
\,
\includegraphics[width=0.3\linewidth]{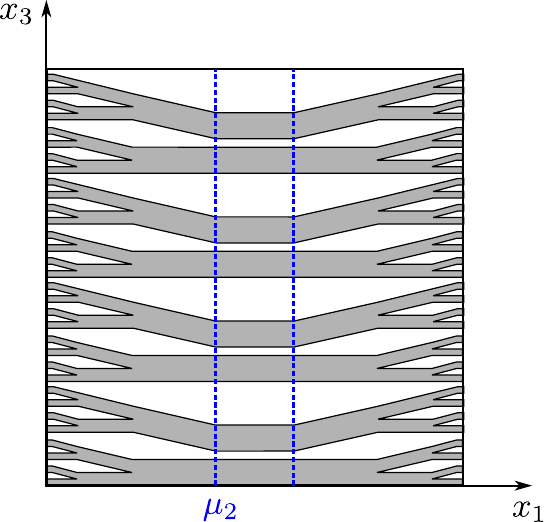}
\,
\includegraphics[width=0.3\linewidth]{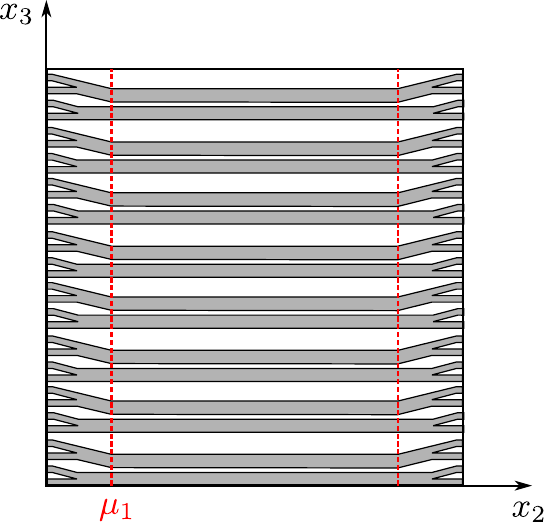}
\end{center}
\caption{Schematic illustration of the ``two-dimensional'' branching construction from the proof of Proposition \ref{prop:3D-3order}. On the left the projection of the domain onto the  $x_1 x_2$-plane and the subdivision into $\Omega_1$ and $\Omega_2$.
In the center the branching construction in the slice $x_2=\mu_2$, on the right that corresponding to $x_1=\mu_1$.}
\label{fig:3Dbra}
\end{figure}\\
The elastic energy contribution is controlled by that of the function $\tilde u$ by the coarea formula, that is
\begin{align*}
\int_{[0,1]^3}\dist^2(\nabla u^{(1)},\{A_4,0\})dx &= \int_{[0,1]^3}\dist^2(\nabla u^{(1)}_3,\{ e_3,0\})dx \\
&= \frac{1}{\sqrt{2}}\int_\frac{1}{2}^1\int_0^1 4\rho\dist^2(\nabla\tilde u_1,\{ e_1,0\})dx_3d\rho \lesssim r^2.
\end{align*}
We define the sets $\Omega_1:=\{x\in[0,1]^3 \,:\, |x_1-\frac{1}{2}|>|x_2-\frac{1}{2}|\}$ and $\Omega_2=[0,1]^3\setminus\overline{\Omega_1}$.
We remark that on $\Omega_1$ the function $u^{(1)}$ depends only on $x_1$ and $x_3$, whereas on $\Omega_2$ only on $x_2$ and $x_3$, see Figure \ref{fig:3Dbra}.
Due to  the symmetric role of $x_1$ and $x_2$, we infer that
$$
|D^2u^{(1)}|(\Omega_1)=|D^2u^{(1)}|(\Omega_2)\le|D^2\tilde u|((0,1)^2)\sim\frac{1}{r}+C,
$$
where the additive constant comes from the jump of the gradient on the diagonals $\p\Omega_1\cup\p\Omega_2\setminus Per(\Omega)$.
Hence,
\begin{equation}\label{eq:3D-step1}
\int_{[0,1]^3}\dist^2(\nabla u^{(1)},\{A_4,0\})dx+\epsilon|D^2u^{(1)}|\big((0,1)^3\big) \lesssim r^2+\frac{\epsilon}{r}.
\end{equation}

\begin{figure}[t]
\begin{center}
\includegraphics{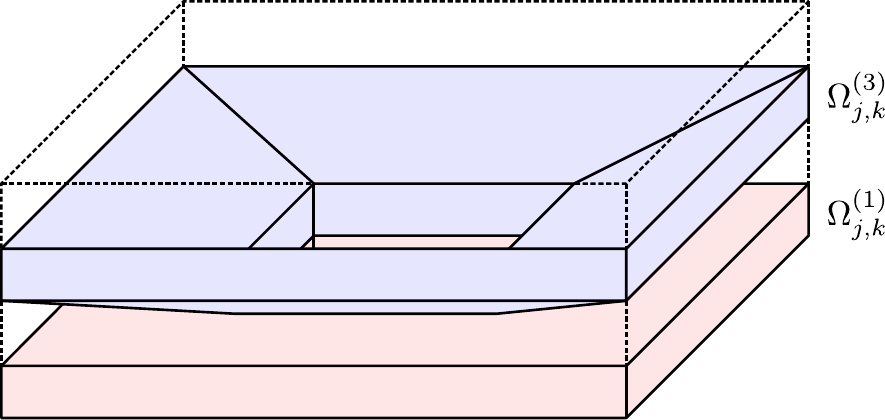}
\end{center}
\caption{Schematic illustration of the auxiliary sets in Proposition \ref{prop:3D-3order}. A representation of the sets $\Omega_{j,k}^{(1)}$ and $\Omega_{j,k}^{(3)}$.
The hashed lines correspond to the edges of the cell $\Omega_{j,k}$.}
\label{fig:3D-cells1}
\end{figure}

\emph{Step 2.}
From the branching procedure in Step 1, $\Omega$ is divided into cells $\Omega_{j,k}$ that are the analogues of $\omega_{j,k}$ defined in \eqref{eq:cells-bra}.
These are
$$
\Omega_{j,k}=\big\{y_j\le\rho(x_1,x_2)\le y_{j+1},\, k\ell_j\le x_3\le(k+1)\ell_j\big\},
$$
where $\ell_j=\frac{r}{2^j}$, $y_j=1-\frac{\theta^j}{2}$ with $\frac{1}{4}<\theta<\frac{1}{2}$, see Figure \ref{fig:3D-cells1}. 
Let $\chi^{(1)}$ denote the projection of $\nabla u^{(1)}$ on $\{A_4,0\}$, then the region $\{\chi^{(1)}=0\}$ (apart from the cut-off region) consists of the union of $\Omega_{j,k}^{(1)}$ and $\Omega_{j,k}^{(3)}$, as represented in Figure \ref{fig:3D-cells1} and defined by the following formulae: 
\begin{equation}\label{eq:3D-cells-2order}
\begin{split}
\Omega_{j,k}^{(1)} &:=\Big\{y_j\le\rho(x_1,x_2)\le y_{j+1},\, k\ell_j\le x_3\le\Big(k+\frac{\lambda}{2}\Big)\ell_j\Big\}, \\
\Omega_{j,k}^{(3)} &:=\Big\{y_j\le\rho(x_1,x_2)\le y_{j+1},\, \Big(k+\frac{\lambda}{2}\Big)\ell_j\le x_3-\frac{(1-\lambda)\ell_j (\rho(x_1,x_2)-y_j)}{2h_j}\le(k+\lambda)\ell_j\Big\},
\end{split}
\end{equation}
where $h_j=y_{j+1}-y_j$.
In each of these subdomains $u^{(1)}$ is piecewise affine and we replace it with the second order branching construction defined by Lemma \ref{lem:3D-2O} attaining boundary condition $u^{(1)}$. 

\begin{figure}[t]
\begin{center}
\includegraphics{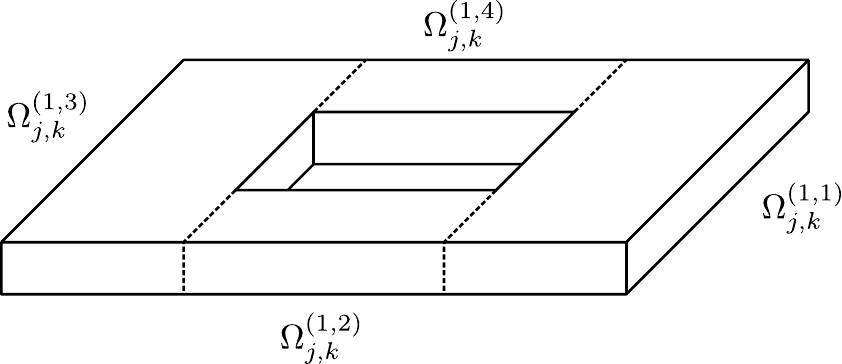}
\end{center}
\caption{Schematic illustration of the auxiliary sets in Proposition \ref{prop:3D-3order}: Subdivision of $\Omega_{j,k}^{(1)}$ into smaller cells.}
\label{fig:3D-cells2}
\end{figure}

\emph{Step 3.}
We first define the construction on $\Omega_{j,k}^{(1)}$.
Since Lemma \ref{lem:3D-2O} works on axis-parallel rectangular cells we split $\Omega_{j,k}^{(1)}$ into four rectangular sets
\begin{align*}
\Omega_{j,k}^{(1,1)} &= [y_j,y_{j+1}]\times[1-y_{j+1},y_{j+1}]\times\Big[k\ell_j,\Big(k+\frac{\lambda}{2}\Big)\ell_j\Big] ,\\
\Omega_{j,k}^{(1,2)} &= [1-y_j,y_j]\times[y_j,y_{j+1}]\times\Big[k\ell_j,\Big(k+\frac{\lambda}{2}\Big)\ell_j\Big],
\end{align*}
with $\Omega_{j,k}^{(1,3)}$ and $\Omega_{j,k}^{(1,4)}$ being reflections, with respect to $(\frac{1}{2},\frac{1}{2},\frac{1}{2})$, of $\Omega_{j,k}^{(1,1)}$ and $\Omega_{j,k}^{(1,2)}$ respectively (Figure \ref{fig:3D-cells2}).
We apply Lemma \ref{lem:3D-2O} (up to suitable translations) on $\Omega_{j,k}^{(1,1)}$ and $\Omega_{j,k}^{(1,3)}$, thus with $L_1=\frac{1-\theta}{2}\theta^j$, $L_2=2y_{j+1}-1$, $L_3=\lambda\frac{r}{2^{j+1}}$, $\rho_1=\lambda\frac{r_2}{2^{j+3}}$ and $\rho_2=\lambda\frac{r_3}{2^{j+3}}$, and on $\Omega_{j,k}^{(1,2)}$ and $\Omega_{j,k}^{(1,4)}$ with $L_1=2y_j-1$, $L_2=\frac{1-\theta}{2}\theta^j$, $L_3=\lambda\frac{r}{2^{j+1}}$, $\rho_1=\lambda\frac{r_2}{2^{j+3}}$ and $\rho_2=\lambda\frac{r_3}{2^{j+3}}$. Here the parameters $r_2, r_3$ satisfy the bounds $0<r_3<r_2<r<1$. Their more precise size is to be determined below.
The resulting function $u_{j,k}^{(1)}$ is defined on the whole $\Omega_{j,k}^{(1)}$, it is continuous by construction and by \eqref{eq:3D-2O} satisfies
\begin{equation}\label{eq:3D-3O-2}
\begin{split}
&\int_{\Omega_{j,k}^{(1)}}\dist^2(\nabla u_{j,k}^{(1)},\{A_1,A_2,A_3\})dx+\epsilon\|D^2 u_{j,k}^{(1)}\|_{TV(\Omega_{j,k}^{(1)})} \\
& \qquad \lesssim \frac{r}{2^j}\Big(\theta^j\Big(\frac{r_3}{r_2}\Big)^2+\theta^j\Big(\frac{r_2}{r}\Big)^2+\theta^j\frac{r_3}{r}+(2\theta)^j\frac{\epsilon}{r_3}+\frac{1}{2^j}r_2\Big).
\end{split}
\end{equation}
The construction inside $\Omega_{j,k}^{(3)}$ is obtained from the one above as follows
\begin{equation}\label{eq:shear}
\begin{split}
&u_{j,k}^{(3)}(x_1,x_2,x_3) \\
&\quad =u_{j,k}^{(1)}\Big(x_1,x_2,x_3-\frac{(1-\lambda)\ell_j(\rho(x_1,x_2)-y_j)}{2h_j}-\frac{\lambda\ell_j}{2}\Big)+\Big(0,0,\frac{(1-\lambda)\ell_j(\rho(x_1,x_2)-y_j)}{2h_j}\Big).
\end{split}
\end{equation}
Thus analogously as in \eqref{eq:2nd-ord-cell2} we get
\begin{equation}\label{eq:3D-3O-3}
\int_{\Omega_{j,k}^{(3)}}\dist^2(\nabla u_{j,k}^{(3)},K_4)dx \lesssim \int_{\Omega_{j,k}^{(1)}}\dist^2(\nabla u_{j,k}^{(1)},K_4)dx+\frac{\ell_j^3}{h_j}.
\end{equation}
The function $u$ is then defined as $u_{j,k}^{(1)}$ and $u_{j,k}^{(3)}$ respectively on $\Omega_{j,k}^{(1)}$ and $\Omega_{j,k}^{(3)}$ and equals $u^{(1)}$ elsewhere.
By \eqref{eq:3D-3O-2}, \eqref{eq:3D-3O-3} and summing on $j$ and $k$ we have
$$
E_{\epsilon,4}(u,\chi) \lesssim \Big(\frac{r_3}{r_2}\Big)^2+\Big(\frac{r_2}{r}\Big)^2+r^2+\frac{r_3}{r}+r_2+\frac{\epsilon}{r_3},
$$

where $\chi$ is the projection of $\nabla u$ on $K_4$ and we have used that $\|D\chi\|_{TV}\le C\|D^2 u\|_{TV}$.
An optimization argument yields that $r_k\sim r^k$ and $r\sim\epsilon^\frac{1}{5}$ and the result is proven.
\end{proof}

\begin{proof}[Proof of the upper bounds from Theorem \ref{thm:K4}] We split the proof into three cases:

\emph{Case 1: $F\in K_4^{(1)}\setminus K_4$.}
Working as in Step 1 of the proof of Proposition \ref{prop:3D-3order}, we can find $u^{(1)}$, a branching construction with gradient oscillating between $A_2$ and $A_3$.
A scale optimization in \eqref{eq:3D-step1} gives the upper bound for $F\in K_4^{(1)}\setminus K_4$ claimed in point (i).

\emph{Case 2: $F\in K_4^{(2)}\setminus K_4^{(1)}$.}
Again we work as in the proof of Proposition \ref{prop:3D-3order}.
In particular, we define the constructions $u_{j,k}^{(1)}$ as in Step 2 by applying Lemma \ref{lem:3D-1O-app} rather than Lemma \ref{lem:3D-2O}.
Thus, the energy contributions corresponding to \eqref{eq:3D-3O-2} and \eqref{eq:3D-3O-3} read
$$
\int_{\Omega_{j,k}^{(1)}}\dist^2(\nabla u_{j,k}^{(1)},\{A_2,A_3\})dx+\epsilon\|D^2 u_{j,k}^{(1)}\|_{TV(\Omega_{j,k}^{(1)})} \lesssim \frac{r}{2^j}\Big(\theta^j\Big(\frac{r_2}{r}\Big)^2+\frac{1}{2^j}r_2+(2\theta)^j\frac{\epsilon}{r_2}\Big),
$$
$$
\int_{\Omega_{j,k}^{(3)}}\dist^2(\nabla u_{j,k}^{(3)},K_4)dx \lesssim \int_{\Omega_{j,k}^{(1)}}\dist^2(\nabla u_{j,k}^{(1)},K_4)dx+\frac{\ell_j^3}{h_j}.
$$
Summing over $j$ and $k$ we get
$$
E_{\epsilon,4}(u^{(2)},\chi^{(2)}) \lesssim \Big(\frac{r_2}{r}\Big)^2+r_2+r^2+\frac{\epsilon}{r_2},
$$
which implies the result after an optimization in $r$ and $r_2$.

\emph{Case 3: $F\in K_4^{(3)}\setminus K_4^{(2)}$.}
Eventually, Proposition \ref{prop:3D-3order} directly proves the upper bound of Theorem \ref{thm:K4} (ii) for $m=4$.
\end{proof}

\section{Laminates of Arbitrary Order: Proof of Theorem \ref{thm:Kn}}
\label{sec:low_n}

\subsection{Proof of the lower bounds from Theorem \ref{thm:Kn}}
In this section, we provide the proof of the lower bounds from Theorem \ref{thm:Kn}. This follows by combining the high and low frequency estimates from Proposition \ref{prop:cone-red-nonl} and Lemma \ref{lem:axes_estimates} along the same lines as the arguments from the previous section.

\begin{proof}[Proof of the lower bounds from Theorem \ref{thm:Kn}]
The argument for laminates of first order follows from exactly the same argument as in the previous sections. We thus only focus on laminates of order two and higher. Let thus $F\in K^{(lc)}_{n+1}\subset \R^{n\times n} $ be a (genuine) laminate of order $2\leq m \leq n$.

We argue in several steps generalizing the arguments from the previous sections.

\emph{Step 1: $\tilde{\chi}_{2,2} + 2\tilde{\chi}_{3,3}$ determines $\tilde{\chi}_{1,1}$.} As in the previous sections, we obtain that for some polynomial $f$ of degree $n$ we have that $\tilde{\chi}_{1,1}=f(\tilde{\chi}_{2,2}+ 2\tilde{\chi}_{3,3})$. The Fourier support of $\tilde{\chi}_{2,2} + 2 \tilde{\chi}_{3,3}$ concentrates on  
\begin{align*}
S_n(C_{2,\tilde{\mu},\tilde{\mu}_2}\cup C_{3,\tilde{\mu},\tilde{\mu}_2 }\cup \hdots \cup  C_{n,\tilde{\mu},\tilde{\mu}_2 }) \subset \{k \in \Z^n: \ |k_1|\lesssim \epsilon^{3\alpha -1}\}.
\end{align*}
By Proposition \ref{prop:cone-red-nonl} we hence infer that for $\tilde{\mu}_2 = c \epsilon^{2\alpha -1}$ (with the prefactor $c>0$ to be determined later) and $\tilde{\mu} = \epsilon^{\alpha}$ it holds that
\begin{align*}
\|\tilde{\chi}_{1,1}-{\chi}_{1,\tilde{\mu},\tilde{\mu}_3}(D)\tilde{\chi}_{1,1}\|_{L^2}^2+ \sum\limits_{j=2}^{n} \|\tilde{\chi}_{j,j}-{\chi}_{j,\tilde{\mu},\tilde{\mu}_2}(D)\tilde{\chi}_{j,j}\|_{L^2}^2 \lesssim \psi_{1-\gamma}(\epsilon^{-2\alpha} (E_{\epsilon,n+1}(\chi,F) + \epsilon \Per(\Omega)))
\end{align*}
where $\tilde\mu_3=\tilde M_1\tilde\mu\tilde\mu_2$ for some constant $\tilde M_1>1$.

\emph{Step 2: $\tilde{\chi}_{1,1} $ determines $\tilde{\chi}_{3,3} + \lambda_4 \tilde{\chi}_{4,4} + \lambda_5 \tilde{\chi}_{5,5} + \cdots + \lambda_{n} \tilde{\chi}_{n n} $, where $\lambda_4 = 1-\nu_4$ and $\lambda_{j+1}=1-\nu_{j+1}(1-\nu_j)$ for $j\in \{4,\dots,n-1\}$.}
By definition of the matrices $A_1,\dots,A_{n+1}$, we obtain that 
\begin{align*}
\tilde{\chi}:=\tilde{\chi}_{3,3} + \lambda_4 \tilde{\chi}_{4,4} + \lambda_5 \tilde{\chi}_{5,5} + \cdots + \lambda_{n} \tilde{\chi}_{n n} = g(\tilde{\chi}_{1,1})
\end{align*}
for some polynomial $g$ and the values of $\lambda_j$ given above.
As a consequence, again by Corollary \ref{cor:cone-red-nonl}, we deduce that 
\begin{align*}
\sum\limits_{j=2}^{n} \|\tilde{\chi}- \chi_{j,\tilde\mu,\tilde\mu_4}(D)\tilde{\chi}\|_{L^2}^2 \le C \psi_{(1-\gamma)^2}(\epsilon^{-2\alpha} (E_{\epsilon,n+1}(\chi,F) + \epsilon \Per(\Omega))),
\end{align*}
where $C>0$ is a constant depending on $\gamma$ and $\tilde\mu_4=\tilde M_2\tilde\mu\tilde\mu_3$ for some $\tilde M_2>1$.

Since the functions $\tilde{\chi}_{j,j}$ have Fourier support in different, disjoint cones $C_{j,2\mu,2\mu_3}$, this additional Fourier cut-off implies that for all $j\in \{3,\dots, n\}$ we, in particular, deduce that
\begin{align*}
\sum\limits_{j=3}^{n} \|\tilde{\chi}_{j,j}- \chi_{j,\tilde\mu,\tilde\mu_4}(D)\tilde{\chi}_{j,j}\|_{L^2}^2 \le C \psi_{(1-\gamma)^2}(\epsilon^{-2\alpha} (E_{\epsilon,n+1}(\chi,F) + \epsilon \Per(\Omega))),
\end{align*}

\emph{Step 3: For $j\in \{3,\dots,n-1\}$ we have that $\tilde{\chi}_{j,j}$ determines $\tilde{\chi}_{j+1,j+1}$.} Again from the definition of the matrices $A_1,\dots,A_n, A_{n+1}$ we have that $\tilde{\chi}_{j+1,j+1} = g_j(\tilde{\chi}_{j,j})$ for some polynomial $g_j$. Further applications of Corollary \ref{cor:cone-red-nonl} hence yield that for $j\geq 3$
\begin{align*}
\|\tilde{\chi}_{j+1,j+1}- \chi_{j+1,\tilde\mu,\tilde\mu_{j+2}}(D)\tilde{\chi}_{j+1,j+1}\|_{L^2}^2 \le C\psi_{(1-\gamma)^2}(\epsilon^{-2\alpha} (E_{\epsilon,n+1}(\chi,F) + \epsilon \Per(\Omega))),
\end{align*}
where $\tilde\mu_{j+2}=\tilde M_{j}\tilde\mu\tilde\mu_{j+1}$.

\emph{Step 4: Conclusion.}
Let thus $m\geq 3$. By the previous steps, we obtain that $\tilde{\chi}_{m, m}$ satisfies
\begin{align*}
\|\tilde{\chi}_{m,m}- \chi_{m,\tilde\mu,\tilde\mu_{m+1}}(D)\tilde{\chi}_{m,m}\|_{L^2}^2 \le C \psi_{(1-\gamma)^{m-1}}(\epsilon^{-2\alpha} (E_{\epsilon,n+1}(\chi,F) + \epsilon \Per(\Omega))),
\end{align*}
for $\mu_{m+1} = \tilde M_{m-1}\tilde\mu\tilde\mu_m=\Pi_{j=1}^{m-1}\tilde M_j \mu^{m-1}\mu_2=c\Pi_{j=1}^{m-1}\tilde M_j\epsilon^{(m+1)\alpha-1}$. Taking $c=\frac{1}{8}\Pi_{j=1}^{m-1}\tilde M_j$ in the definition of $\tilde\mu_2$, the truncated cone $C_{m,2\tilde\mu,2\tilde\mu_{m+1}}$ has width $8\tilde\mu\tilde\mu_{m+1}=\epsilon^{(m+2)\alpha -1}$. Choosing $\alpha = \frac{1}{m+2}$, we thus arrive at 
\begin{align*}
\|\tilde{\chi}_{m,m}- \chi_{m,\tilde\mu,\tilde\mu_{m+1}}(D)\tilde{\chi}_{m,m}\|_{L^2}^2 \le C \psi_{(1-\gamma)^{m-1}}(\epsilon^{-\frac{2}{m+2}} (E_{\epsilon,n+1}(\chi,F) + \epsilon \Per(\Omega))) .
\end{align*}
Combining this with Lemma \ref{lem:axes_estimates} further yields
\begin{align*}
\|\tilde{\chi}_{m,m} \|_{L^2}^2 \le C \psi_{(1-\gamma)^{m-1}}(\epsilon^{-\frac{2}{m+2}} (E_{\epsilon,n+1}(\chi,F) + \epsilon \Per(\Omega))) .
\end{align*}
As a consequence, and by the characterization of $K^{(m)}_{n+1}$, we arrive at
\begin{align*}
0 &< \min\{|F_{m,m}|,|F_{m,m}-1|^2 \} \leq \|\chi_{m,m} - F_{m,m} \|_{L^2}^2 \leq \|\tilde{\chi}_{m,m}\|_{L^2}^2\\
& \le C \psi_{(1-\gamma)^{m-1}}(\epsilon^{-\frac{2}{m+2}} (E_{\epsilon,n+1}(\chi,F) + \epsilon \Per(\Omega))) .
\end{align*}
Solving for $E_{\epsilon,n+1}({\chi},F)$ (by carrying out a case distinction as in the other lower bound proofs), absorbing the perimeter contribution in the left-hand-side and minimizing in $\chi$ then implies the desired result.
The case $m=2$ is obtained analogously by exploiting the estimate of Step 1.
\end{proof}

\subsection{Upper bounds}

In this subsesction we proceed similarly as we did for three-dimensional constructions. In addition, we then use an iterative argument to obtain the upper bound for laminates of order $n$.

\subsection{Lower order auxiliary constructions}

Again, we begin by giving branched simple laminate constructions whose oscillation occur in direction $e_m$ and which branch in direction $e_{m+1}$ for any $m\in \{1,\dots,n-1\}$.
This is a straightforward generalization of the three-dimensional result of Lemma \ref{lem:3D-1O-app}. We thus omit its proof.

\begin{lem}\label{lem:nD-1O}
Let $A,B\in \diag(n,\R)$ such that $A-B=c e_m\otimes e_m$ with $m\in \{1,\dots,n-1\}$, let $F=\lambda A+(1-\lambda)B$ for some $0<\lambda<1$ and let
$$
\omega=\prod_{l=1}^n[0,\tilde{\ell}_l], \quad V=\prod_{l=1}^n\tilde{\ell}_l,
$$
with $0<\tilde{\ell}_l\le1$.
Then for every $0<\rho<\frac{\tilde{\ell}_{m+1}}{4}$ and $\epsilon\in(0,1)$ there exists $v\in W^{1,\infty}(\omega;\R^n)$ such that $v(x)=Fx$ for every $x\in\partial\omega$ and satisfying
\begin{equation}\label{eq:nD-1O}
\int_\omega\dist^2(\nabla v,\{A,B\})dx+\epsilon\|D^2v\|_{TV(\omega)} \lesssim |c|^2V \Big( \frac{1}{\tilde{\ell}_{m+1}^2} \rho^2+\sum_{l\neq m,m+1}\frac{1}{\tilde{\ell}_l}\rho+\frac{\epsilon}{\rho}\Big).
\end{equation}
\end{lem}

We now give the energy contribution of a branching construction of order $h\le n-1$.
To do so, without loss of generality and for notational simplicity, we rearrange the order of our matrices slightly and consider $K=\tilde K_{n+1}=\{\tilde A_1,\dots,\tilde A_{n+1}\}$, where $\tilde A_j=A_j$ for every $j>3$ and $\tilde A_3=e_2\otimes e_2$, $\tilde A_2=e_1\otimes e_1-e_2\otimes e_2$ and $\tilde A_1=-e_1\otimes e_1-e_2\otimes e_2$, where $A_j\in K_{n+1}$ are defined as in \eqref{eq:energyKn}.

We take into account boundary conditions in the $(n-1)$ lamination convex hull $\tilde K_{n+1}^{(n-1)}$ (see Definition \ref{defi:laminates}). 
We emphasize that all constructions for $\tilde{K}_{n+1}$ also yield constructions for $K_{n+1}$ by switching the roles of $x_1$ and $x_2$ in $K_{n+1}$. In particular, all constructions are applicable in our setting from Section \ref{sec:n_wells_intro}.

\begin{lem}\label{lem:nD-korder}
Let $\{L_l\}_{l=1}^n$ be positive parameters with $0<L_l\le1$,
$$
R=\prod_{l=1}^n [0,L_l], \quad V=\prod_{l=1}^n L_l,
$$
and let $K_{n+1}$ be defined in \eqref{eq:energyKn}.
Given $h\le n-1$, let $0<\rho_h<\dots<\rho_1<\frac{L_{h+1}}{4}\le\rho_0:=1$ be an arbitrary choice of parameters.
Let $J_{j+1}=\tilde A_{j+1}-e_j\otimes e_j$ for $j=4,\dots,n$ according to the notation of Section \ref{sec:n_wells_intro}, and let $J_4=0$, $J_3=-\tilde A_3$.
Then there exist $u^{(h)}\in W^{1,\infty}(R;\R^n)$ such that $v(x)=J_{h+2}x$ for every $x\in\partial R$ and
\begin{equation}\label{eq:nD-hO}
\begin{split}
& \int_R\dist^2(\nabla u^{(h)},K_{h+1})dx + \epsilon\|D^2 u^{(h)}\|_{TV(R)} \\
& \qquad \lesssim V\Big(\sum_{i=2}^h\Big(\Big(\frac{\rho_i}{\rho_{i-1}}\Big)^2+\frac{\rho_i}{\rho_{i-2}}\Big)+\frac{1}{L_{h+1}^2}\rho_1^2+\frac{1}{L_{h+1}}\rho_2+\frac{\epsilon}{\rho_h}+\sum_{l\neq h,h+1}\frac{1}{L_l}\rho_1\Big),
\end{split}
\end{equation}
where $\tilde K_{h+1}:=\{\tilde A_1,\dots,\tilde A_{h+1}\}$.
\end{lem}

\begin{proof}
We work with an inductive procedure.
The induction base is provided by Lemma \ref{lem:nD-1O} with $F=J_3$, $A=\tilde A_1$ and $B=\tilde A_{2}$.

\emph{Inductive hypothesis.}
We assume that, we can find a $(h-1)$-th-order branching $v^{(h-1)}$ such that $v(x)=J_{h+1} x$ for every $x\in\partial R'$
\begin{equation}\label{eq:nD-(k-1)order}
\begin{split}
& \int_{R'}\dist^2(\nabla v^{(h-1)},\tilde K_h)dx+\epsilon\|D^2 v^{(h-1)}\|_{TV(R')} \\
& \qquad \lesssim V'\Big(\sum_{i=2}^{h-1}\Big(\Big(\frac{\rho_i'}{\rho_{i-1}'}\Big)^2+\frac{\rho_i'}{\rho_{i-2}'}\Big)+\frac{1}{(L_h')^2}(\rho_1')^2+\frac{1}{(L_h')}\rho_2'+\frac{\epsilon}{\rho_{h-1}'}+\sum_{l\neq h-1,h}\frac{1}{L_l}\rho_1'\Big),
\end{split}
\end{equation}
for every arbitrary choice $\{L_l'\}_{l=1}^n$ and $0<\rho_{h-1}'<\dots<\rho_1'<\frac{L_h'}{4}<\rho_0':=1$ with $R':=\prod\limits_{j=1}^n [0,L_j']$ and $V'=|R'|$.

\emph{Inductive step.}
Let $u^{(1)}$ denote a first order construction given by Lemma \ref{lem:nD-1O} with $A=\tilde A_{h+1}$, $B=J_{h+1}$, $F=J_{h+2}$, $\rho=\rho_1$ and let $\chi^{(1)}$ be the projection of $\nabla u^{(1)}$ onto $\{\tilde A_{h+1},J_{h+1}\}$.
The regions $\{\chi^{(1)}=J_{h+1}\}$ (apart from the cut-off regions) consist of the union of $R_{j,k}^{(1)}$ and $R_{j,k}^{(3)}$ as below
\begin{equation}
\begin{split}
R_{j,k}^{(1)} &= \big\{(x_h,x_{h+1})\in\omega_{j,k}^{(1)},\, x_l\in[\rho_1 L_l,(1-\rho_1)L_l],\, l\neq h,h+1\big\}, \\
R_{j,k}^{(3)} &= \big\{(x_h,x_{h+1})\in\omega_{j,k}^{(3)},\, x_l\in[\rho_1 L_l,(1-\rho_1)L_l],\, l\neq h,h+1\big\},
\end{split}
\end{equation}
where $\omega_{j,k}^{(1)}$ and $\omega_{j,k}^{(3)}$ are defined by \eqref{eq:cells-2order} with $\ell_j=\frac{\rho_1}{2^j}$, $y_j=L_{h+1}-L_{h+1}\frac{\theta^j}{2}$, $h_j:= y_{j+1}-y_j$ and $\lambda= \nu_{h+1}$.
We define $u^{(h)}$ by substituting to $u^{(1)}$ a $(h-1)$-th-order branching on the set $\{\chi^{(1)}=J_{h+1}\}$ by applying the inductive hypothesis on every $R_{j,k}^{(1)}$ 
corresponding to the choice of parameter $\rho'_i= \nu_{h+1}\frac{\rho_{i+1}}{2^{j+2}}$.
Thus, denoting with $u_{j,k}^{(1)}$ and $u_{j,k}^{(3)}$ the construction inside $R_{j,k}^{(1)}$ and $R_{j,k}^{(3)}$ respectively, we have
\begin{equation}\label{eq:nD-hO-1}
\begin{split}
& \int_{R_{j,k}^{(1)}}\dist^2(\nabla u^{(h)},\tilde K_{h+1})dx+\epsilon\|D^2 u_{j,k}^{(1)}\|_{TV(R_{j,k}^{(1)})} \\
& \quad \lesssim \frac{V}{L_h}\theta^j\frac{\rho_1}{2^j} \Big(\sum_{i=3}^h \Big(\Big(\frac{\rho_i}{\rho_{i-1}}\Big)^2+\frac{\rho_i}{\rho_{i-2}}\Big)+\Big(\frac{2^j}{\rho_1}\Big)^2\Big(\frac{\rho_2}{2^j}\Big)^2+\frac{2^j}{\rho_1}\frac{\rho_2}{2^j}+\frac{2^j\epsilon}{\rho_h}+\frac{1}{(2\theta)^j}\sum_{l\neq h-1,h}\frac{1}{L_l}\rho_2\Big).
\end{split}
\end{equation}
The construction on $R_{j,k}^{(3)}$ is equal up to a shear, i.e.
\begin{align*}
&u_{j,k}^{(3)}(\dots,x_{h-1},x_h,x_{h+1},\dots) \\
&\quad =u_{j,k}^{(1)}\Big(\dots,x_{h-1},x_h-\frac{(1-\nu_{h+1})\ell_j (x_{h+1}-y_j)}{2(y_{j+1}-y_j)}-\frac{\nu_{h+1}\ell_j}{2},x_{h+1},\dots\Big)+\frac{\nu_{h+1}\ell_j}{2}J_{h+1} e_h \\
& \quad \qquad +\frac{(1-\nu_{h+1})\ell_j(x_{h+1}-y_j)}{2(y_{j+1}-y_j)}\tilde A_{h+1}e_h.
\end{align*}
Thus, analogously as in \eqref{eq:2nd-ord-cell2} and \eqref{eq:3D-2O-app2} we infer that
\begin{equation}\label{eq:nD-hO-2}
\int_{R_{j,k}^{(3)}}\dist^2(\nabla u^{(h)},\tilde K_{h+1})dx \lesssim \int_{R_{j,k}^{(1)}}\dist^2(\nabla u^{(h)},\tilde K_{h+1})dx+\Big(\frac{\rho_1}{2^j}\Big)^3\frac{1}{L_{h+1}\theta^j}\frac{V}{L_h L_{h+1}}.
\end{equation}
Eventually, combining \eqref{eq:nD-hO-1} and \eqref{eq:nD-hO-2}, summing over $j$ and $k$ and adding the cut-off term of amplitude $\rho_1$, we obtain
\begin{equation}
\begin{split}
& \int_{R}\dist^2(\nabla u^{(h)},\tilde K_{h+1})dx+\epsilon\|D^2 u^{(h)}\|_{TV(R)} \\
& \quad \lesssim V\Big(\sum_{i=2}^h \Big(\Big(\frac{\rho_i}{\rho_{i-1}}\Big)^2+\frac{\rho_i}{\rho_{i-2}}\Big)+\frac{\epsilon}{\rho_h}+\sum_{l\neq h-1,h}\frac{1}{L_l}\rho_2+\frac{1}{L_{h+1}^2}\rho_1^2+\sum_{l\neq h,h+1}\frac{1}{L_l}\rho_1\Big).
\end{split}
\end{equation}
We recall that the summability in $j$ (for every $k$) is ensured by taking $\frac{1}{4}<\theta<\frac{1}{2}$.
Lastly, noting that
$$
\sum_{l\neq h-1,h}\frac{1}{L_l}\rho_2 \lesssim \frac{1}{L_{h+1}}\rho_2+\sum_{l\neq h,h+1}\frac{1}{L_l}\rho_1,
$$
the result is proven.
\end{proof}

\subsubsection{$n$-th-order construction}

Thanks to the general result proven in the previous subsection we have everything in place to prove the upper bounds in Theorem \ref{thm:Kn}.

\begin{proof}[Proof of the upper bounds from Theorem \ref{thm:Kn}]
Working as in Step 1 of the proof of Proposition \ref{prop:3D-3order}, let $\tilde u$ be defined by Lemma \ref{prop:vert-bra} with $L=H=1$, $N=\frac{4}{r}$ and $p=2$, with $0<r<1$ a small length scale which is to be determined below.
We define $u_i^{(1)}=F_{i,i}$ for every $i\in\{1,\dots,n-1\}$ and
$$
u_n^{(1)}(x_1,\dots,x_n)=\tilde u_1(x_n,\rho(x_1,\dots,x_{n-1})), \quad \rho(x_1,\dots,x_{n-1}):=\max_{1\le i\le n-1}\Big\{\Big|x_i-\frac{1}{2}\Big|\Big\}+\frac{1}{2}.
$$
Let $\chi^{(1)}$ be the projection of $\nabla u^{(1)}$ on $\{A_{n+1},J_{n+1}\}$.
Proceeding as in the proof of Proposition \ref{prop:3D-3order}, the region $\{\chi^{(1)}=J_{n+1}\}$ consists of the union of the sets $\Omega_{j,k}^{(1)}$ and $\Omega_{j,k}^{(3)}$ defined as in \eqref{eq:3D-cells-2order} with $\rho(x_1,\dots,x_{n-1})$ in place of $\rho(x_1,x_2)$.
In each of these subdomains we replace $u^{(1)}$ with the $(n-1)$-th order branching defined by Lemma \ref{lem:nD-korder}.
As in the proof of Proposition \ref{prop:3D-3order} the energy on $\Omega_{j,k}^{(1)}$ is controlled by that inside the intervals
$$
R_j=[0,1]^{n-2}\times[0,\theta^j]\times\Big[0,\frac{r}{2^j}\Big].
$$
Thus, applying Lemma \ref{lem:nD-korder} (with the roles of $x_1$ and $x_2$ switched) with $L_l=1$, $L_n=\lambda\frac{r}{2^{j}}$ $\rho_i=\lambda\frac{r_{i+1}}{2^{j+2}}$, the resulting function $u$ satisfies
\begin{equation}\label{eq:nD-nO-2}
\begin{split}
&\int_{\Omega_{j,k}^{(1)}}\dist^2(\nabla u,K_{n+1})dx+\epsilon\|D^2 u\|_{TV(\Omega_{j,k}^{(1)})} \\
& \qquad \lesssim \frac{r}{2^j}\theta^j\Big(\sum_{i=2}^n\Big(\Big(\frac{r_i}{r_{i-1}}\Big)^2+\Big(\frac{r_i}{r_{i-2}}\Big)\Big)+\frac{2^j\epsilon}{r_n}+\frac{r_2}{2^j}\Big),
\end{split}
\end{equation}
here $r_n<\dots<r_2<r<1$ are parameters which are to be determined below.
Again, thanks to a shear (see e.g. \eqref{eq:shear} and \eqref{eq:3D-3O-3}) we define $u$ also on $\Omega_{j,k}^{(3)}$ and there it holds
\begin{equation}\label{eq:nD-nO-3}
\int_{\Omega_{j,k}^{(3)}}\dist^2(\nabla u_{j,k}^{(3)},K_{n+1})dx \lesssim \int_{\Omega_{j,k}^{(1)}}\dist^2(\nabla u_{j,k}^{(1)},K_{n+1})dx+\frac{r^3}{(8\theta)^j}.
\end{equation}
We then put $u$ equal to $u^{(1)}$ outside $\Omega_{j,k}^{(1)}$ and $\Omega_{j,k}^{(3)}$.
By \eqref{eq:nD-nO-2}, \eqref{eq:nD-nO-3} and summing (thanks to the condition on $\theta$) over $j$ and $k$, we have
$$
E_{\epsilon,n+1}(u,\chi) \lesssim \sum_{i=2}^n\Big(\frac{r_i}{r_{i-1}}\Big)^2+r^2+\sum_{i=3}^n\frac{r_i}{r_{i-2}}+r_2+\frac{\epsilon}{r_n},
$$
where $\chi$ is the projection of $\nabla u$ on $K_{n+1}$ and we have used that $\|\nabla\chi\|_{TV}\le C\|D^2 u\|_{TV}$.
An optimization argument yields that $r_i\sim r^i$ and $r\sim\epsilon^\frac{1}{n+2}$ and the result is proven.
\end{proof}

\begin{rmk}
We highlight that -- due to cut-off contributions on the coarsest scale --  an analogue of Lemma \ref{lem:nD-korder} pushed up to the $n$-th order of lamination would give
$$
E_{\epsilon,n+1}(u^{(n)},\chi^{(n)})\lesssim r+\frac{\epsilon}{r^n}
$$
which would not be enough to deduce the energy scaling of Theorem \ref{thm:Kn}.
The construction of $u^{(1)}$, which branches towards every direction orthogonal to the lamination reduces the elastic energy term from $r$ (which is due to the simple cut-off) to $r^2$ that is the term typical of branching constructions (with quadratic growth condition).

We conjecture that, in the general case of $p$-growth condition, concatenating branched lamination as $u^{(1)}$ the resulting test function $u$ would satisfy
$$
\int_{[0,1]^n}\dist^p(\nabla u,K_{n+1})dx+\epsilon\|D^2 u\|_{TV([0,1]^n)}\lesssim r^p+\frac{\epsilon}{r^n}
$$
giving rise to an energy scaling of $\epsilon^\frac{p}{p+n}$, analogous to that of the two-dimensional one (cf. Corollary \ref{cor:2D-korder-gen}).
\end{rmk}

\appendix

\section{Two-Dimensional Arbitrarily High Order Branching Constructions}

\label{sec:second-order}

Last but not least, in this section we analyze the energy scaling of branching constructions of arbitrarily high order, in general, two-dimensional rectangular domains and for general $p$-growth conditions.

\begin{lem}\label{lem:2D-korder-gen}
Let $L,H>0$ and $R=[0,L]\times[0,H]$ and let $k\in\N$.
Let $\{A_j\}_{j=1}^{k+1}$, $\{J_j\}_{j=3}^{k+1} \subset \diag(2,\R)$ be such that
$$
\rank(A_1-A_2)=1, \quad \rank(A_j-A_{j'})=2 \text{ for every } j\neq j', j,j'\ge2
$$
and
$$
 J_3=\lambda_3 A_1+(1-\lambda_3)A_2, \quad J_j=\lambda_j A_j+(1-\lambda_j) J_{j-1} \text{ for some } 0<\lambda_j<1, j\ge4.
$$
Let
$$
F=\begin{cases}
\lambda A_1+(1-\lambda)A_2 & k=1, \\
\lambda A_{k+1}+(1-\lambda)J_{k+1} & k\ge 2,
\end{cases}
$$
for some $0<\lambda<1$ and let $E_\epsilon$ be defined as in \eqref{eq:tot_p_en} with $K=\{A_j\}_{j=1}^{k+1}$.
Then, there exists $u^{(k)}\in W^{1,\infty}(R;\R^2)$, $\chi^{(k)}\in BV(R;K)$ with $u^{(k)}(x)=Fx$ for every $x\in\p R$ such that
\begin{equation}\label{eq:2D-iteration}
E_\epsilon(u^{(k)},\chi^{(k)}) \lesssim LH \Big(\sum_{j=1}^{k-1} \Big(\frac{r_{j+1}}{r_j}\Big)^p+\frac{\epsilon}{r_k}\Big)+\frac{L}{H^{p-1}}r_1^p
\end{equation}
for every arbitrary choice of parameters $0<r_k<\dots<r_2<r_1<\frac{1}{4}\min\{L,H\}$.
\end{lem}
\begin{proof}
We prove the claimed result with an inductive procedure.

\emph{Induction base.}
The base of the induction $k=1$ is provided by Lemma \ref{prop:vert-bra}.
Indeed, consider $u^{(1)}\in W^{1,\infty}(R;\R^2)$ defined by Lemma \ref{prop:vert-bra} with $A=A_1$, $B=A_2$ and $N=\frac{L}{r_1}$.
We set $\chi^{(1)}$ to be the projection of $\nabla u^{(1)}$ onto $K$ and from \eqref{eq:prop-vbra} we have
\begin{equation}\label{eq:ind-1}
E_\epsilon(u^{(1)},\chi^{(1)})\lesssim \frac{L}{H^{p-1}}r_1^p+LH\frac{\epsilon}{r_1}.
\end{equation}

\emph{Inductive hypothesis.}
Now assume that, given $\omega=[0,\ell]\times[0,h]$ and $0<\rho_{k-1}<\dots<\rho_1<\frac{1}{4}\min\{\ell,h\}$ arbitrary parameters we can find $v^{(k-1)}$ with affine boundary condition $J_{k+1}$ and such that
\begin{equation}\label{eq:2D-iteration-IH}
\int_\omega\dist^p(\nabla v^{(k-1)},\{A_j\}_{j=1}^k)dx+\epsilon\|D^2v^{(k-1)}\|_{TV(\omega)} \lesssim \ell h\Big(\sum_{j=1}^{k-2}\Big(\frac{\rho_{j+1}}{\rho_j}\Big)^p+\frac{\epsilon}{\rho_{k-1}}\Big)+\frac{h}{\ell^{p-1}}\rho_1^p.
\end{equation}

\emph{Inductive step.}
Let $\omega_{j,k}$ be the subdomains defined in \eqref{eq:cells-bra}, that are rectangles of dimensions $\ell_j\times h_j$ with
$$
\ell_j=\frac{r}{2^j}, \quad h_j=H\frac{1-\theta}{2}\theta^j.
$$
Mimicking what has been done for the second order branching construction (Step 2 of the proof of the upper bound of Theorem \ref{thm:K3}) we define $u^{(k)}$ by substituting to $u^{(1)}$ the function $v^{(k-1)}$ on $\{\chi^{(1)}=J_3\}$, that is the union of $\omega_{j,k'}^{(1)}$ and $\omega_{j,k'}^{(3)}$ as defined in \eqref{eq:cells-2order}.

Therefore, we define $v_j^{(k-1)}$ by applying the inductive hypothesis on $\omega=\omega_{j,k'}^{(1)}$ taking $\rho_j=\lambda_j\frac{r_{j+1}}{2^{j+2}}$.
Then $v_j^{(k-1)}$ is defined on $\omega_{j,k'}^{(3)}$ thanks to a shear (see e.g. \eqref{eq:shear2D}) and equals $u^{(1)}$ elsewhere.
We thus get
\begin{align*}
&\int_{\omega_{j,k'}}\dist^p(\nabla v_j^{(k-1)},K)dx+\epsilon\|D^2 v_j^{(k-1)}\|_{TV(\omega_{j,k'})} \\
& \qquad \lesssim \frac{H\theta^jr_1}{2^j}\Big(\sum_{l=2}^k\Big(\frac{r_{l+1}}{r_l}\Big)^p+2^j\frac{\epsilon}{r_k}\Big)+\frac{H\theta^j(2^{p-1})^j}{r_1^{p-1}}\Big(\frac{r_2}{2^j}\Big)^p+\frac{r_1^{p+1}}{H^{p-1}(2^{p+1}\theta^{p-1})^j},
\end{align*}
where the last term comes from the analogue of equation \eqref{eq:2nd-ord-cell2}.
Again, since for every $j$ there are $\frac{L2^j}{r_1}$ copies of $\omega_{j,k'}$, the overall contribution is
$$
E_\epsilon(u^{(k)},\chi^{(k)}) \lesssim \sum_{j=1}^{j_0} \Big(LH\theta^j\sum_{l=2}^k\Big(\frac{r_{l+1}}{r_l}\Big)^p+LH(2\theta)^j\frac{\epsilon}{r_k}+LH\theta^j\Big(\frac{r_2}{r_1}\Big)^p+\frac{L}{H^{p-1}}\Big(\frac{1}{2^p\theta^{p-1}}\Big)^j r_1^p\Big)
$$
which yields \eqref{eq:2D-iteration} for $\theta$ satisfying \eqref{eq:theta}, i.e. $\frac{1}{2^\frac{p}{p-1}}<\theta<\frac{1}{2}$.
\end{proof}

Thanks to an optimization argument in the parameters $r_j$ and $\epsilon$, it is straightforward to infer the following result.

\begin{cor}\label{cor:2D-korder-gen}
Let $k\in\N$ and let $K$, $F$ and $E_\epsilon$ be as in the statement of Lemma \ref{lem:2D-korder-gen}.
Then, for every $p\in[1,\infty)$ and $\epsilon\in(0,1)$ there exist $u_\epsilon\in W^{1,\infty}((0,1)^2;\R^2)$ and $\chi_\epsilon\in BV((0,1)^2;K)$ such that $u_\epsilon(x)=Fx+b$ for every $x\in\R^2\setminus(0,1)^2$ and $b\in\R^2$ and
$$
E_\epsilon(u_\epsilon,\chi_\epsilon)\lesssim \epsilon^\frac{p}{k+p}.
$$
\end{cor}

\begin{rmk}
\label{rmk:Lorent}
The result of Lemma \ref{cor:2D-korder-gen}, when taking $k=2$ and $p=1$, matches with the energy scaling behaviour proved by \cite{Lorent06} of a second-order laminate construction for elastic energy of linear growth.

As remarked in the introduction (see Section \ref{sec:Lpintro}), for $p=1$ laminates and branching have the same scaling order for every order of lamination (see also \cite[Section 2]{RT21} for a computation of scalings of laminates of arbitrarily high order).
\end{rmk}

\bibliographystyle{alpha}
\bibliography{citations1}

\end{document}